\documentclass[11pt]{amsart}
\usepackage{mathrsfs}
\usepackage{mathtools}
\usepackage{amsfonts}
\usepackage{amssymb}
\usepackage{amsxtra}
\usepackage{dsfont}
\usepackage{color}
\usepackage{graphicx}
\usepackage[english,polish,french]{babel}
\usepackage{enumitem}
\usepackage{soul}

\usepackage[T1]{fontenc}
\usepackage{newtxmath}
\usepackage{newtxtext}

\usepackage[compress, sort]{cite}

\usepackage[margin=2.5cm, centering]{geometry}

\usepackage[colorlinks,citecolor=blue,urlcolor=blue,bookmarks=true]{hyperref}

\hypersetup{
pdfpagemode=UseNone,
pdfstartview=FitH,
pdfdisplaydoctitle=true,
pdfborder={0 0 0}, 
pdftitle={Martin compactifications of affine buildings},
pdfauthor={Bertrand R\'emy and Bartosz Trojan},
pdflang=en-US
}

\newcommand{\CC}{\mathbb{C}}
\newcommand{\ZZ}{\mathbb{Z}}
\newcommand{\NN}{\mathbb{N}}
\newcommand{\RR}{\mathbb{R}}

\newcommand{\Ss}{\mathbb{S}}

\newcommand{\calR}{\mathcal{R}}
\newcommand{\calC}{\mathcal{C}}

\newcommand{\calO}{\mathcal{O}}
\newcommand{\calM}{\mathcal{M}}
\newcommand{\calB}{\mathcal{B}}
\newcommand{\calV}{\mathcal{V}}
\newcommand{\calA}{\mathcal{A}}

\newcommand{\calH}{\mathcal{H}}
\newcommand{\calP}{\mathcal{P}}
\newcommand{\calQ}{\mathcal{Q}}

\newcommand{\bfc}{\mathbf{c}}
\newcommand{\bfb}{\mathbf{b}}

\newcommand{\scrA}{{\mathscr{A}}}

\newcommand{\scrC}{{\mathscr{C}}}

\newcommand{\scrS}{{\mathscr{S}}}
\newcommand{\scrX}{{\mathscr{X}}}

\newcommand{\frakA}{\mathfrak{a}}

\newcommand{\bfPhi}{\mathbf{\Phi}}

\newcommand{\Aut}{\operatorname{Aut}}

\newcommand{\res}{\operatorname{res}}

\newcommand{\pl}[1]{\foreignlanguage{polish}{#1}}
\newcommand{\fr}[1]{\foreignlanguage{french}{#1}}

\DeclarePairedDelimiter{\norm}{\lvert}{\rvert}
\DeclarePairedDelimiter{\abs}{\lvert}{\rvert}

\newcommand{\sprod}[2]{\langle {#1}, {#2} \rangle}

\newcommand{\cspan}{\operatorname{\mathbb{C}-span}}
\newcommand{\zspan}{\operatorname{\mathbb{Z}-span}}

\newcommand{\GL}{\text{GL}}
\newcommand{\Aff}{\text{Aff}}
\newcommand{\proj}{\operatorname{proj}}
\newcommand{\St}{\operatorname{St}}
\newcommand{\cat}{\operatorname{CAT}(0)}

\newcommand{\ind}[1]{{\mathds{1}_{{#1}}}}

\newcommand{\conv}{\operatorname{conv}}
\newcommand{\supp}{\operatorname{supp}}

\newcommand{\igerm}{{\operatorname{germ}}_\infty}

\newcommand{\clr}[2][3]{{}\mkern#1mu\overline{\mkern-#1mu#2}}

\newcommand{\vphi}{\varphi}

\newtheorem{theorem}{Theorem}[section]
\newtheorem{proposition}[theorem]{Proposition}
\newtheorem{lemma}[theorem]{Lemma}
\newtheorem{corollary}[theorem]{Corollary}
\newtheorem{claim}[theorem]{Claim}
\newtheorem*{theorem*}{Theorem}

\numberwithin{equation}{section}

\newcounter{thm}

\newtheorem{main_theorem}[thm]{Theorem}

\theoremstyle{definition}

\newtheorem{remark}{Remark}

\allowdisplaybreaks

\title[Martin compactifications]
{Martin compactifications of affine buildings}

\author{Bertrand R\'emy}
\address{
	\fr{
	Bertrand R\'emy\\
	Unit\'e de Math\'ematiques Pures et Appliquees (UMR 5669)\\
	ENS de Lyon
	}
}
\curraddr{
	\fr{
	\'Ecole normale sup\'erieure de Lyon,
	46 all\'ee d'Italie,
	69364 Lyon Cedex 07,
	France}
}
\email{bertrand.remy@ens-lyon.fr}

\author{Bartosz Trojan}
\address{
	\pl{
	Bartosz Trojan\\
	Instytut Matematyczny
	Polskiej Akademii Nauk\\
	ul. \'Sniadeckich 8\\
	00-696 Warszawa\\
	Poland}
}
\curraddr{
	\pl{
	Wydzia\l{} Matematyki\\
	Politechnika Wroc\l{}awska\\
	Wyb. Wyspia\'{n}skiego 27\\
	50-370 Wroc\l{}aw\\
	Poland}
}
\email{bartosz.trojan@pwr.edu.pl}

\thanks{The research was partial supported by the \fr{Centre de Math\'ematiques Laurent Schwartz}
and the National Science Centre, Poland, Grant 2016/23/B/ST1/01665} 

\begin{document}
\selectlanguage{english}

\begin{abstract}
	We carry out an in-depth study of Martin compactifications of affine buildings, from the viewpoint of
	potential theory and random walks. This work does not use any group action on buildings,
	although all the results are also stated within the framework of the Bruhat--Tits theory of semisimple
	groups over non-Archimedean local fields. This choice should allow the use of these building compactifications
	in intriguing geometric group theory situations, where only lattice actions are available. 
	The resulting compactified spaces use and, at the same time, make it possible to understand geometrically the descriptions of
	asymptotic behavior of kernels resulting from the non-Archimedean harmonic analysis on affine buildings. 
	Along the paper, we make explicit the most substantial differences with the case of symmetric spaces, namely absence of a 
	group action but existence of precise asymptotics of Green kernels and, of course, no possibility to stand by standard 
	techniques from PDEs. 
\end{abstract}

\maketitle

\section*{Introduction}
This paper deals with the Martin compactifications of affine buildings. In other words, it makes a connection between
two very different mathematical topics. On the one hand, affine buildings are relevant to algebra and geometry and, on the
other hand, Martin compactifications refer to analysis, more precisely potential and probability theory. Therefore,
our first task in this introduction, before mentioning the new results, is to introduce these two fields independently
but in a way making them compatible with one another. At this stage let us simply say that dealing with
compactifications associated with potential theory allows us to construct, from a concrete viewpoint, compactifications that
before this approach could only be obtained artificially. Conversely, these compactifications provide a geometric way
of understanding the various factors in the asymptotic formula for the Green function obtained previously by
Gelfand--Fourier analytic methods. 

In what follows, the geometry on which the various analytic concepts are defined (such as random walks, or heat and Martin
kernels) are affine buildings. In many situations, the latter spaces are non-Archimedean analogues of Riemannian symmetric
spaces; they were indeed designed for this purpose by F.~Bruhat and J.~Tits (see \cite{BruhatTits1972} and
\cite{BruhatTits1984} for the full theory, and \cite{Tits1977} for an overview). Affine buildings thus provide the
well-adapted geometry that enables one to understand semisimple algebraic groups over non-Archimedean local fields,
such as classical matrix groups over finite extensions of the field of $p$-adic numbers $\mathbb{Q}_p$ (e.g. the group
${\rm SL}_n(\mathbb{Q}_p)$ itself). However, some affine buildings of low rank do not come from algebraic groups, and 
thus have interesting features in geometric group theory. For this reason we avoid using group actions for the basic results 
in the paper, even though we are led by this analogy and even though we eventually provide the group-theoretic statements. 

This approach is completely parallel to the way a semisimple real Lie group $G$
without compact factor is understood, that is via its action on the associated symmetric space $X=G/K$ where $K$ is a
maximal compact subgroup (see \cite{Helgason1979} and \cite{Maubon}). As a result, geometric proofs of crucial
tools in Lie theory and in representation theory, such as the well-known Cartan and Iwasawa decompositions for
non-Archimedean semisimple Lie groups, are obtained. The main difference with Lie theory over the reals is that maximal compact
subgroups are now also open, which is consistent with the fact that affine buildings are discrete structures, namely
products of simplicial complexes. For instance, affine buildings in rank $1$, {\it i.e.} corresponding to hyperbolic
spaces, are semi-homogeneous trees.

Apart from the above well-known algebraic consequences, this led to the possibility of studying questions which are also
relevant to spherical harmonic analysis as inspired by works of Harish-Chandra. Indeed, I. Satake showed that the
non-Archimedean semisimple Lie groups can be considered in the general framework of Gelfand pairs since some suitable
Hecke algebras of bi-invariant functions were shown to be commutative for the convolution law \cite{Satake1963}. The
corresponding spherical functions were computed by I.G. Macdonald in \cite{macdo0}. This opened the way to beautiful
combinatorial problems \cite{MacdonaldSym}. It also provided a Gelfand--Fourier transform allowing both to attack more
advanced questions and to get deeper understanding of analytic objects such as heat kernels.

Among these more advanced analytic problems we naturally find the ones related to random walks and integral
representations of the corresponding harmonic functions. The pioneering work in the field is due to H. Furstenberg who
developed the crucial notion of boundaries from a probabilistic viewpoint in this Lie-theoretic context
\cite{Furstenberg1963}. The latter notion had a strong impact on many questions in group theory, in particular in rigidity
theory \cite{Margulis}, and it is likely to be still useful in many situations in geometric group theory where rich
structures from Lie theory are missing and need to be replaced by more flexible measure-theoretic ones \cite{BaderShalom}. 
The Martin compactification procedure is relevant to this context. It deals with positive harmonic functions on symmetric
spaces with respect to the Laplace--Beltrami operator, which is a bi-invariant second order differential operator on
the automorphism group. Later, it was extended to more general situations, at least at the bottom of the spectrum;
for instance to integral equations with respect to well-behaved probability measures in the terminology introduced by
Guivarc'h--Ji--Taylor \cite{gjt}. Our aim is to construct Martin compactifications for affine buildings. Note that in this
case, the absence of differential structure so to speak leads us to unify the approach via the use of averaging operators 
({\it i.e.} difference operators). Averaging operators naturally correspond to Markov chains. We also wish to cover situations
where no sufficiently transitive group action is available which is a way to include some intriguing lower-dimensional exotic
affine buildings (for existence see e.g. \cite{RonanExotic}), which is a first deviation from \cite{gjt}. 

To be more precise, let us recall that an affine building $\scrX$ is a simplicial complex covered by subcomplexes
all modeled on a given affine tiling, called apartments, see \cite[Chapter V]{Bourbaki2002}. The latter subcomplexes
are required to satisfy natural incidence axioms: any two simplices must be contained in an apartment, and given any two
apartments there must exist a simplicial isomorphism fixing their intersection (see Section \ref{sec:10} for 
definitions). These axioms are particularly well-adapted to the construction of a complete non-positively curved distance
on $\scrX$ which makes buildings a beautiful source of examples of $\cat$-spaces for geometric group theory
\cite{Bridson1999}. In this paper we try to stick as much as possible to the discrete point of view, which enables us to
use more easily many notions from probability theory. The approach to Martin compactification in the discrete setup was
described by J.L. Doob \cite{DoobDiscretePot}. We want to compactify the set of all the so-called special vertices 
\cite[Section 1.3.7]{BruhatTits1972} of a given affine building $\scrX$. However, in the case when the root system of
$\scrX$ is non-reduced, Fourier analytic tools lead us to split the set of special vertices into two subsets $V_g$ and
$V_g^\varepsilon$ where $V_g$ consists of the special vertices having the same type as the origin $o$ (we call them
good vertices, see Section \ref{sec:8} and \cite{park2}). In the reduced cases, all special vertices are
good, so to treat all the buildings in a uniform way, we prefer to use the terminology good vertices.
In any case, each maximal simplex (called an alcove) contains at least one good vertex, hence $V_g$ is sufficiently large
to provide a satisfactory compactification of~$\scrX$. 

We study an averaging operator $A$ acting on functions on good vertices that is related to the transition function
$p(x, y)$ of a finite range random walk defined on $V_g$ (see Section \ref{ss - Martin embeddings}) as follows 
\[
	A f(x) = \sum_{y \in V_g} p(x, y) f(y), \qquad x \in V_g.
\]
Finite range of the random walk guarantees that the embedding $\iota_\zeta$ defined in \eqref{iota} has discrete image. 
We also assume that the random walk is isotropic and irreducible, which are the natural conditions corresponding to
the well-behaved probability measures on symmetric spaces. Let us stress that in the building case, there is no choice of 
a specific averaging operator which would correspond to the Laplace--Beltrami operator on symmetric spaces; this explains 
why we work with this class of measures. Let $\varrho$ be the spectral radius of the operator $A$ acting on $\ell^2(V_g)$;
it can be computed in purely Lie-theoretic terms even without any group action, see formula \eqref{eq:60}.
Classically, for each $\zeta \geqslant \varrho$ we can define the $\zeta$-Green function
\[
	G_\zeta(x, y) = \sum_{n \geqslant 0} \zeta^{-n} p(n; x, y), \qquad \text{for } x, y \in V_g
\]
which leads to the Martin kernels
\[
	K_\zeta(x, y) = \frac{G_\zeta(x, y)}{G_\zeta(o, y)}.
\]
We are now in position to define the Martin embedding associated with the transition function $p$ and to the real parameter $\zeta$. 
Let us denote by $\calB_\zeta(V_g)$ the set of positive $\zeta$-superharmonic functions on $V_g$ ({\it i.e.} functions $f$ on
$V_g$ such that $Af \leqslant \zeta f$), normalized to take value $1$ at the origin $o$. The set $\calB_\zeta(V_g)$
endowed with the topology of pointwise convergence is a compact second countable Hausdorff space, thus it is metrizable. 
The corresponding Martin embedding is the map 
\begin{equation}
	\label{iota}
	\begin{alignedat}{1}
		\iota_\zeta: V_g &\longrightarrow \calB_\zeta(V_g) \\
		y &\longmapsto K_\zeta(\,\cdot\,, y) 
	\end{alignedat}
\end{equation} 
which can be shown to be injective with discrete image, and $\Aut(\scrX)$-equivariant for a suitably defined projective
action on $\calB_\zeta(V_g)$ (see formula \eqref{eq - proj action}). Such an embedding with these properties is a typical
map we will use in this paper in order to define compactifications in the sense of Section \ref{sec:11}.
The closure of the image of $\iota_\zeta$, which we denote by $\clr{\scrX}_{M, \zeta}$, is called the Martin
compactification of $\scrX$ (associated with $p$ and to the parameter $\zeta \geqslant \varrho$). The following theorem
collects the main results of our paper. 

\begin{main_theorem}
	\label{th - Martin} 
	Let $\scrX$ be a thick regular locally finite affine building. 
	\begin{enumerate}[label=(\roman*), ref=\roman*]
	\item 
	\label{en:8:1}
	For any isotropic irreducible finite range random walk on $\scrX$, the following dichotomy holds:
	\begin{itemize}
		\item[$\bullet$]~\emph{[At the bottom of the spectrum]}~If $\zeta=\varrho$, then $\clr{\scrX}_{M, \varrho}$
		is $\Aut(\scrX)$-equivariantly isomorphic to any of the Furstenberg (measure-theoretic) 
		or the Caprace--L\'ecureux (combinatorial) compactifications of the set $V_g$ of good vertices.
		\item[$\bullet$]~\emph{[Above the bottom of the spectrum]}~If $\zeta > \varrho$, then
		$\clr{\scrX}_{M, \zeta}$ is $\Aut(\scrX)$-equivariantly isomorphic to the join of any of the previous
		compactifications with the Gromov (horofunction) compactification. 
	\end{itemize}
	\item 
	\label{en:8:2}
	If the root system of the building is non-reduced, there exists an isotropic irreducible finite range random walk
	on $\scrX$ providing Martin compactifications of special vertices satisfying the same dichotomy.
	\end{enumerate}
\end{main_theorem}

Note that when the root system of the building is reduced, the notions of special and good vertices coincide. 
On Figure \ref{fig:1} we present an example of the closure of an apartment in a Martin compactification when the parameter
$\zeta$ is above the bottom of the spectrum $\varrho$. As the picture suggests, and as stated in the theorem, the Martin
compactification dominates both: the Gromov and the (maximal) Furstenberg compactifications. It is illustrated on
Figure \ref{fig:2} at the level of closures of apartments. 

The idea of the proof is (as in \cite{gjt}) to use a family of remarkable unbounded sequences -- called 
\emph{core sequences}~(\ref{sec:11}) -- such that: 
\begin{itemize}[label=\tiny$\bullet$]
	\item~any unbounded sequence of good vertices admits a core subsequence; 
	\item~any core sequence converges in most of our compactifications, the limit being precisely localized thanks to 
	the parameters describing the core sequence. 
\end{itemize}
In this context, identifying two compactifications then amounts to showing that the exact localization of the limit of a 
core sequence in the boundaries is done via the same process out of the parameters of the core sequence for both 
compactifications. This explains why the heart of the proof of the identification theorems as above is the combination of 
a convergence and of a uniqueness statement. Consequently, the proof of Theorem \ref{th - Martin}\eqref{en:8:1} in the case 
$\zeta=\varrho$ basically follows from Theorem \ref{thm:4} for convergence and Theorem \ref{thm:6} for uniqueness 
(see Theorem \ref{thm:7}), and the proof of Theorem \ref{th - Martin}\eqref{en:8:1} in the case $\zeta>\varrho$ follows 
from Theorem \ref{thm:10} for convergence and Theorem \ref{thm:2} for uniqueness (see Theorem \ref{thm:8}); at last, 
Theorem \ref{th - Martin}\eqref{en:8:2} follows from the Appendix  (see Theorem \ref{thm:9}). 

\begin{figure}[ht!]
	\includegraphics[width=17em]{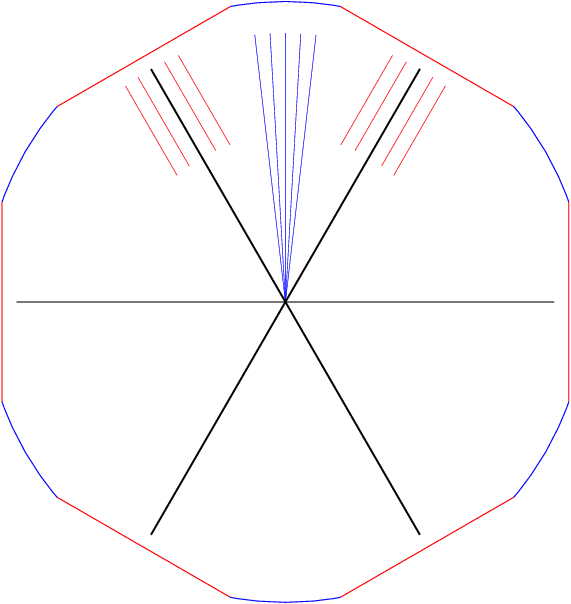}
	\caption{Closure of an apartment in the Martin compactification above the bottom of the spectrum
	($\widetilde{{\rm A}}_2$ case)}
	\label{fig:1}
\end{figure}

\begin{figure}[ht!]
	\begin{center}
  	\includegraphics[width=0.35\linewidth]{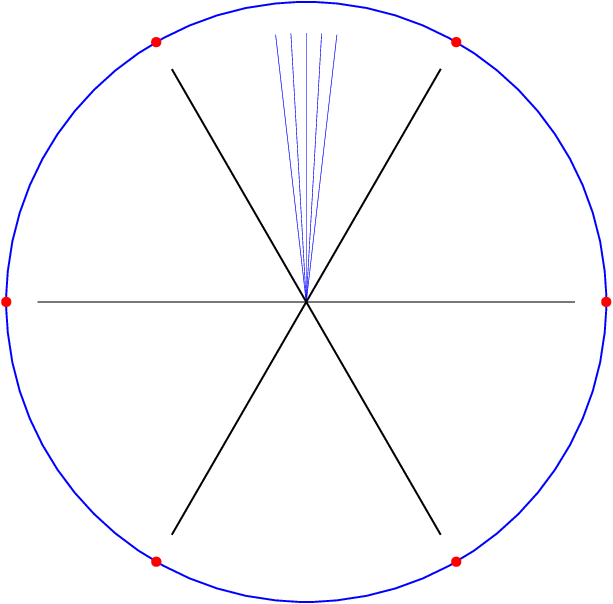}
  	\hspace{0.1\linewidth}
	\includegraphics[width=0.3\linewidth]{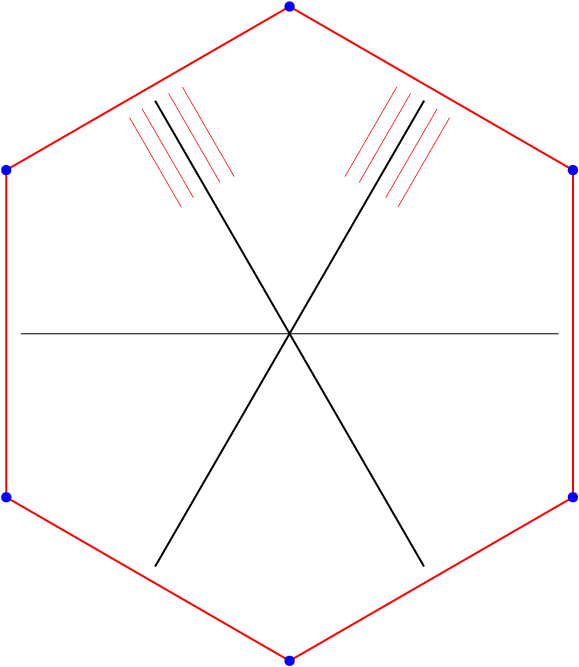}
 	\caption{The left picture is the closure of an apartment in the Gromov compactification
	whereas the right picture is the closure of an apartment in the (maximal) Furstenberg compactification
	($\widetilde{{\rm A}}_2$ case)}
	\label{fig:2}
\end{center}
\end{figure}

Let us discuss now some pre-existing compactifications. Many of them were initially defined when
the affine building comes from a semisimple algebraic group $\mathbf{G}$ over a locally compact non-Archimedean valued
field $k$, the first ones being due to E.~Landvogt thanks to a gluing procedure \cite{Landvogt}. In this context, by
Bruhat--Tits theory \cite{BruhatTits1972}, the group $\mathbf{G}(k)$ acts on an affine building $\scrX$ and the action
is strongly transitive in the sense that $\mathbf{G}(k)$ acts transitively on the inclusions of an alcove ({\it i.e.} a maximal
simplex) in an apartment. To our knowledge, the richest situation from the perspective of algebraic structures where
full Bruhat--Tits buildings are compactified (not only their vertices) and where integral models of $\mathbf{G}$
(as defined in \cite{BruhatTits1984}) are taken into account, is treated in \cite{RTW10}. This requires to use Berkovich
analytic geometry over non-Archimedean fields, and at the end this leads to connections with representation theory \cite{RTW12}
and algebraic geometry \cite{RTW17}. The outcome is a finite family of compactifications, indexed (as for 
symmetric spaces) by the conjugacy classes of parabolic subgroups. In the biggest one, corresponding to the choice of a
minimal parabolic subgroup, the closure of the vertices is equivariantly isomorphic to the group-theoretic
compactification, hence to the Martin compactification at the bottom of the spectrum ({\it i.e.} when $\zeta=\rho$). Therefore,
the case of a parameter above the bottom of the spectrum, $\zeta>\rho$, provides new compactifications of vertices which
take into account additional radial parameters for the convergence of suitable unbounded sequences (in addition to distances 
with respect to some faces of the Weyl sectors in the building). The following theorem collects Theorems \ref{thm:BerkoMartin} 
and \ref{thm:BrTMartinAbove} (more precisely, the identification of the Martin compactification at the bottom of the spectrum 
with the maximal Satake--Berkovich, polyhedral or group-theoretic compactification is provided by Theorem \ref{thm:BerkoMartin}, 
and the Martin compactification above the bottom of the spectrum is described in Theorem \ref{thm:BrTMartinAbove}). 

\begin{main_theorem}
	\label{th - alg gr} 
	Let $\mathbf{G}$ be a semisimple algebraic group defined over a locally compact non-Archimedean valued field $k$, and
	let $\scrX$ be its Bruhat--Tits building. We choose a good vertex in $\scrX$ and denote by $K$ its stabilizer.
	Let $p$ be a compactly supported bi-$K$-invariant well-behaved probability measure on $\mathbf{G}(k)$. Then the
	Martin compactification $\clr{\scrX}_{M, \varrho}$ can be equivariantly identified with the maximal Satake--Berkovich
	compactification of $\scrX$ (from analytic geometry); hence for $\zeta>\varrho$, the Martin compactification
	$\clr{\scrX}_{M, \zeta}$ is the join of $\clr{\scrX}_{M, \varrho}$ and of the Gromov compactification. 
\end{main_theorem}
In other words, in Theorems \ref{th - Martin} and \ref{th - alg gr} we provide a concrete potential-theoretic way to construct
compactifications that could only be obtained artificially by means of joining two previously known compactifications
({\it i.e.} by means of embedding diagonally the building in their product and taking the closure of the image). 

In this paper, we also make the choice of working in a situation which is as discrete as possible. 
This is consistent with standard references on random walks on graphs \cite{Woess} and with some recent works dealing with
spherical harmonic analysis on buildings, in particular with those due to A.M.~Mantero and A.~Zappa (e.g. \cite{mz1} and
\cite{mz2}) and to J.~Parkinson (e.g. \cite{park3} and \cite{park2}). This leads also to significant differences with \cite{gjt}.
Moreover, the statements of Theorem \ref{th - Martin} are valid even with small, possibly trivial, automorphism groups for 
$\scrX$: according to J.~Tits' classification \cite{Tits1986}, there exist affine buildings that are not relevant to any
algebraic group situation only when the dimension is $1$ (trees) or $2$. Nevertheless, the excluded cases are very interesting
because in dimension $2$ they lead to situations in which natural questions such as the linearity of some automorphism groups, 
(super)rigidity, arithmeticity or, on the contrary, simplicity properties of lattices, and also property (T) and strengthening 
of it, make sense and are strong motivations to develop more tools of geometric and analytic nature (see for instance 
\cite{BadCapLec} and \cite{LecSalWit} for striking recent results in this field).

Another purely geometric approach was developed by C.~Charignon \cite{Charignon2009} and by G.~Rousseau
\cite{Rousseau2023}. For our purposes this is an interesting viewpoint since, without any group action, it sticks as much
as possible to Bruhat--Tits' original approach governed by Lie combinatorics. By this we mean that the obtained
compactifications, isomorphic to the previous ones, are those which make most easily appear the modular structure of
the compact space: the boundary can be seen as a disjoint union of affine buildings at infinity of smaller rank, called
fa\c cades (\footnote{The terminology fa\c cade refers to a French word meaning the front face of a building;
here a fa\c cade is an affine building at infinity.}) in [loc. cit.]. In the Bruhat--Tits case, these fa\c cades are proven
to be the affine buildings attached to the parabolic subgroups of the initial non-Archimedean semisimple group \cite{RTW10}.
We rely on this already known structure in order to study the convergence of sequences of harmonic measures on affine buildings
in the spirit of Furstenberg compactifications. Each of these harmonic measures is defined on the maximal boundary $\Omega$
of the affine building $\scrX$, which is the set of parallelism classes of sectors endowed with a natural totally disconnected
topology ($\Omega$, as a set, consists of the chambers of the spherical building at infinity $\scrX^\infty$, see 
\cite[Section 11.8]{Abramenko2008}).

Each harmonic measure is attached to a well-defined special vertex in the building and, roughly speaking, is characterized
by the fact that it is the most symmetric probability measure on $\Omega$ with respect to the vertex 
(see \cite[Chapter 7]{park} and Section \ref{ss - harm meas}). Noting that each stratum at infinity, being an affine
building, can carry its own harmonic measures on its own maximal boundary, the following result, which we prove in Theorem
\ref{thm:3}, makes sense.
\begin{main_theorem}
	\label{th - harm conv} 
	Let $\scrX$ be a thick regular locally finite affine building. The closure of the collection of harmonic measures
	on $\scrX$ in the space of probability measures $\calP(\Omega)$ on the maximal boundary $\Omega$ endowed with 
	the weak-$*$ topology is $\Aut(\scrX)$-equivariantly isomorphic to the polyhedral or to the combinatorial compactification
	of $\scrX$. More precisely, the maximal boundary of each affine building at infinity, or stratum, can be seen as a residue
	in $\Omega$ and any cluster value of any unbounded sequence of harmonic measures in $\scrX$ is a harmonic measure on a
	well-defined stratum. 
\end{main_theorem}
Let us roughly sum up this part of the paper: affine buildings provide a suitable framework to generalize to higher dimensions
the classical study of harmonic measures on infinite graphs (since the set of initial harmonic measures has a strong
geometric structure related with Lie theory). The resulting measures at infinity are again harmonic measures for smaller
affine buildings in the boundary, and therefore provide an analogous geometric structure at infinity. 
We emphasize that in our purely geometric context, we cannot use the interpretation of maximal boundaries in terms of maximal 
flag varieties, and the related fibrations from the maximal flag variety to a flag variety associated with a non-minimal 
parabolic subgroup. Instead of this, we have to investigate geometrically sets of residues of a given type and interpret them as 
maximal boundaries of smaller affine buildings (at infinity). We also need to prove a disintegration formula for harmonic
measures in this context which can be seen as measure-theoretic substitutes for maps between flag varieties.

The reader who knows about compactifications of non-compact Riemannian symmetric spaces has already understood that our
results perfectly parallel the latter situation, at least at the level of the obtained statements. From this perspective, we
owe a lot to the book \cite{gjt} by Y.~Guivarc'h, L.~Ji and J.-C.~Taylor (see also \cite{gu} and \cite{MR1832435}) where, among
many other things, the precise descriptions of the Martin compactifications of symmetric spaces, both at the bottom and above
the bottom of the spectrum, are given. We used from there the idea of exhibiting well-chosen classes of unbounded sequences that
become convergent after applying a suitable embedding with compact metrizable target space (e.g. a space of probability measures
on a flag manifold, the Chabauty space of closed subgroups of the isometry group etc.); this is a good tool to compare the
various compactifications, including the most algebraic ones \cite{RTW10}. The book and Y.~Guivarc'h's quoted articles are the 
first places where most Bruhat--Tits analogues were conjectured. The latter references mainly study Riemannian symmetric spaces,
for which the most important ingredients are sufficiently precise Green kernel estimates \cite{aj} and some uniqueness results
for solutions of well-chosen PDEs taking into account invariance under suitable group actions. 

We decided to push the logic of using sequences to its fullest application. This means that we have chosen to study the
compactifications, possibly by adapting the class of sequences according to the finally expected boundary, through the 
parametrization of the points at infinity provided by the initial data characterizing the unbounded sequence used. Both for 
symmetric spaces and for Bruhat--Tits buildings, we know that the closure of a Weyl sector completely describes the 
compactification (it is a consequence of the Cartan decomposition). For a given compactification, the question is then to 
know which geometric parameters related to these simplicial cones are to be taken into account. In the case of the Gromov 
compactification, whose boundary is in all cases (Archimedean or not) a single spherical building, we know that the parameters
are radial and provide a direction of escape in the cone. In the case of all other compactifications considered, with the
exception of the Martin compactification above the bottom of the spectrum, the correct parameters are a partition of codimension
$1$ faces of Weyl sectors into two subsets: one for which the distances to the corresponding walls explode and the other for 
which the distances to the walls converge. In our approach, we obtain an identification between compactifications by showing
that the redundancies of parametrization of the limit points according to the sequence parameters are exactly the same on both
sides. The beauty of the Martin compactification above the bottom of the spectrum is that we have to use a similar partition as
before, but to use an additional radial (partial) parameter for the subset of walls with exploding distances 
(see Theorem \ref{thm:10}): this explains why it is obtained by joining the visual and any of the other previous
compactifications. 

Apart from the already mentioned fact that we avoid using group actions in order to make our results available to the
study of exotic situations, another significant difference with the book \cite{gjt} is the fact that one key ingredient
there was provided by estimates of the heat kernel and of Green functions due to J.-Ph. Anker and L.~Ji \cite{aj}, while
we use here asymptotics for the Green kernels, previously obtained by the second author in \cite{tr}. The latter formulas
are exact asymptotics of the requested kernels, so they can be directly used for our purposes. In particular, we cannot 
(and need not) use uniqueness arguments for solutions of partial differential equations, assumed in addition to be invariant
under some unipotent subgroup of the full isometry group (as in \cite[Theorem 7.22]{gjt}). This more direct approach can be
seen as a further manifestation of the fact that some formulas in spherical harmonic analysis can be simplified more efficiently
in the non-Archimedean case: Harish-Chandra's integral formula for spherical functions remains what it is on real numbers,
while it was algebraized by I.G.~Macdonald as early as in the 1970s \cite{macdo0}.

\subsection*{Choices and conventions}
Let us repeat quickly some choices: we are generically dealing with affine buildings without assuming the existence
of any sufficiently transitive group action. When dealing with affine buildings arising from semisimple groups over local
fields, we will explicitly call the corresponding spaces \emph{Bruhat--Tits buildings}: in other words, no Bruhat--Tits
building if no group of rational points $\mathbf{G}(k)$. Also, we use the notation ${\rm Stab}_G(x)$ to denote the stabilizer 
of a point $x$ in a group $G$ acting on a set $X$ containing $x$. 

In order to optimally use the requested analytic formulas (e.g. Green kernel asymptotics), we mainly see buildings as sets of 
(special or good) vertices; the only exception to this rule is Section \ref{sec:6} introducing affine buildings at infinity 
according to G.~Rousseau's approach (fa\c cades in his terminology). Moreover for a given affine building, the only apartment 
system we use on it is the complete one (in other words, and said in the metric language: any subset isometric to a Euclidean 
space and maximal for this property is an apartment). 

At last, we will be led to use subroot systems corresponding to subsets of simple roots: if $I$ is such a subset, then 
the index ${}_I$ attached to the standard notation for a given notion will mean that the object under consideration is defined 
with respect to the subroot system generated by $I$; this convention applies to root systems: $\Phi_I$ is the subroot system 
generated by $I$, but also to analytic notions: for instance, $c_I$ will be the Harish Chandra function associated with $\Phi_I$ 
etc. 

\subsection*{Structure of the paper}
Section \ref{sec:9} recalls as quickly as possible the useful notions from building theory; it introduces the
relevant classes of unbounded sequences and the definition of a compactification in our context. 
Section \ref{sec:6} is the place where we use
different, more metric, definitions of buildings, in order to recall the purely geometric construction of affine buildings
at infinity. Section \ref{sec:5} contains a discrete variation on the theme of visual compactifications; this adaptation
is useful to describe Martin compactifications above the bottom of the spectrum when seeing a building as a set of vertices. 
Section \ref{s - comb} deals with combinatorial compactifications of buildings which can be introduced in a remarkably 
elementary way in the affine case and which will mainly be used as a tool in the paper. Section \ref{s - max boundary} deals 
with harmonic measures on the maximal boundary of affine buildings; it contains results preparing the study of the Furstenberg 
compactification which may be useful in their own. The goal of Section \ref{s - Furst comp} is precisely to describe
compactifications of affine buildings obtained by suitably embedding them into the spaces of probability measures on maximal 
boundaries; the point is to show that cluster values of unbounded sequences of harmonic measures are still harmonic measures 
for affine buildings at infinity. Section \ref{sec:7} is the main part of the paper: it studies the Martin compactifications 
of affine buildings and proves the main Theorem \ref{th - Martin}; this requires to recall some notions from potential and 
probability theory. In Appendix \ref{ap:1}, we construct a distinguished random walk when the root system of the affine building 
is non-reduced; this random walk provides the desired Martin compactifications on the set of all special vertices (not only the 
good ones). Note that at the end of each relevant section, we illustrate our geometric results by providing their Bruhat--Tits 
consequences. 

\section{Affine buildings and compactifications}
\label{sec:9}

\subsection{Buildings}
\label{sec:10}
A family $\scrX$ of non-empty finite subsets of some set $V$ is an \emph{abstract simplicial complex}
if for all $\sigma \in \scrX$, each subset $\gamma \subseteq \sigma$ also belongs to $\scrX$. The elements of $\scrX$
are called \emph{simplices}. The dimension of a simplex $\sigma$ is $\#\sigma - 1$. Zero dimensional simplices are called
\emph{vertices}. The set $V(\scrX) = \bigcup_{\sigma \in \scrX} \sigma$ is the \emph{vertex set} of $\scrX$.
The dimension of the complex $\scrX$ is the maximal dimension of its simplices. A \emph{face} of a simplex $\sigma$
is a non-empty subset $\gamma \subseteq \sigma$. For a simplex $\sigma$  we denote by $\St(\sigma)$ the collection of
simplices containing $\sigma$; in particular, $\St(\sigma)$ is a simplicial complex. Two abstract simplicial complexes
$\scrX$ and $\scrX'$ are \emph{isomorphic} if there is a bijection $\psi: V(\scrX) \rightarrow V(\scrX')$ such that for
all $\sigma = \{x_1, \ldots, x_k\} \in \scrX$ we have $\psi(\sigma) = \{\psi(x_1), \ldots, \psi(x_k)\} \in \scrX'$.
With every abstract simplicial complex $\scrX$ one can associate its \emph{geometric realization} $|\scrX|$ in the vector 
space of functions $V \rightarrow \RR$ with finite support, see e.g. \cite[\S2]{Munkres1996}.

A set $\scrC$ equipped with a collection of equivalence relations $\{\sim_i : i \in I\}$ where $I = \{0, \ldots, r\}$,
is called a \emph{chamber system} and the elements of $\scrC$ are called \emph{chambers}. A \emph{gallery} of type
$f = i_1 \ldots i_k$ in $\scrC$ is a sequence of chambers $(c_1, \ldots, c_k)$ such that for all $j \in \{1,2, \ldots k\}$, 
we have $c_{j-1} \sim_{i_j} c_j$, and $c_{j-1} \neq c_j$. If $J \subseteq I$, a \emph{residue} of type $J$ is a subset of 
$\scrC$ such that any two chambers can be joined by a gallery of type $f = i_1 \ldots i_k$ with $i_1, \ldots, i_k \in J$. 
From a chamber system $\scrC$ we can construct an abstract simplicial complex where each residue of type $J$ corresponds to 
a simplex of dimension $r - \#J$. Then, for a given vertex $x$, we denote by $\calC(x)$ the set of chambers containing $x$. 

A \emph{Coxeter group} is a group $W$ given by a presentation
\[
	\left\langle
	r_i : (r_i r_j)^{m_{i, j}} = 1, \text{ for all } i, j \in I
	\right\rangle
\]
where $M = (m_{i,j})_{I \times I}$ is a symmetric matrix with entries in $\ZZ \cup \{\infty\}$ such that for all $i, j \in I$,
\[
    m_{i,j} =
    \begin{cases}
        \geqslant 2 &\text{if } i \neq j, \\
        1 & \text{if } i = j.
    \end{cases}
\]
For a word $f=i_1 \cdots i_k$ in the free monoid $I$ we denote by $r_f$ an element of $W$ of the form
$r_f= r_{i_1} \cdots r_{i_k}$. The length of $w \in W$, denoted $\ell(w)$, is the smallest integer $k$ such that there
is a word $f=i_1 \cdots i_k$ and $w=r_f$. We say that $f$ is reduced if $\ell(r_f) = k$. A Coxeter group $W$ may be turned
into a chamber system by introducing in $W$ the following collection of equivalence relations: $w \sim_i w'$ if and only if 
$w = w'$ or $w = w' r_i$. The corresponding simplicial complex $\Sigma$ is called 
\emph{Coxeter complex}. 

A simplicial complex $\scrX$ is called a \emph{building of type $\Sigma$} if it contains a family of subcomplexes called
\emph{apartments} such that
\begin{enumerate}[start=0, label=(B\arabic*), ref=B\arabic*]
	\item \label{en:3:1}
	each apartment is isomorphic to $\Sigma$,
	\item \label{en:3:2}
	any two simplices of $\scrX$ lie in a common apartment,
	\item \label{en:3:3}
	for any two apartments $\scrA$ and $\scrA'$ having a chamber in common there is an 
	isomorphism $\psi: \scrA \rightarrow \scrA'$ fixing $\scrA \cap \scrA'$ pointwise.
\end{enumerate}
The rank of the building is the cardinality of the set $I$. We always assume that $\scrX$ is irreducible. A simplex $c$ is a 
chamber in $\scrX$ if it is a chamber in any of its apartments. By $C(\scrX)$ we denote the set of chambers in $\scrX$. Using
the building axioms we see that $C(\scrX)$ has a chamber system structure. However, it is not unique. A geometric realization
of the building $\scrX$ is its geometric realization as an abstract simplicial complex. In this article we assume that
the system of apartments in $\scrX$ is \emph{complete}, meaning that any subcomplex of $\scrX$ isomorphic to $\Sigma$ is an
apartment. We denote by $\Aut(\scrX)$ the group of automorphisms of the building $\scrX$.

\subsection{Affine Coxeter complexes}
\label{sec:1}
In this section we recall basic facts about root systems and Coxeter groups. A general reference is \cite{Bourbaki2002},
which deals with Coxeter systems attached to reduced root systems. Since we use from the beginning possibly non-reduced
root systems, we will also refer to \cite{mz1, park}.

Let $\Phi$ be an irreducible, but not necessary reduced, finite root system in Euclidean space $\mathfrak{a}$ with associated
norm denoted by $|\cdot|$. Select $\{\alpha_i: i \in I_0\}$, where $I_0=\{1, \ldots, r\}$, a fixed base of $\Phi$, 
and let $\Phi^+$ be the corresponding set of all positive roots. Since $\Phi$ is irreducible, there is a unique highest root
$\alpha_0=\sum_{i \in I_0} m_i \alpha_i$, $m_i \in \NN_0$. We set
\[
	I_g=\{0\} \cup \{i \in I_0: m_i = 1\}.
\]
For each $\alpha \in \Phi$, we define a dual root
\[
	\alpha\spcheck = \frac{2}{\sprod{\alpha}{\alpha}} \alpha.
\]
Let $\Phi\spcheck = \{\alpha\spcheck : \alpha \in \Phi\}$ be the dual root system. Then the \emph{co-root lattice} $Q$
is the $\zspan$ of $\Phi\spcheck$. Let $Q^+ = \sum_{\alpha \in \Phi^+} \NN_0 \alpha\spcheck$. The dual basis to
$\{\alpha_i: i \in I_0\}$ are fundamental co-weights $\{\lambda_i: i \in I_0\}$. The co-weight lattice $P$ is the $\zspan$ of 
$\{\lambda_i: i \in I_0\}$. A co-weight $\lambda \in P$ is called dominant if $\lambda = \sum_{i \in I_0} x_i \lambda_i$ 
where $x_i \geqslant 0$ for all $i \in I_0$. Finally, the cone of all dominant co-weights is denoted by $P^+$. If $x_i > 0$ 
for all $i \in I_0$, then $\lambda$ is strongly dominant. We set
\[
	\tilde{\rho} = \frac{1}{2} \sum_{\alpha \in \Phi^+} \alpha.
\]
Let $\calH$ be the family of affine hyperplanes, called \emph{walls}, being of the form
\[
	H_{\alpha; k} = \big\{x \in \frakA : \langle x, \alpha \rangle = k \big\}
\]
where $\alpha \in \Phi^+$ and $k \in \ZZ$. Each wall determines two half-apartments
\[
	H^-_{\alpha; k} = \big\{x \in \frakA : \langle x, \alpha \rangle \leqslant k\big\}
	\quad\text{and}\quad
	H^+_{\alpha; k} = \big\{x \in \frakA : \langle x, \alpha \rangle \geqslant k\big\}.
\]
Note that for a given $\alpha$, the family $H^-_{\alpha; k}$ is increasing in $k$ while the family $H^+_{\alpha; k}$ is
decreasing. To each wall we associate $r_{\alpha; k}$ the orthogonal reflection in $\frakA$ defined by 
\[
	r_{\alpha; k}(x) = x - \big(\sprod{x}{\alpha} - k\big)\alpha\spcheck.
\]
Set $r_0 = r_{\alpha_0; 1}$, and $r_i = r_{\alpha_i; 0}$ for each $i \in I_0$. 

The \emph{finite Weyl group} $W$ is the subgroup of $\GL(\mathfrak{a})$ generated by $\{r_i: i \in I_0\}$. Let us denote
by $w_0$ the longest element in $W$. The \emph{fundamental sector} in $\frakA$ defined as
\[
	S_0 = \big\{x \in \frakA : \sprod{x}{\alpha_i} \geqslant 0 \text{ for all } i \in I_0\big\} 
	= \bigoplus_{i \in I_0} \RR_+ \lambda_i 
	= \bigcap_{i \in I_0} H^+_{\alpha_i; 0}
\]
is the fundamental domain for the action of $W$ on $\frakA$.

The \emph{affine Weyl group} $W^a$ is the subgroup of $\Aff(\mathfrak{a})$ generated by $\{r_i: i \in I\}$. Observe that 
$W^a$ is a Coxeter group. The hyperplanes $\calH$ give the geometric realization of its Coxeter complex $\Sigma_\Phi$. To see
this, let $C(\Sigma_\Phi)$ be the family of closures of the connected components of
$\frakA \setminus \bigcup_{H \in \calH} H$. By $C_0$ we denote the \emph{fundamental chamber}
(or \emph{fundamental alcove}), {\it i.e.}
\[
	C_0 = \big\{x \in \frakA : \sprod{x}{\alpha_0} \leqslant 1 \text{ and } \sprod{x}{\alpha_i} \geqslant 0 \text{ for all } 
	i \in I_0\big\} = \bigcap_{i \in I_0} H^+_{\alpha_i; 0}  \cap H^-_{\alpha_0; 1}
\]
which is the fundamental domain for the action of $W^a$ on $\frakA$. Moreover, the group $W^a$ acts simply transitively 
on $C(\Sigma_\Phi)$. This allows us to introduce a chamber system in $C(\Sigma_\Phi)$: For two chambers $C$ and $C'$ and
$i \in I$, we set $C \sim_i C'$ if and only if $C = C'$ or there is $w \in W^a$ such that $C = w . C_0$ and
$C' = w r_i . C_0$. 

The vertices of $C_0$ are $\{0, \lambda_1/m_1, \ldots, \lambda_r/m_r\}$. Let us denote the set of vertices of all 
$C \in C(\Sigma_\Phi)$ by $V(\Sigma_\Phi)$. Under the action of $W^a$, the set $V(\Sigma_\Phi)$ is made up of $r+1$
orbits $W^a.0$ and $W^a.(\lambda_i/m_i)$ for all $i \in I_0$. Thus setting $\tau_{\Sigma_\Phi}(0) = 0$, and
$\tau_{\Sigma_\Phi}(\lambda_i/m_i) = i$ for $i \in I_0$, we obtain the unique labeling
$\tau_{\Sigma_\Phi} : V(\Sigma_\Phi) \rightarrow I$ such that any chamber $C \in C(\Sigma_\Phi)$ has one vertex with each label. 

For each simplicial automorphism $\vphi: \Sigma_\Phi \rightarrow \Sigma_\Phi$ there is a permutation $\pi$ of the set $I$
such that for all chambers $C$ and $C'$, we have $C \sim_i C'$ if and only if $\vphi(C) \sim_{\pi(i)} \vphi(C')$, and 
\[
	\tau_{\Sigma_\Phi}(\vphi(v)) = \pi(\tau_{\Sigma_\Phi}(v)),
	\qquad\text{for all } v \in V(\Sigma_\Phi).
\]
A vertex $v$ is called \emph{special} if for each $\alpha \in \Phi^+$ there is $k$ such that $v$ belongs to $H_{\alpha; k}$.
The set of all special vertices is denoted by $V_s(\Sigma_\Phi)$. A co-dimension $1$ simplex whose vertices have their labels
in $I \setminus \{i\}$ is called an \emph{$i$-panel}.

Given $\lambda \in P$ and $w \in W^a$, the set $S = \lambda + w . S_0$ is called a \emph{sector} in $\Sigma_\Phi$ with a
\emph{base vertex} $\lambda$. Its sector $i$-panel is $\lambda + w . \big(S_0 \cap H_{\alpha_i; 0}\big)$.

Moreover, by \cite[Corollary 3.20]{Abramenko2008}, an affine Coxeter complex $\Sigma_\Phi$ uniquely determines the affine Weyl
group $W^a$ but not a finite root system $\Phi$. In fact, the root systems $\text{C}_r$ and $\text{BC}_r$ have the same
affine Weyl group.

\subsection{Affine buildings}
\label{sec:8}
A building $\scrX$ of type $\Sigma$ is called an \emph{affine building} if $\Sigma$ is a Coxeter complex corresponding
to an affine Weyl group. Select a chamber $c_0$ in $C(\scrX)$ and an apartment $\scrA_0$ containing $c_0$. Using
an isomorphism $\psi_0: \scrA_0 \rightarrow \Sigma$ such that $\psi_0(c_0) = C_0$, we define the labeling in $\scrA_0$ by
\[
	\tau_{\scrA_0}(v) = \tau_\Sigma(\psi_0(v)), \qquad v \in V(\scrA_0).
\]
Now, thanks to the building axioms the labeling can be uniquely extended to $\tau: V(\scrX) \rightarrow I$. We turn 
$C(\scrX)$ into a chamber system over $I$ by declaring that two chambers $c$ and $c'$ are $i$-adjacent if they share all
vertices except the one of type $i$ (equivalently, they intersect along an $i$-panel). For each $c \in C(\scrX)$ and $i \in I$,
we define
\[
	q_i(c) = \#\big\{c' \in C(\scrX) : c' \sim_i c \big\} - 1.
\]
In all the paper, we \emph{assume} that $q_i(c)$ only depends on $i$, {\it i.e.} that $q_i(c)$ is independent of $c$, and
therefore the building $\scrX$ is \emph{regular}; we henceforth write $q_i$ instead of $q_i(c)$. We also assume that 
$1 < q_i(c) < \infty$ and therefore the building $\scrX$ is \emph{thick} and \emph{locally finite}. 
Notice that for Bruhat--Tits buildings all of the assumptions about thicknesses are automatic. 
A vertex of $\scrX$ is special if it is special in any of its apartments. The set of special vertices is denoted by $V_s$. We
choose the finite root system $\Phi$ in such a way that $\Sigma$ is its Coxeter complex. In all cases except when the
affine group has type $\text{C}_r$ or $\text{BC}_r$, the choice is unique. In the remaining cases we select $\text{C}_r$
if $q_0 = q_r$, otherwise we take $\text{BC}_r$. This guarantees that $q_{\tau(\lambda)} = q_{\tau(\lambda+\lambda')}$
for all $\lambda, \lambda' \in P$, see the discussion in \cite[Section 2.13]{mz1}. 

Let us define the set of \emph{good} vertices $V_g$ consisting of those $x \in V(\scrX)$ having the
type $\tau(x) \in I_g$. If $\Phi$ is reduced, then all special vertices are good. However, we do not have this property
in the case of $\text{BC}_r$. Indeed, there are two types of special vertices, $0$ and $r$, but only those
of type $0$ are mapped to $P$. To deal with the vertices of type $r$ we modify the isomorphisms $\psi_0$ that we 
have used to define $\tau$. For this purpose we use the non-trivial automorphism $\vphi_\varepsilon$ of the Coxeter
complex which permutes the vertices of the base chamber $C_0$. Let $\varepsilon$ be the corresponding permutation of
the set $I$. We compose the isomorphism $\psi_0$ with the automorphism $\vphi_\varepsilon$ getting a new isomorphism
$\psi_\varepsilon$ which we use to get the new labeling $\tau_\varepsilon = \varepsilon \circ \tau$ on $V(\scrX)$, which 
we call $\varepsilon$-type. In particular, the vertices of type $r$ have $\varepsilon$-type $0$ now and vice-versa. We set
$V_g^\varepsilon = \{x \in V(\scrX) : \tau(x) = r\}$. Hence, $V_s = V_g \sqcup V_g^\varepsilon$. Now, all the statements
for good vertices hold true for $V_g^\varepsilon$ after application of the permutation $\varepsilon$. In general, all notations
with an index $\varepsilon$ will correspond to the standard notions after applying this twist of types. 

A half-apartment in $\scrX$ is a half-apartment in any of its apartments. For a subset $X$ of $V(\scrX)$,
the convex hull of $X$ is the set of all vertices of $\scrX$ that belong to every half-apartment containing $X$. 
The convex hull of two vertices $x$ and $y$ is denoted by $[x, y]$. A subcomplex $\scrS$ is called a sector of 
$\scrX$ if it is a sector in any apartment. We say that $\pi$ is the sector $i$-panel of $\scrS$ if there are an apartment 
$\scrA$ containing $\scrS$ and a type-preserving isomorphism $\psi: \scrA \rightarrow \Sigma$, such that $\psi(\pi)$ is 
the sector $i$-panel of the sector $\psi(\scrS)$. Two sectors are \emph{equivalent} if they contain a common subsector. 
The set of equivalence classes of sectors is denoted by $\Omega$ and is called the \emph{maximal boundary} of $\scrX$. 
For any special vertex $x \in V_s$ and $\omega \in \Omega$, there is a unique sector denoted by $[x, \omega]$ which has base 
vertex $x$ and represents $\omega$. For any two elements $\omega_1$ and $\omega_2$ from $\Omega$ there is an apartment of
$\scrX$ containing sectors $[x, \omega_1]$ and $[x, \omega_2]$ for a certain special vertex $x$. If the apartment is unique 
we say that $\omega_1$ and $\omega_2$ are \emph{opposite} while the apartment is denoted as $[\omega_1, \omega_2]$. 
For $x \in V_s$ and $y \in V(\scrX)$, we set
\[
	\Omega(x, y) = \big\{\omega \in \Omega : y \in [x, \omega] \big\}.
\]
Then for each $x \in V_s$, the collection $\big\{ \Omega(x, y) : y \in V_s \big\}$ generates a totally disconnected
compact Hausdorff topology in $\Omega$. In fact, the topology is independent of the choice of the reference vertex $x$,
see e.g. \cite[Proposition 3.15]{mz1}. Let us observe that each $\Omega(x, y)$ is clopen in $\Omega$, that is open and at 
the same time closed. The action of the automorphism group $\Aut(\scrX)$ can be extended to a continuous
action on $\Omega$. We fix once and for all the origin $o$ which is a good vertex of $c_0$.

The maximal boundary $\Omega$ can be equipped with adjacency relations to turn it into a chamber system. To be more
precise, we say that two chambers $\omega$ and $\omega'$ are $i$-adjacent if there are sectors $\scrS$ and $\scrS'$ 
representing $\omega$ and $\omega'$ respectively, and such that $\scrS \cap \scrS'$ contains an $i$-panel. In fact,
the obtained structure is a spherical building called the \emph{spherical building at infinity} $\scrX^\infty$. For 
details see e.g. \cite[Theorem 9.6]{ron}.

Given two special vertices $x, y \in V_s$, let $\scrA$ be an apartment containing $x$ and $y$ and let 
$\psi: \scrA \rightarrow \Sigma$ be a type-rotating isomorphism such that $\psi(x) = 0$ and $\psi(y) \in S_0$, 
see \cite[Definition 4.1.1]{park}. We set $\sigma(x, y) = \psi(y) \in \frac{1}{2} P^+$. If $x$ and $y$ are good vertices then
$\sigma(x, y) \in P^+$. For $\lambda \in P^+$ and $x \in V_g$, we denote by $V_\lambda(x)$ the set of all good vertices
$y \in V_g$ such that $\sigma(x, y) = \lambda$. The building axioms entail that the cardinality of $V_\lambda(x)$ depends only
on $\lambda$, see \cite[Proposition 1.5]{park2}. Let $N_\lambda$ be the common value. 

Let $\Phi^{++}$ be the set of roots $\alpha \in \Phi^+$ such that $\frac{1}{2} \alpha \notin \Phi^+$. If
$\alpha \in \Phi^{++}$ then $q_{\alpha} = q_{i}$ provided that $\alpha \in W.\alpha_i$ for $i \in I$. We define
\[
	\tau_\alpha = 
	\begin{cases}
		1 	& \text{ if } \alpha \notin \Phi, \\
		q_\alpha & \text{ if } \alpha \in \Phi, \text{ but } \tfrac{1}{2} \alpha, 2 \alpha \notin \Phi, \\
		q_{\alpha_0} & \text{ if } \alpha, \tfrac{1}{2} \alpha \in \Phi, \text{ and therefore } 2 \alpha \notin \Phi, \\
		q_\alpha q_{\alpha_0}^{-1} 
		& \text{ if } \alpha, 2 \alpha \in \Phi, \text{ and therefore } \tfrac{1}{2}\alpha \notin \Phi.
	\end{cases}
\]
We set
\begin{align*}
	\eta 
	&= \frac{1}{2} \sum_{\alpha \in \Phi^+} (\log \tau_\alpha) \alpha \\
	&= \frac{1}{2} \sum_{\alpha \in \Phi^{++}} (\log \tau_\alpha \tau_{2\alpha}^2) \alpha.
\end{align*}
Since $r_i$ sends $\alpha_i$ to $-\alpha_i$ and permutes other indivisible positive roots, we have
\[
	r_i . \eta = \eta - \big(\log \tau_{\alpha_i} \tau_{2\alpha_i}^2\big) \alpha_i.
\]
thus
\begin{equation}
	\label{eq:32}
	\eta = \frac{1}{2} \sum_{i = 1}^r (\log \tau_{\alpha_i} \tau_{2\alpha_i}^2) \sprod{\alpha_i}{\alpha_i} \lambda_i.
\end{equation}
Let us define a multiplicative function on $\mathfrak{a}$ by the following formula
\[
	\chi(\lambda) = \prod_{\alpha \in \Phi^+} \tau_\alpha^{\sprod{\lambda}{\alpha}},
	\qquad \lambda \in \mathfrak{a}.
\]
For $w \in W^a$ having the reduced expression $w = r_{i_1} \cdots r_{i_k}$, we set $q_{w} = q_{i_1} \cdots q_{i_k}$. Then
by \cite[Proposition 1.5 \& Appendix A]{park2},
\begin{equation}
	\label{eq:8}
	N_\lambda = \frac{W(q^{-1})}{W_\lambda(q^{-1})} \chi(\lambda) 
\end{equation}
where $W_\lambda = \big\{w \in W :  w .\lambda = \lambda \big\}$, and where for any subset $U \subseteq W$ we set 
\[
	U(q^{-1}) = \sum_{w \in U} q_w^{-1}.
\]
We also have
\begin{equation}
	\label{eq:5}
	N_{\varepsilon;\lambda} =  \frac{W(q_\varepsilon^{-1})}{W_\lambda(q_\varepsilon^{-1})} \chi(\lambda).
\end{equation}
Lastly, let us define the horocycle (or Busemann) function $h: V_s \times V_s \times \Omega \rightarrow \frac{1}{2} P$: 
For two special vertices $x$ and $y$, and $\omega \in \Omega$, we set
\begin{equation}
	\label{eq:15}
	h(x, y; \omega) = \sigma(x, z) - \sigma(y, z)
\end{equation}
where $z$ is any special vertex belonging to $[x, \omega] \cap [y, \omega]$. In fact, the value
$h(x, y; \omega)$ is independent of $z$, see e.g. \cite[Proposition 3.3]{mz1}. Figure \ref{fig:3} presents the geometric
interpretation of $h(x, y; \omega)$.

\begin{figure}[ht!]
   \includegraphics[width=20em]{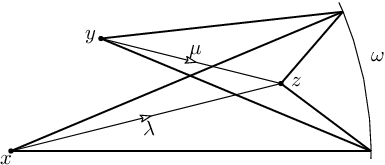}
   \caption{Geometric interpretation of the quantity $h(x,y;\omega)=\lambda-\mu$}
   \label{fig:3}
\end{figure}
In view of the Kostant's convexity theorem, see \cite{Kostant} or \cite[Lemma 3.19]{park2}, if $\sigma(x, y) = \lambda$
then for each $\omega \in \Omega$ and $w \in W$, we have $\lambda - w.h(x, y; \omega) \in Q^+$. In particular, 
\begin{equation}
	\label{eq:29}
	|h(x, y; \omega)| \leqslant |\sigma(x, y)|, \quad\text{and}\quad
	\sprod{h(x, y; \omega)}{\tilde{\rho}} \leqslant \sprod{\sigma(x, y)}{\tilde{\rho}}.
\end{equation}
Moreover, $h$ satisfies a cocycle relation, that is for every $x, y, z \in V_s$ and $\omega \in \Omega$,
\begin{equation}
	\label{eq:35}
	h(x, y; \omega) = h(x, z; \omega) + h(z, y; \omega).
\end{equation}
Fix $x \in V_s$ and $\omega \in \Omega$, and let $\scrA$ be an apartment containing the sector $[x, \omega]$. 
The \emph{retraction} on $\scrA$ with respect to $\omega$ and centered at $x$ is the mapping
$\rho^{x, \omega}_{\scrA}: \scrX \rightarrow \scrA$ defined as follows: If $\gamma$ is any simplex in $\scrX$, there is a
subsector $\scrS$ of $[x, \omega]$ such that $\scrS$ and $\gamma$ are in one apartment $\scrA'$. By the building axiom 
\eqref{en:3:3}, there is an isomorphism $\phi: \scrA' \rightarrow \scrA$ fixing $\scrS$ pointwise. Then
$\rho^{x, \omega}_{\scrA}(\gamma) = \phi(\gamma)$. The definition of $\rho^{x, \omega}_{\scrA}(\gamma)$ is independent of
the apartment $\scrA'$. If $x$ is a special vertex and $\psi: \scrA \rightarrow \Sigma$ is a type-rotating isomorphism
such that $\psi([x, \omega]) = S_0$, then
\[
	h(x, y; \omega) = \psi\big(\rho^{x, \omega}_{\scrA} (y)\big)
\]
for all $y \in V_s$. 

Given $\omega \in \Omega$, let $\Omega'(\omega)$ denote the set of all $\omega' \in \Omega$ which are opposite
to $\omega$. It is an open subset of $\Omega$. Indeed, let $\omega' \in \Omega(\omega)$. Set $\scrA = [\omega', \omega]$. 
We select any special vertex $x \in \scrA$ and take $y \neq x$, $y \in [x, \omega']$. Then $\rho^{y, \omega}_{\scrA}$
restricted to $[y, \omega] \cap [y, \omega'']$ for any $\omega'' \in \Omega(x, y)$, is an isomorphism onto 
$[y, \omega] \cap [y, \omega']$. Hence, there is a unique apartment containing $[x, \omega]$ and $[x, \omega']$, thus 
$\Omega(x, y) \subset \Omega'(\omega)$. The subset $\Omega'(\omega)$ is called the \emph{big cell}~associated with $\omega$.

\begin{lemma}
	\label{lem - residues are closed}
	Let $F$ be a facet  in the spherical building at infinity $\scrX^\infty$. 
	Then the corresponding residue $R(F)$ is closed as a subset of the maximal boundary $\Omega$. 
\end{lemma}
\begin{proof}
	Let $R$ be the residue attached to $F$; it is a proper subset of $\Omega$. Let us show that $\Omega \setminus R$ is open. 
	We pick $\omega \not\in R$. There exists an apartment $\scrA$ whose boundary $\scrA^\infty$ contains $F$ and
	$\omega$. Let $-\omega$ be the chamber in $\scrA^\infty$ which is opposite $\omega$. Since the big cell $\Omega'(-\omega)$ 
	is an open subset containing $\omega$, it is enough to show that $\Omega'(-\omega) \cap R = \varnothing$. Let $\rho$ 
	denote the retraction of the spherical building $\scrX^\infty$ onto its apartment $\scrA^\infty$ centered at $-\omega$. 
	The map $\rho$ preserves the Weyl-distance from $-\omega$ and preserves each adjacency relation in $\scrX^\infty$. The 
	first point implies that $\rho\bigl( \Omega'(-\omega) \bigr) = \{\omega \}$ and the second one implies that $\rho(R)$ is 
	the residue of $F$ in $\scrA^\infty$. As a consequence, since $\omega \not\in R$ we obtain 
	$\Omega'(-\omega) \cap R = \varnothing$, as requested. 
\end{proof}

\subsection{Compactifications}
\label{sec:11}
A \emph{compactification} of the building $\scrX$ is a pair $(\iota, H)$ where $H$ is a  compact second countable
Hausdorff space and $\iota$ is an embedding of all special vertices of $\scrX$ into $H$ such that $\iota(V_s)$ is a discrete
set. Then the closure of the image $\iota(V_s)$ equipped with the induced topology is a compact Hausdorff space. Since $H$
is metrizable, to describe $\clr(\iota(V_s))$ it is sufficient to consider sequences of special vertices $(x_n)$ such that
$(\iota(x_n))$ converges in $H$. If $G$ is a subgroup of the automorphism group of $\scrX$ that acts on $H$, the
compactification is called a \emph{$G$-compactification} if $\iota$ is $G$-equivariant, that is $\iota(g . x) 
= g . \iota(x)$ for all $x \in V_s$ and $g \in G$.

We say that $(\iota_1, H_1)$ \emph{dominates} $(\iota_2, H_2)$ if there is a continuous mapping
$\eta: H_1 \rightarrow H_2$ such that $\iota_2 = \eta \circ \iota_1$. If $\eta$ happens to be a homeomorphism,
$(\iota_1, H_1)$ is said to be isomorphic to $(\iota_2, H_2)$. Lastly, if both are $G$-compactifications then 
$(\iota_1, H_1)$ is $G$-isomorphic to $(\iota_2, H_2)$ provided $\eta$ is $G$-equivariant.

Given two compactifications $(\iota_1, H_1)$ and $(\iota_2, H_2)$ of the building $\scrX$, we can produce another
compactification $(\iota, H_1 \times H_2)$, by setting $\iota(x) = (\iota_1(x), \iota_2(x))$ for $x \in V_s$. It is
denoted as $(\iota_1, H_1) \vee (\iota_2, H_2)$. Notice that $(\iota_1, H_1) \vee (\iota_2, H_2)$ is the smallest
compactification that dominates both $(\iota_1, H_1)$ and $(\iota_2, H_2)$.

To describe the compactifications, we study sequences of good vertices of $\scrX$ approaching infinity. To be precise, 
a sequence $(x_n : n \in \NN)$ of vertices of $\scrX$ \emph{approaches infinity} if for any finite subset
$F \subset V_s$ there is $N \in \NN$, such that for all $n \geqslant N$, $x_n \notin F$. In fact, the way how the sequence
approaches infinity can be refined. A sequence $(x_n : n \in \NN)$ of good vertices of $\scrX$ is a \emph{core 
sequence}, if all have the same type and there are $\omega \in \Omega$ and sequence of good vertices 
$(u_n : n \in \NN_0)$, a subset $J \subsetneq I_0$, possibly empty, and numbers $(c_j : j \in J) \in \frac{1}{2} \NN_0^J$, 
such that
\begin{enumerate}
	\item
	\label{en:4:1} 
	$u_n \in [o, \omega]$;
	\item
	\label{en:4:2}
	for all $m \geqslant n$,
	\[
		\Omega(o, x_m) \subseteq \Omega(o, u_m) \subsetneq \Omega(o, u_n),
	\]
	but
	\[
		\Omega(o, x_m) \not\subseteq \Omega(o, u_{m+1});
	\]
	\item
	\label{en:4:3}
	for all $j \in J$ and $n \in \NN$,
	\[
		\sprod{\sigma(o, u_n)}{\alpha_j} = \sprod{\sigma(o, x_n)}{\alpha_j} = c_j;
	\]
	\item
	\label{en:4:4}
	for all $i \in I_0 \setminus J$,
	\[
		\lim_{n \to \infty} \sprod{\sigma(o, u_n)}{\alpha_i} = \infty.
	\]
\end{enumerate}
In this case $(x_n)$ is often called an $(\omega, J, c)$-{\emph core sequence} (with respect to the origin $o$) and $(u_n)$ 
is called an \emph{auxiliary sequence} for $(x_n)$. Let us observe that if $(x_n) \subset V_g$ then 
$(c_j : j \in J) \subset \NN_0^J$.

\begin{lemma}
	\label{lem:2}
	Any unbounded sequence of good vertices contains a core subsequence.
\end{lemma}
\begin{proof}
	Let $(x_n : n \in \NN)$ be a sequence leaving any compact subset. We set 
	\[
		J = \big\{ i \in I_0 : \limsup_{n \to \infty} \sprod{\sigma(o, x_n)}{\alpha_i} < +\infty \big\}.
	\]
	Since $x_n$ approaches infinity, $J \subsetneq I_0$ and, up to extracting, we can find 
	$(c_j : j \in J) \subset \frac{1}{2} \NN_0^J$ such that
	\begin{enumerate}
		\item for all $j \in J$ and all $n \in \NN$, we have: $\sprod{\sigma(o, x_n)}{\alpha_j} = c_j$,
		\item for all $i \in I_0\setminus J$, we have: 
		$\displaystyle \liminf_{n \to \infty} \sprod{\sigma(o, x_n)}{\alpha_i} = \infty$. 
	\end{enumerate}
	Let $K = \sum_{j \in J} c_j$. For each $m \geqslant 1$, we define the finite set $B_m$ consisting of all special vertices in 
	the vectorial spheres $V_\mu(o)$ for any $\mu \in P^+ \cap (Km \rho - Q^+)$. 
	There is $n_0 \geqslant 1$, such that $x_n \notin B_1$ for all $n \geqslant n_0$. For every $n \geqslant n_0$, there is 
	$v_n^{(1)} \in [o, x_n] \cap B_1$. Since $B_1$ is finite there is $u_1 \in \{v_n^{(1)} : n \geqslant n_0\}$ such that 
	$\{n \geqslant n_0 : v_n^{(1)} = u_1\}$ is infinite. Let $\varphi_1 : \NN \to \NN$ be an extraction such that 
	$u_1 = v^{(1)}_{\varphi_1(n)}$ for all $n \geqslant 1$. We repeat the above procedure for the sequence $(x_{\varphi_1(n)})$
	replacing $B_1$ by $B_2 \setminus B_1$. Now, by the diagonal extraction process, we obtain an extraction 
	$\varphi: \NN \to \NN$ and a sequence $(u_m : m \in \NN)$ such that $\sprod{\sigma(o, u_m)}{\alpha_j} = c_j$ for all 
	$j \in J$ and $m \geqslant 1$, and such that $u_m = v^{(m)}_{\vphi(n)} \in B_m \setminus B_{m-1}$ for all $n \geqslant 1$.

	At this stage, the last conditions \eqref{en:4:3} and \eqref{en:4:4} on core sequences (dealing with vectorial distances) 
	are fulfilled. It remains to extract both in the sequences $(x_{\varphi(n)})$ and $(u_m)$ to fulfill the first and last 
	conditions on shadows in \eqref{en:4:3} of the definition of core sequences. Finally, the sets $\Omega(o, u_m)$ are 
	clopen and decreasing, thus by compactness of $\Omega$, there is $\omega \in \Omega$ such that 
	\[
		\omega \in \bigcap_{m \geqslant 1} \Omega(o, u_m),
	\]
	which provides \eqref{en:4:1}, and the lemma follows.
\end{proof}
To describe the Gromov compactification (Section \ref{sec:5}) and Martin compactification above the bottom
of the spectrum (Section \ref{sec:7.3}), we need a further refinement of the sequences approaching infinity. Let
\[
	\Ss_+^{r-1} = \big\{u \in \Ss^{r-1} : \sprod{u}{\alpha} \geqslant 0, \text{ for all } \alpha \in \Phi^+\big\} 
\]
where $\Ss^{r-1}$ denotes the unite sphere in $\frakA$. A sequence $(x_n : n \in \NN)$ of good vertices of $\scrX$
is an \emph{angular} sequence if there are $\omega \in \Omega$, sequence of good vertices $(u_n : n \in \NN)$ and 
$\theta \in \Ss_+^{r-1}$ such that 
\begin{enumerate}
	\item $u_n \in [o, \omega]$;
	\item for all $m \geqslant n$, 
	\[
		\Omega(o, x_m) \subseteq \Omega(o, u_m) \subsetneq \Omega(o, u_n),
	\]
	but
	\[
		\Omega(o, x_m) \not\subseteq \Omega(o, u_{m+1});
	\]
	\item for all $i \in I_0$, if $\sprod{\theta}{\alpha_i} \neq 0$, then
	\[
		\liminf_{n \to \infty} \sprod{\sigma(o, u_n)}{\alpha_i} = + \infty;
	\]
	\item 
	\[
		\lim_{n \to \infty} \frac{\sigma(o, x_n)}{\norm{\sigma(o, x_n)}} = 
		\theta.
	\]
\end{enumerate}
We often say that $(x_n)$ is $(\omega, \theta)$-sequence. The sequences that are at the same time core and angular are
called angular core sequences. By compactness of $\Ss^{r-1}$ and Lemma \ref{lem:2} we have:
\begin{lemma}
	\label{lem:3}
	Any unbounded sequence of good vertices contains an angular core subsequence. \hfill$\square$
\end{lemma}

\subsection{The algebraic group setting}
\label{s:B1}
In this section we reformulate the previous notions in the situation when the building $\scrX$ is Bruhat--Tits, {\it i.e}
when it is associated with a semisimple algebraic group over a (locally compact) non-Archimedean local field $k$.
Let $k^\circ$ be its ring of integers; we denote by $\varpi$ a uniformizer of $k$, by $\kappa = k^\circ/\varpi k^\circ$
its residue field of $k$ and by ${\rm ord}_k : k^\times \to \mathbb{Z}$ the associated valuation. We let ${\bf G}$ be a 
connected semisimple algebraic group over $k$. Then it follows from Bruhat--Tits theory that the group ${\bf G}(k)$ acts
on an affine building $\scrX = \scrX({\bf G},k)$, see \cite[Section 2]{Tits1977}. The action is well-balanced in the sense
that it is strongly transitive ({\it i.e.} transitive on the inclusions of an alcove in an apartment) and proper
({\it i.e.} facet stabilizers are compact); the action is type-preserving as soon as we assume that ${\bf G}$ is simply
connected. In this context, apartments in the building are in bijective correspondence with maximal $k$-split tori. Let us
fix ${\bf S}$ such a maximal $k$-split torus and let us denote by ${\bf Z}={\bf Z}_{\bf G}({\bf S})$ (resp. by 
${\bf N}={\bf N}_{\bf G}({\bf S})$) its centralizer (resp. its normalizer). The spherical Weyl group, classically defined as
the quotient $W^{\rm sph}= {\bf N}(k) / {\bf Z}(k)$, is also the Weyl group associated with spherical root system
$\Phi^{\rm sph} = \Phi({\bf G},{\bf S})$, defined as generated by reflections, see \cite[Th\'eor\`eme 5.3]{BorelTits65}. 

Let us now recall how the apartment ${\mathcal A} = {\mathcal A}({\bf S},k)$ associated with ${\bf S}$ is constructed
\cite[Sect. 1]{Tits1977}. For any algebraic $k$-group ${\bf H}$, we denote by $X^*({\bf H})$ (resp. by $X_*({\bf H})$) the group
of characters ${\bf H} \to {\rm GL}_1$ (resp. the group of co-characters ${\rm GL}_1 \to {\bf H}$) defined over $k$.
The geometric realization ${\mathcal A}$ of the apartment $\scrA$ is a Euclidean affine space under the real vector space
$V = X_*({\bf Z}) \otimes_{\mathbb Z}{\mathbb R}$ admitting a suitable ${\bf N}(k)$-action
$\xi : {\bf N}(k) \to {\rm Aff}({\mathcal A})$. Roughly speaking, it is constructed as follows. We first define a map
$\xi : {\bf N}(k) \to {\rm Aff}({\mathcal A})$ by duality; namely, for any $z \in {\bf Z}(k)$ the image $\xi(z)$ is 
characterized by
\[
	\chi\bigl( \xi(z) \bigr) = - {\rm ord}_k\bigl( \chi(z) \bigr)
\]
for all $\chi \in X^*({\bf Z})$ where $\chi$ on the left-hand side is seen as a linear form on $V$. The kernel of the map $\xi$
is denoted by $Z_c$: it is the unique maximal compact subgroup of ${\bf N}(k)$ \cite[Proposition 1.2]{Landvogt} (it does not
come from an algebraic subgroup). The quotient group $\Lambda = {\bf Z}(k)/Z_c$ is a free Abelian group of rank equal to
${\rm dim}\,{\bf S} = {\rm dim}\,V$, see \cite[Lemma 1.3]{Landvogt}. Setting $\widetilde W = {\bf N}(k)/Z_c$, we obtain an exact
sequence 
\[
	0 \longrightarrow \Lambda \longrightarrow \widetilde W \longrightarrow W^{\rm sph} \longrightarrow 1. 
\]
The desired ${\bf N}(k)$-action will be via $\widetilde W$. The provisional map $\xi$ obtained so far corresponds to the action
of the translation part $\Lambda$ of $\widetilde W$. More precisely, by a standard pushforward argument
\cite[Proposition 1.6]{Landvogt}, we finally obtain 
\begin{itemize}[label=\tiny$\bullet$]
	\item~an affine space ${\mathcal A}$ with underlying Euclidean vector space $V$, 
	\item~an affine action $\xi : {\bf N}(k) \to {\rm Aff}({\mathcal A})$, 
	\item~a collection of affine linear forms $\Phi^{\rm aff}$ on ${\mathcal A}$, 
	\item~a map $\alpha \mapsto X_\alpha$ attaching to each $\alpha \in \Phi^{\rm aff}$ a subgroup $X_\alpha$ of ${\bf G}(k)$
\end{itemize}
such that 
\begin{enumerate}[label=(\roman*), ref=\roman*]
	\item 
	for any $n \in {\bf N}(k)$ we have $n X_\alpha n^{-1} = X_{\alpha \circ \xi(n)}$,  
	\item 
	the set of vectorial parts of the affine linear forms in $\Phi^{\rm aff}$ is equal to $\Phi^{\rm sph}$, 
	\item 
	for any $a \in \Phi^{\rm sph}$ the subgroups $X_\alpha$, for $\alpha$ of vectorial part equal to $a$, form a filtration of 
	${\bf U}_a(k)$. 
\end{enumerate}
Any affine linear form $\alpha \in \Phi^{\rm aff}$ is called an affine root of ${\bf G}$ over $k$. The zero set $\partial\alpha$
of an affine root $\alpha$ is called a wall and we denote by $r_\alpha$ the Euclidean reflection with respect $\partial\alpha$.
The reflections $r_\alpha$ generate a Euclidean reflection group $W^a$ called the affine Weyl group of $\mathcal A$; it is a 
finite index normal subgroup in $\tilde W$. The root system $\Phi$ we introduce in Section \ref{sec:1} (and its Weyl group $W$)
is related to $\Phi^{\rm sph} $ by the fact that they both provide, up to proportionality, the vectorial parts of the affine
roots in $\Phi^{\rm aff}$, in particular $W^{\rm sph}\simeq W$ (but these two finite root systems are not globally proportional
in general \cite[Section 1.7]{Tits1977}). 

It follows from Borel--Tits theory, see \cite[Theorem 21.15]{BorelBook}, that if we choose a minimal parabolic $k$-subgroup
${\bf P}$ containing ${\bf S}$, the couple $\bigl( {\bf P}(k), {\bf N}(k) \bigr)$ is a BN-pair (or Tits system in Bourbaki's
terminology \cite[IV.2]{Bourbaki2002}) for the group ${\bf G}(k)$; we will use this when going back to the maximal boundary
$\Omega$ and its big cells.  The spherical building at infinity $\scrX^\infty$, which is defined geometrically
\cite[Proposition 1]{Tits1986}, is the building that is naturally associated with this combinatorial structure. To each chamber at
infinity $\omega$ of $\scrX^\infty$ is attached a minimal parabolic $k$-subgroup ${\bf P}_\omega$ such that
${\bf P}_\omega(k) = {\rm Stab}_{{\bf G}(k)}(\omega)$. More generally, if $F$ is a facet in $\scrX^\infty$ there exists a
parabolic $k$-subgroup ${\bf P}_F$ such that ${\bf P}_F(k) = {\rm Stab}_{{\bf G}(k)}(F)$, and all parabolic $k$-subgroups of 
${\bf G}$ are obtained this way. 

\subsubsection{Cartan and Iwasawa decompositions, Busemann functions}
\label{sec:4}
Let $\mathcal A$ be the apartment associated with a maximal $k$-split torus ${\bf S}$ as before, and let $c$ be an alcove in it
whose closure contains a special vertex $x$. The cone with tip $x$ and generated by $c$ is an open Weyl sector in $\mathcal A$,
which we denote $\mathcal S$; its closure $\overline{\mathcal S}$ is a fundamental domain for the action of
${\rm Stab}_{W^a}(x) \simeq W^{\rm sph}$ on $\mathcal A$. We denote by $\omega$ the chamber at infinity represented by
$[x,\omega]$. In order to formulate suitably the Cartan and Iwasawa decompositions with respect to these geometric choices, 
we use the map $\xi : {\bf N}(k) \to {\rm Aff}({\mathcal A})$ giving the affine action of the group ${\bf N}(k)$ on $\mathcal A$
and whose image is isomorphic to $\widetilde W$. We denote by $Y$ the group of all translations contained in $\widetilde W$ and
by $Y^+$ the translations of $Y$ sending $x$ to a vertex in $\overline{\mathcal S}$. 

The Cartan decomposition of ${\bf G}(k)$ with respect to the choices of $x \in \mathcal A$ in ${\mathcal X}$ is the following
partition
\[
	{\bf G}(k) = \bigsqcup_{t \in Y^+} K_x \xi^{-1}(t) K_x
\]
where we use the short notation $K_x = {\rm Stab}_{{\bf G}(k)}(x)$ for the maximal compact subgroup ${\rm Stab}_{{\bf G}(k)}(x)$
in ${\bf G}(k)$. The geometric interpretation of the Cartan decomposition is the fact that a fundamental domain for the
$K_x$-action on the building ${\mathcal X}$ is given by the closed Weyl sector $\overline{\mathcal S}$. 

This can be seen by going back to the problem of describing (at least partially) root group actions on ${\mathcal X}$.
More precisely, we denote by $({\bf U}_a : a \in \Phi^{\rm sph})$ the collection of root groups in ${\bf G}$ with respect to
the maximal $k$-split torus ${\bf S}$ defining ${\mathcal A}$ (these metabelian groups are denoted by ${\bf U}_{(a)}$ in
\cite[Proposition 21.9]{BorelBook} and \cite[Section 5.2]{BorelTits65} but we stick to the notation in \cite{Tits1977}). 
By construction of the action $\xi$ (see Section \ref{s:B1}), the group ${\bf S}(k)$ acts by translations; the question
here is to describe the action of a group ${\bf U}_a(k)$ and for this we use its filtration by the subgroups $X_\alpha$ where
the $\alpha$'s have linear part $a$. If we fix such an affine root $\alpha$, then the compact subgroup $X_\alpha$ of
${\bf U}_a(k)$ fixes the positive half-space $A_\alpha = \alpha^{-1}(\{0 \})$ and folds the other half-apartment of $\mathcal A$
into another apartment in $\mathcal X$. This description is very useful for the interpretation of retractions in terms of
actions of well-chosen unipotent subgroups; retractions centered inside the building are related to affine Bruhat--Tits
decompositions and Cartan decomposition, while retractions centered at a chamber (or, more generally, at a spherical facet)
at infinity are related to Iwasawa decompositions (or, more generally, to horospherical decompositions). 

For the Cartan decomposition with respect to $x \in \mathcal A$ in ${\mathcal X}$, we choose the family of groups $X_\alpha$
geometrically characterized by the condition: $x \in A_\alpha$; in view of the previous description, all such $X_\alpha$'s
are included in $K_x$. In a first step we can even impose the slightly stronger condition that $c \subset A_\alpha$.
Then strong transitivity of the action is illustrated by the fact that any alcove or vertex $y$ in $\mathcal X$ can be sent into
$\mathcal A$ by making act finitely many elements of well chosen subgroups $X_\alpha$ with $c \subset A_\alpha$: this can be
proceeded by an induction whose last step is illustrated in Figure \ref{fig:10}. 

At this stage, we found $b$ fixing $c$ and sending an arbitrary vertex, $y$ say, into $\mathcal A$. We can now work in the
apartment $\mathcal A$ and make $W_x={\rm Stab}_{W^a}(x)$ act: since $x$ was chosen to be special, a fundamental domain for
the $W_x$-action $\mathcal A$ is given by $\overline{\mathcal S}$ and it remains to note that if $n \in {\bf N}(k)$ lifts an
element of $W_x$ which sends $b.y$ into $\overline{\mathcal S}$, then $n \in K_x$. 

\begin{figure}[ht!]
	\includegraphics[width=25em]{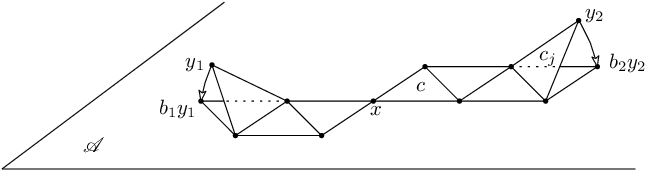}
	\vskip3ex
	\includegraphics[width=25em]{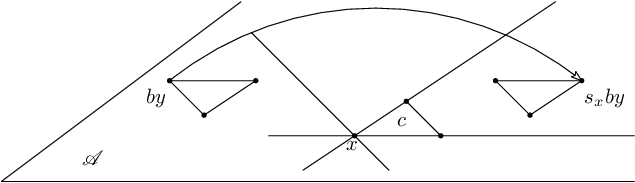}
	\caption{Cartan decomposition: first, fold onto the chosen apartment by root group action
   	(top picture) then use spherical Weyl group action to go to the chosen Weyl sector (bottom picture).}
	\label{fig:10}
\end{figure}

Figure \ref{fig:5} illustrates in one stroke the two steps when the $k$-rank of ${\bf G}$ is equal to $1$ 
({\it i.e.} when $\mathcal X$ is a tree): the effects of retracting with respect to an edge $c$ in a given geodesic, and then
of using (if needed) the symmetry with respect to the chosen vertex $x$. 

\begin{figure}[ht!]
   \includegraphics[width=27em]{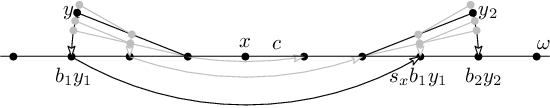}
   \caption{Cartan decomposition in the rank 1 case}
   \label{fig:5}
\end{figure}

The Iwasawa decomposition of ${\bf G}(k)$ with respect to the choices of $[x,\omega] \subset \mathcal A$ in $\scrX$ is
the partition
\[
	{\bf G}(k) = \bigsqcup_{t \in Y} K_x \xi^{-1}(t) {\bf U}^\omega(k)
\]
where ${\bf U}^\omega$ denotes the unipotent radical of the minimal parabolic $k$-subgroup ${\bf P}_\omega$. 

For this decomposition we also have a geometric interpretation but now we have to use a retraction onto the apartment 
$\mathcal A$ and based at $\omega$. The difference is that the collection of groups used to perform the foldings now consists
of all the full root groups ${\bf U}_a(k)$ where $a$ runs over the set of positive roots defined by the choice of the chamber
$\omega$ in the spherical building at infinity $\scrX^\infty$. Concretely, for a positive root $a \in \Phi$ there
is no restriction on the affine root $\alpha$ (with vectorial part $a$) used since the fixed half-apartment can be, so to speak,
arbitrarily close to $\omega$ to fold by induction galleries from an arbitrary alcove into $\mathcal A$. Figure \ref{fig:6}
illustrates an inductive step of such a folding applied to a gallery given by the projection of an arbitrary alcove into 
$\mathcal A$. 

\begin{figure}[ht!]
   \includegraphics[width=27em]{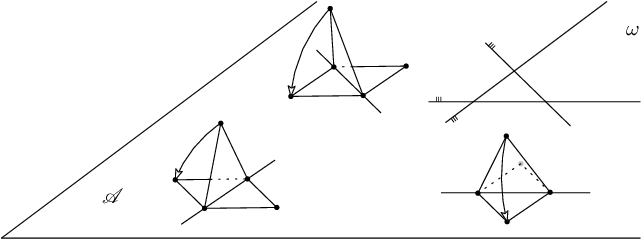}
   \caption{Iwasawa decomposition in general}
   \label{fig:6}
\end{figure}

Figure \ref{fig:7} illustrates the geometric interpretation of the Iwasawa decomposition in the tree case. 

\begin{figure}[ht!]
   \includegraphics[width=27em]{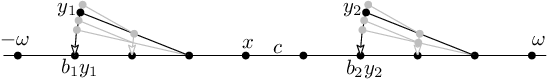}
   \caption{Iwasawa decomposition in the rank 1 case}
   \label{fig:7}
\end{figure}

In order to be complete, we shall explain the connection between these decompositions and the various distances and similar
quantities we use to study the convergence of unbounded sequences of vertices. We keep the geometric choices made before. 
If $y$ is a special vertex, then the Cartan decomposition with respect to $x \in \mathcal A$ says that there exists $k \in K_x$
such that $k.y \in \overline{\mathcal S}$ and this vertex is represented by a positive coweight in the identification between
${\mathcal S} \subset {\mathcal A}$ and $S_0 \subset \mathfrak{a}$: this coweight is $\sigma(x,y)$; it is the non-Archimedean
version of the radial component in bi-invariant harmonic analysis. If $y$ lies in the same ${\bf G}(k)$-orbit as $x$, then there
exists $n \in {\bf Z}(k)$ such that $\xi(n)$ is a translation of $Y^+$ and $k.y = \xi(n).x$; if the ${\bf G}(k)$-action on 
$\mathcal X$ is type preserving, then a vertex $y$ is in the same ${\bf G}(k)$-orbit as $x$ if and only is the vertices have the 
same type. Keeping the special vertex $y$, the Iwasawa decomposition of ${\bf G}(k)$ with respect to
$[x,\omega] \subset \mathcal A$ implies that there exists $u \in {\bf U}^\omega(k)$ such that $u.y \in {\mathcal A}$. The vertex
$u.y$ is represented by an arbitrary coweight in the identification between $[x,\omega] \subset \mathcal A$ and
$S_0 \subset \mathfrak{a}$: this coweight is $h(x,y;\omega)$. If $y$ lies in the same ${\bf G}(k)$-orbit as $x$, then there
exists $n \in {\bf Z}(k)$ such that $\xi(n)$ is a translation of $Y$ and $u.y = \xi(n).x$. 

\subsubsection{The maximal boundary from the algebraic viewpoint}
\label{sec:2}
In this paper, we make intensive use of the maximal boundary $\Omega$ to do analysis on affine buildings, and in particular
to define and study harmonic measures and Furstenberg compactifications of affine buildings. 

First, recall that $\Omega$ was defined in a purely geometric way: it is the set of parallelism classes of sectors, and 
therefore it can be seen also as the set of chambers in the spherical building at infinity $\scrX^\infty$, {\it i.e.} 
$\Omega=C(\scrX^\infty)$. From this viewpoint, the topology on $\Omega$ is defined as the one generated by the shadows 
$\Omega(x,y)$ emanating from a given special vertex $x$ (with $y$ varying). Nevertheless, since the obtained topology does
not depend on the choice of $x$, see \cite[Proposition 3.15]{mz1}, we can see it as generated by all shadows $\Omega(x,y)$ 
where the special vertices $x$ and $y$ both vary in $V_s$.

Going back to group actions, for $g \in {\bf G}(k)$ we have
\[
	y \in [x,\omega] \Longleftrightarrow g.y \in g.[x,\omega] \Longleftrightarrow g.y \in [g.x, g.\omega], 
\]
implying that ${\bf G}(k)$ permutes the shadows and therefore acts continuously on $\Omega$. Since ${\bf G}(k)$ acts strongly
transitively on $\scrX$, so does it on $\scrX^\infty$; in particular, the ${\bf G}(k)$-action on $\Omega$ is transitive. 
As a consequence, if we pick $\omega \in \Omega$ the orbit map of $\omega$ provides a homeomorphism 
\[
	{\bf G}(k)/{\bf P}_\omega(k) \simeq \Omega
\]
where ${\bf P}_\omega$ is the minimal parabolic $k$-subgroup associated with $\omega$ (note that compactness of
${\bf G}(k)/{\bf P}_\omega(k)$ follows from instance from a suitable Iwasawa decomposition -- for more details, please go to 
the group-theoretic, alternative, definition of harmonic measures below in Section \ref{ss - Furst group}). 

As it is well-known from Borel--Tits theory \cite[\S 21]{BorelBook}, the group ${\bf G}(k)$ admits a Bruhat decomposition
\[
	{\bf G}(k) = \bigsqcup_{w \in W^{\rm sph}} {\bf P}(k) w {\bf P}(k)
\]
where ${\bf P}$ is any minimal parabolic $k$-subgroup of ${\bf G}$. In fact, for each $w \in W^{\rm sph}$ the double class 
${\bf P}(k) w {\bf P}(k)$ can be written in a better way, avoiding in particular redundancies. If we pick a maximal $k$-split
torus ${\bf S}$ in ${\bf P}$, we have (positive) root subgroups with respect to ${\bf S}$ included in the unipotent radical
${\bf U}^+ = {\rm rad}_{\rm u}({\bf P})$, and similarly we have (negative) root groups included in the unipotent radical
${\bf U}^-$ of the minimal parabolic subgroup that is opposite ${\bf P}$ with respect to ${\bf S}$. Using this, we have 
${\bf P}(k) w {\bf P}(k) = {\bf U}^+_w(k) w {\bf P}(k)$ with ${\bf U}^+_w = {\bf U}^+ \cap w {\bf U}^- w^{-1}$. Multiplying 
by the longest element $\bar w$ in $W^{\rm sph}$, we obtain the refined Birkhoff decomposition 
\[
	{\bf G}(k) = \bigsqcup_{w \in W^{\rm sph}} {\bf U}^{-,w}(k) w {\bf P}(k)
\]
with ${\bf U}^{-,w} = {\bf U}^- \cap w {\bf U}^- w^{-1}$ for each $w \in W^{\rm sph}$. Using root groups, it is also known
that group multiplication provides an isomorphism of $k$-varieties, see \cite[Proposition 21.9 and Theorem 21.20]{BorelBook},
\[
	\prod_{b \in \Phi^{{\rm sph},-} \cap w\Phi^{{\rm sph},-}} {\bf U}_b \simeq {\bf U}^{-,w}.
\] 
Note that the biggest subgroup ${\bf U}^{-,w}$ is ${\bf U}^-$ and corresponds to $w=1$.

Going back to the maximal boundary, and denoting by $-\omega$ (or ${\rm opp}_{\bf S}(\omega)$ when necessary) the chamber which 
is opposite $\omega$ with respect to the apartment ${\mathcal A}({\bf S})$ corresponding to ${\bf S}$, the isomorphism 
${\bf G}(k)/{\bf P}_{\omega}(k) \simeq \Omega$ and the Birkhoff decomposition of ${\bf G}(k)$ provide a decomposition of
the maximal boundary
\[
	\Omega = \{ -\omega \} \quad \sqcup \quad \bigg( \bigsqcup_{\stackrel{w \in W^{\rm sph}}{w \neq 1, \bar w}}
	{\bf U}^{-,w}(k) w.\omega \bigg) \quad \sqcup \quad {\bf U}^-(k).\omega
\]
which is valid for each $\omega \in \Omega$ and each $k$-split torus ${\bf S}$ such that ${\mathcal A}({\bf S})^\infty$
contains $\omega$. The last subset of the partition is nothing else than the big cell $\Omega'(-\omega)$ of chambers which 
are opposite $-\omega$; more generally, the above partition is indexed by the Weyl-distance of chambers from $-\omega$ in
the spherical building $\scrX^\infty$. The big cell $\Omega'(-\omega)$ is an open neighborhood of $\omega$.
Note that the fact that any residue can be seen as a compact subset of $\Omega$ (this fact in the group-free case will be proved
in Lemma \ref{lem - residues are closed}) admits a more natural proof in the algebraic group context since a residue can be seen 
as an orbit under a well-chosen parabolic subgroup acting on $\Omega$. In fact, this orbit is also the orbit of a Levi factor 
of the parabolic subgroup, and even the orbit of a maximal compact subgroup in the latter Levi factor: this is a useful remark
when dealing with Furstenberg compactifications -- see Section \ref{ss - Furst group} below. 

Some smaller neighborhoods than big cells can be produced thanks to the filtrations on coordinates of the root groups 
${\bf U}_b(k)$. More precisely, we can start from the previous big cell
\[
	\Omega'(-\omega) = {\bf U}^-(k).\omega 
	= \bigg( \prod_{b \in \Phi^{{\rm sph},-}} {\bf U}_b(k) \bigg).\omega \quad 
	\simeq \prod_{b \in \Phi^{{\rm sph},-}} {\bf U}_b(k)
\]
to see that the root groups provide a system of coordinates (note that any root group is isomorphic, as a variety, to an affine
space \cite[ Theorem 21.20 (i)]{BorelBook}). Imposing some valuation conditions on the additive parameters of each root group
still provides some (now compact) open subsets: in other words, we are replacing here each factor $k$ of the coordinate system
by a compact open factor $\varpi^m k^\circ$ for some integer $m \in \mathbb Z$. The choices of the parameters can be made
consistent thanks to geometric considerations. For instance, we can pick a special vertex $x \in {\mathcal A}({\bf S})$.
This gives a sector $[x,\omega]$ and any special vertex $y \in [x,\omega]$ leads to a shadow $\Omega(x,y)$ which, in the above
parametrization of $\Omega'(-\omega)$, corresponds to choosing for each $b \in \Phi^{{\rm sph},-}$ the largest compact open
subgroup of the Bruhat--Tits filtration of ${\bf U}_b(k)$ fixing $y$, see \cite[Section 6.2]{BruhatTits1972}.

\begin{figure}[ht!]
   \begin{center}
  \includegraphics[width=0.3\linewidth]{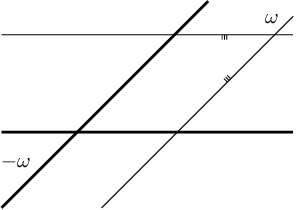}
  \caption{The picture shows shadows as constructed with root group parametrizations.}
  \end{center}
\end{figure}

\section{Affine buildings at infinity}
\label{sec:6} 
In this section, we adopt the geometric viewpoint on affine buildings (e.g., we recall the existence of complete, non-positively
curved metrics) and we introduce various related structures, mainly at infinity. For instance we go back to the classical
definition of the spherical building at infinity of a given affine building (Section \ref{sec:12}) and, less
classically, we introduce some auxiliary affine buildings of smaller rank, called \emph{fa\c cades} (after G.~Rousseau 
\cite{Rousseau2023}). There are fa\c cades of two kinds. The outer fa\c cades are defined thanks to a wide generalization of
the parallelism equivalence relation on sectors: the generalization applies to the family of subsets taken into account
(we use chimneys and not only sectors, see Section \ref{sec:3.2}), and also on the equivalence relations (roughly speaking,
we take into account the value of the Hausdorff distance between these subsets). The inner fa\c cades are (inessential)
sub-buildings that are very useful for technical purposes. The former fa\c cades were recently introduced in order to construct,
in a purely combinatorial way and without any group consideration, the polyhedral compactification of an affine building
\cite{Charignon2009}, while the latter ones had already been introduced by Bruhat--Tits \cite[7.6]{BruhatTits1972} in a
group-theoretic context, which we can eventually get rid of. We conclude this section by understanding maximal boundaries of
fa\c cades in terms of residues in the spherical building at infinity; this is useful when studying limits of harmonic measures
in Section \ref{s - Furst comp}. 

\subsection{Geometric and metric realizations}
\label{sec:12}
In this paper, even though our starting point is a combinatorial definition of buildings (which is well-adapted to our analytic 
arguments), we are sometimes led to using geometric realizations of affine buildings, as well as the associated non-positively 
curved distances. The sets of vertices on which we perform harmonic analysis are discrete subsets of their geometric 
realizations. 

For geometric realizations, we proceed as in J.~Tits' paper on the classification of affine buildings \cite{Tits1986}. 
This allows us to use G.~Rousseau's book \cite{Rousseau2023} in which apartments are by definition Euclidean
affine spaces (in order to directly treat non-discrete buildings). We will also use \cite{Charignon2009} which sticks
to the case we consider, namely locally finite affine buildings. Let us fix the following notational convention: if
$\scrA$ is a (discrete) apartment, then $\mathcal{A}$ denotes its geometric realization; $\mathcal{A}$ is thus a Euclidean
affine space on which the affine Weyl group $W^a$ acts by isometries. Accordingly, we denote by $\mathcal{X}$ the
geometric realization of $\scrX$ and we adopt the same convention for sectors: if $\scrS$ is a sector, we denote by
$\mathcal{S}$ its geometric realization. This additional structure on the geometric realization of each apartment is the
suitable context to define cones in apartments (in the usual real affine sense, {\it i.e.} via stability under $\RR_+$-action by
scalars) and we can also use the notions of vectorial and conical directions in apartments. Each sector $\mathcal{S}$
is a simplicial cone, and so is each of its faces, which we call a \emph{sector face} of $\mathcal{X}$ without
refereeing to any specific sector containing it. 

For any type of building, whenever the model for the apartments has a geometric realization admitting a distance which
is invariant under the Weyl group, the geometric realization of the ambient building can be endowed with a distance which
restricts to the initial one on each apartment \cite[Theorem 10A.4]{Bridson1999}. If we start with a Weyl-invariant
$\cat$-distance on the apartments, we obtain a $\cat$-distance on the building and each type-preserving automorphism
extends to an isometry. This is obviously the case for affine buildings, and actually the $\cat$-distance we will use was
considered from the very beginning of affine building theory \cite[Section 2.5]{BruhatTits1972}. The distance on
$\mathcal{X}$, as for any metric space, leads to the notion of \emph{Hausdorff distance} $d_{\rm H}$ between subsets.
Namely, for $A,B \subset \mathcal{X}$, we set 
\[
	d_{\rm H}(A,B) = \inf \big\{  \varepsilon>0 \, : \; A \subset V_\varepsilon(B) \,\, \hbox{\rm and} \,\, B \subset 
	V_\varepsilon(A) \big\}
\]
where $V_\varepsilon(C)$ denotes the $\varepsilon$-neighborhood $\{ x \in \mathcal{X} : d(x,C) < \varepsilon\}$ of
$C \subset \mathcal{X}$. We say that two subsets in $\mathcal{X}$ are \emph{asymptotic}~if they are at finite Hausdorff
distance from one another; this is an equivalence relation. 

The non-positively curved distance $d$ on an affine building $\mathcal{X}$ provides an intrinsic definition of geodesic
segments between two points, namely: 
\[
	[x,y] = \big\{ z \in \mathcal{X} \, | \, d(x,z)+d(z,y) = d(x,y) \big\}
\]
for any $x,y \in \mathcal{X}$. Geodesic segments themselves are the starting point of convexity arguments. The convex hull
of a subset $Y$ of $\mathcal{X}$ is the smallest subset containing it and stable by taking geodesic segments between points
in $Y$; when $Y$ is contained in an apartment, it is also the intersection of the affine half-spaces contained in the apartment
and containing $Y$. The discrete version of the convex hull in this context was defined in \cite[2.4]{BruhatTits1972}: 
the \emph{enclosure} of a subset $Y$ in $\mathcal{X}$ intersecting an alcove is the smallest subset of $\mathcal{X}$ containing
$Y$ and stable by taking (closures of) minimal galleries between alcoves intersecting $Y$; we denote it by $\mathrm{encl}(Y)$. 
When $Y$ is contained in an apartment, it is also the intersection of the (closed) affine half-spaces contained in the apartment,
bounded by a wall and containing $Y$. The advantage of the latter definition is that it applies to arbitrary subsets of apartments
\cite[Proposition 2.4.5]{BruhatTits1972}. 

Let us denote by $\partial_\infty \mathcal{X}$ the set of equivalence classes of geodesic rays; we call it
the \emph{visual boundary} of $\mathcal{X}$. By \cite[Section 3.2.13]{Rousseau2023}, $\partial_\infty \mathcal{X}$ is a 
geometric realization of the \emph{spherical building at infinity} $\scrX^\infty$. This allows us to see any facet $F$ of 
$\scrX^\infty$ as a set of asymptotic classes of geodesic rays. There is a unique apartment system in 
$\partial_\infty \mathcal{X}$ which consists of the boundaries $\partial_\infty \mathcal{A}$ of the apartments 
$\mathcal{A}$ in $\mathcal{X}$ (recall that we systematically deal with full apartment systems for all the affine buildings 
in this paper). For any cone $\mathcal{C}$ in some apartment $\mathcal{A}$ (e.g. for any sector face in $\mathcal{A}$), 
we denote by $\partial_\infty \mathcal{C}$ the set of asymptotic classes of rays drawn in $\mathcal{C}$.
The set of asymptotic classes of sector faces is the set of facets in the visual boundary $\partial_\infty \mathcal{X}$ 
(see \cite{Tits1986} where the asymptotic equivalence relation is called \emph{parallelism}, or also 
\cite[Section 11.8]{Abramenko2008}). 

\subsection{Sector faces, chimneys and equivalence relations}
\label{sec:3.2}
Let us recall the construction of an affine building associated with each facet in $\partial_\infty \mathcal{X}$. The resulting 
collection of such buildings gives a stratification of the boundary of almost all the compactifications we consider in this 
paper (the exceptions are the Gromov and the Martin
compactifications above the bottom of the spectrum: we explain this in Section \ref{sec:7.3}). The affine buildings at
infinity, indexed by the facets in $\partial_\infty \mathcal{X}$, are called \emph{fa\c cades} in 
\cite[Section 3.3]{Rousseau2023}; they provide an important geometric viewpoint on the study of convergence of core
sequences.

Recall first that a {\it sector face}~in the building $\mathcal{X}$ is a cone $\mathcal{F}$, in some apartment
$\mathcal{A}$, of the form $\mathcal{F}=x+\vec{F}$ where $x \in \mathcal{A}$ and $\vec{F}$ is some face in the
vectorial model of $\mathcal{A}$; we say then that $\vec{F}$ is the \emph{direction}~of the sector face. This notion can
be generalized in the following way: for any facet $\sigma$ in $\mathcal{A}$, the enclosure 
$\mathrm{encl}(\sigma+\vec{F})$ is called the \emph{chimney}~$\mathcal{R}(\sigma, \vec{F})$ \emph{based}~at $\sigma$ and
\emph{directed}~by $\vec{F}$ in $\mathcal{A}$, see \cite[Definition 4.2.2]{Charignon2009}. 

Loosely speaking, taking into account the numerical values of the Hausdorff distance between sector faces (or, more generally,
between cones) in the same asymptotic class leads to a more subtle equivalence relation on cones with a given direction.
Chimneys are technical tools which are used to define a simplicial
structure (eventually, a building structure) on the corresponding set of equivalence classes. 

More precisely, the asymptotic equivalence relation can also be formulated in more combinatorial terms, see
\cite[Definition 4.7.1]{Charignon2009}. Let $\mathcal{F}=x+\vec{F}$ be a sector face. A subsector face $\mathcal{F}'$
of $\mathcal{F}$ ({\it i.e.}, a sector face $\mathcal{F}'$ contained in $\mathcal{F}$) is said to be \emph{full} in 
$\mathcal{F}$ if the sector faces $\mathcal{F}$ and $\mathcal{F}'$ have the same direction. Similarly, but taking into
account the fact that the base of a chimney is now a facet and not a mere point, we say that a subchimney $\mathcal{R}'$
of $\mathcal{R}$ ({\it i.e.}, a chimney $\mathcal{R}'$ contained in $\mathcal{R}$) is \emph{full} in $\mathcal{R}$, if
$\mathcal{R}'$ and $\mathcal{R}$ have the same direction and generate the same affine subspace in any apartment containing
$\mathcal{R}$. We define the \emph{germ (at infinity)}~$\igerm(\mathcal{R})$ of a chimney $\mathcal{R} = 
\mathrm{encl}(\sigma+\vec{F})$: a chimney $\mathcal{R}'$ belongs to $\igerm(\mathcal{R})$ if and only if
$\mathcal{R} \cap \mathcal{R}'$ contains a chimney which is full in both $\mathcal{R}$ and $\mathcal{R}'$. 

\begin{figure}[ht!]
\begin{center}
  \includegraphics[width=0.25\linewidth]{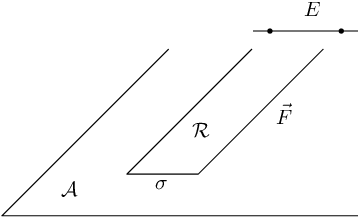}
  \hspace{0.05\linewidth}
  \includegraphics[width=0.25\linewidth]{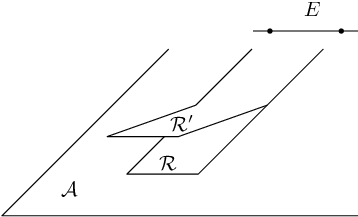}
  \hspace{0.05\linewidth}
  \includegraphics[width=0.30\linewidth]{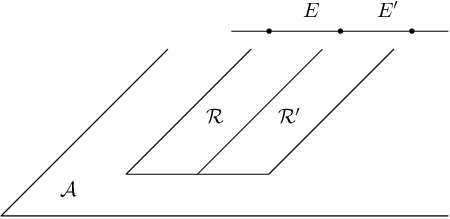}
  \caption{The first figure shows a chimney, the second one shows two equivalent chimneys and the third one illustrates 
  adjacency of chimneys (and of their classes) in the fa\c cade.}
  \label{fig:4}
\end{center}
\end{figure}

\subsection{Affine buildings at infinity, or (outer) fa\c cades}
\label{sec:3.3}
Let us choose $F$ a facet in $\scrX^\infty$. We denote by $\mathcal{X}(F)$ the set all
germs of sector faces in the asymptotic class $F$, that is
\[
	\mathcal{X}(F) = 
	\big\{ \igerm(\mathcal{F}) : \mathcal{F} \text{ sector face such that } \partial_\infty \mathcal{F} = 
	F \big\}.
\]
The set $\mathcal{X}(F)$ admits an affine building structure \cite[Section 3.3.15]{Rousseau2023} where the
apartments are given by apartments $\mathcal{A}$ such that $\partial_\infty \mathcal{A} \supset F$: for such 
an $\mathcal{A}$, the corresponding apartment in $\mathcal{X}(F)$ is 
 \[
 	\mathcal{A}(F) = 
	\big\{ \igerm(\mathcal{F}) : \mathcal{F} \text{ sector face contained in $\mathcal{A}$ and such that } 
	\partial_\infty \mathcal{F} = F \big\}.
\]
The facets are the germs of chimneys whose direction is given by $F$ \cite[Section 2.6]{Ciobotaru2020}. 
Note that the building $\mathcal{X}(F)$ is not contained in $\mathcal{X}$ since it is abstractly constructed out
of germs at infinity of vector faces. Nevertheless there is a natural map 
\begin{alignat*}{1}
	\pi_{F} : \mathcal{X} & \to \mathcal{X}(F) \\
	x & \mapsto \igerm ( x + F)
\end{alignat*}
which is a homomorphism of buildings in the sense of \cite[Section 2.1.13]{Rousseau2023}. 

There is a good compatibility between this construction and buildings at infinity. Let $R(F)$ denote the residue of
$F$ in the spherical building at infinity $\scrX^\infty$. Let $\mathcal{S}$ be a sector one of whose
sector faces represents $F$: this means that $F$ is a facet of $\partial_\infty\mathcal{S}$, or also that
the parallelism class of $\mathcal{S}$, denoted by  $[\mathcal{S}]_{/\!\!/}$, belongs to $R(F)$. The sectors in
$\mathcal{X}(F)$ are the images under the above map $\pi_{F}$ of such sectors \cite[Section 2.6]{Ciobotaru2020}. 
Let $\mathcal{S}$ and $\mathcal{S}'$ be two such sectors. If we assume in addition that they are parallel, there is a subsector
$\mathcal{S}''$ in $\mathcal{S} \cap \mathcal{S}'$. Then $\pi_{F}(\mathcal{S}'')$ is a common subsector of the sectors
$\pi_{F}(\mathcal{S})$ and $\pi_{F}(\mathcal{S}')$ in $\mathcal{X}(F)$, so that 
$\pi_{F}(\mathcal{S})$ and $\pi_{F}(\mathcal{S}')$ are parallel. This allows us to define the surjective map 
\begin{alignat*}{1}
	\varphi_{F} : R(F) & \to \Omega_{\mathcal{X}(F)} \\
	[\mathcal{S}]_{/\!\!/} & \mapsto [\pi_{F}(\mathcal{S})]_{/\!\!/}
\end{alignat*}
where the notation $[\cdot]_{/\!\!/}$ corresponds to taking the relevant parallelism class on each side of the map. 
If $J$ is the type of $F$, the maximal boundary $\Omega_{\mathcal{X}(F)}$ is the residue $\res_J(c_\infty)$ where $c_\infty$ 
is any chamber at infinity whose closure contains $F$. 

We can also put a building structure on a well-chosen subset of $\mathcal{X}$, making it a sub-building in the sense of 
\cite[Section 2.1.15]{Rousseau2023}. This requires to choose an apartment $\calA$ for which 
$F \subset \partial_\infty \calA$. We assume that $o \in \calA$. We also fix a facet $F'$ opposite $F$
in $\calA$ (which we may write $F' = {\mathrm{opp}}_{\calA}(F)$ sometimes). We consider the subset of apartments
\[
	\mathcal{A}(F,F') = 
	\big\{ \mathcal{A} \in \mathscr{A} : \partial_\infty \mathcal{A} \supset F \cup F' \big\}
\]
and we set 
\[
	\mathcal{X}(F,F') = \bigcup_{\mathcal{A}\in \mathcal{A}(F,F')} \mathcal{A}.
\]
Then $\mathcal{X}(F,F')$ is an affine building for which $\mathcal{A}(F,F')$ is a system of apartments 
\cite[Theorem 3.3.14.1]{Rousseau2023}. In general, this building is \emph{non-essential}, meaning that the smallest vectorial 
facets in any apartment are of positive dimension. Such a facet is contained in the Euclidean factor acted upon trivially when 
decomposing the affine Weyl group action on an apartment (this corresponds to the Euclidean factor in the de Rham decomposition 
of Riemannian symmetric spaces). Note that there always exists an \emph{essentialization map} factoring out the trivial Euclidean
factor of each apartment \cite[Section 2.1.12]{Rousseau2023}. The sub-building $\mathcal{X}(F,F')$ is called the 
\emph{inner fa\c cade} associated with $F$ and $F'$ \cite[Section 3.3.14]{Rousseau2023}. The restricted map
$\pi_{F}|_{\mathcal{X}(F,F')}$ is the essentialization map of $\mathcal{X}(F,F')$: see also \cite[Theorem 2.8]{Ciobotaru2020}, 
where more precise statements can be found.

Let us see what the concrete counterpart of these constructions is, at the level of apartments; this will be useful when
studying the convergence of core sequences for some compactifications. We work in the Euclidean space $\mathfrak{a}$ that
is the vectorial model of the apartments in $\mathcal{X}$, hence in $\scrX$ (see Section \ref{sec:1}). Given
$J \subsetneq I_0$, let $\Phi_J$ to be the set of roots $\alpha \in \Phi$ such that $\sprod{\alpha}{\lambda_k} = 0$ if
$k \notin J$. Let $\tilde{\alpha}_0$ be the highest root in $\Phi_J$. In $\mathfrak{a}$, we consider the subcollection of
walls $\calH_J$ consisting of all hyperplanes $H_{j; k}$ for $j \in J$ and $k \in \ZZ$, together with
\[
	H_{\tilde{\alpha}_0; k} = \big\{x \in \mathfrak{a} : \sprod{x}{\tilde{\alpha}_0} = k\big\}
	\quad\text{for all } k \in \ZZ.
\]
Let $\tilde{r}_0$ be the orthogonal reflection with respect to $H_{\tilde{\alpha}_0; 1}$. We denote by $W_J$ and
$W^a_J$ the subgroups of $W$ and $W^a$ generated by $\{r_j : j \in J\}$ and $\{\tilde{r}_0 \} \cup \{r_j : j \in J\}$,
respectively. Let $C_J(\Sigma)$ be the family of all open connected components of
$\mathfrak{a} \setminus \bigcup_{H \in \calH_J} H$. Since the group $W_J^a$ acts transitively on $C_J(\Sigma)$, we can
turn it into a chamber system. Let $\Sigma_J$ denote the resulting abstract simplicial complex. 
The space $\mathfrak{a}$ is an inessential realization of $\Sigma_J$ in the sense of \cite[Chapter V]{Bourbaki2002}. 
Since $\calH_J \subset \calH$, we can see $\Sigma_J$ as a subcomplex of $\Sigma$. In particular, each half apartment in
$\Sigma_J$ is a half apartment in $\Sigma$. For each $(k_j : j \in J) \in \ZZ^J$, the set
\[
	\big\{x \in \mathfrak{a} : \sprod{x}{\alpha_j} = k_j \text{ for all } j \in J \big\}
\]
is called a $J$-vertex. The fundamental $J$-sector is
\[
	S^J = \big\{x \in \mathfrak{a} : \sprod{x}{\alpha_j} \geqslant 0 \text{ for all } j \in J \big\}.
\]
All $J$-sectors in $\Sigma_J$ are of the form $w . S^J$ for some $w \in W_J^a$. 

In order to obtain from $\mathfrak{a}$ an essential realization of $\Sigma_J$ we need to introduce the orthonormal
projection $\pi_J$ onto $\mathfrak{a}_J = \bigoplus_{j \in J} \mathbb{R} \alpha_j$ which is a direct factor in the
orthogonal decomposition $\mathfrak{a} = \mathfrak{a}_J \oplus \bigoplus_{i \in I_0 \setminus J} \mathbb{R} \lambda_i$. 
For $x \in \mathfrak{a}$, we have 
\[
	\pi_J(x) = P_J(x) = \sum_{j \in J} \sprod{x}{\alpha_j} P_J(\lambda_j)
\]
where
\begin{equation}
	\label{eq:77}
	P_J(x) = \frac{1}{|W_J|} \sum_{w \in W_J} \big(x - w . x\big).
\end{equation}
We also set $Q_J={\rm id}-P_J$. To see that $P_J$ indeed performs the orthogonal projection onto $\mathfrak{a}_J$,
note first that for each $j \in J$ and each $i \in I_0 \setminus J$ we have $s_j . \lambda_i=\lambda_i$, hence
$w . x = x$ for each $w \in W_J$ and each $x \in \bigoplus_{i \in I_0 \setminus J} \mathbb{R} \lambda_i$; moreover, 
the fact that $P_J$ acts as the identity on $\mathfrak{a}_J$ comes from the irreducibility of the standard linear representation
of an irreducible Coxeter system. 

To sum up, when $J$ varies over the proper subsets of $I_0$, the Euclidean spaces $\mathfrak{a}$ endowed with the
subcollections of walls $\calH_J$ are models for the apartments in the inner fa\c cades $\mathcal{X}(F,F')$
associated with opposite facets at infinity; after projection onto $\mathfrak{a}_J$, we obtain the models for the apartments
of the fa\c cades at infinity $\mathcal{X}(F)$ (or of the essentializations of inner fa\c cades). 

\begin{remark}
\label{rk - point at infinity for core} 
	Let $(x_n)$ be an $(\omega, J, c)$-core sequence with auxiliary sequence $(u_n)$. This sequence obviously defines a spherical
	facet at infinity, say $F$, namely the facet attached to the $J$-residue of $\omega$ in the spherical building at infinity 
	$\scrX^\infty$. Therefore we have a projection $\pi_F : \mathcal X \to \mathcal{X}(F)$. Let $\mathcal A$ be an apartment 
	containing the sector $[o,\omega]$. Then, in view of the above description of the restriction of $\pi_F$ to $\mathcal A$ and 
	by definition of a core sequence, we see that the sequence $(\pi_F(u_n) : n \in \NN)$ is constant; we denote by $x_F$ its 
	unique value. The facet $F$ is called the \emph{spherical facet} associated to the core sequence $(x_n)$ and the point $x_F$ 
	of $\mathcal{X}(F)$ is called the \emph{vertex at infinity} associated to $(x_n)$. 
\end{remark}

\subsection{Polyhedral compactification, after Rousseau}
\label{ss - polyhedral compactification}
In this section we introduce a further compactification procedure which we will use as a tool and which has the advantage to 
be well-adapted to projections to outer fa\c cades, so that the latter buildings naturally appear as strata at infinity. 
Apart from Lemma \ref{lem:core CV in poly} below, what follows is a summary of \cite[Section 3.4]{Rousseau2023} 
(to which we refer for more details). 

First of all, recall that in any complete, proper, ${\rm CAT}(0)$ metric space, the nearest point projection onto any given 
convex subset is a $1$-Lipschitz map \cite[Chapter II.2]{Bridson1999}. This fact is the starting point to define 
compactifications at this level of generality. If $Z$ is such a space endowed with an increasing sequence 
$( B_k )_{k \geqslant 0}$ of compact convex subsets such that $Z = \bigcup_k B_k$, and if one denotes by ${\rm pr}_k$ the nearest
point projection onto $B_k$, then the projective system $(B_k, {\rm pr}_k)$ provides a compactification 
$\overline Z = \varprojlim_{k \geqslant 0} B_k$. When $( B_k )_{k \geqslant 0}$ is given by a sequence of balls centered at a 
common point $z$ and of radii $r_k \to +\infty$ (for the given, \emph{i.e.} numerical, distance), the resulting compactification 
is the horospherical one \cite[Chapter II.8, Definition 8.5 and Theorem 8.13]{Bridson1999}.

In the case of affine buildings, Rousseau suggests to use different compact exhaustions $( B_k )_{k \geqslant 0}$, replacing
metric balls by balls defined in terms of vectorial distance. 

\subsubsection{Vectorial balls}
For a point $z$ in $\mathcal X$ and for a vector $\xi$ in the model $S_0$ of sectors, we introduce the 
\emph{vectorial ball} $B^v(z,\xi)$ of center $z$ and of vectorial radius $\xi$ by setting 
\[
	B^v(z,\xi) = \{ y \in \mathcal{X} : \sigma(z,y) \preccurlyeq \xi \},
\]
where $\sigma(z,y) \preccurlyeq \xi$ means that the vector $\xi - \sigma(z,y)$ belongs to the cone 
$\bigoplus_{i \in I_0} \RR_+ \alpha_i$ which dual to $S_0$.
Vectorial balls are compact and convex \cite[Lemma 3.4.3.1]{Rousseau2023} and the intersection of a vectorial ball with an 
apartment containing its center can be described precisely. Indeed, for each apartment $\mathcal A$ containing $z$, 
the (topological) boundary $\mathcal P$ of $\mathcal A \cap B^v(z,\xi)$ is an explicit polyhedron
\cite[Proposition 3.4.1.1]{Rousseau2023}: each vectorial face $F$ of the partition of $\mathcal A$ into Weyl chambers gives rise 
to a face of $\mathcal P$ which is orthogonal to $F$; this parametrizes bijectively the faces of the boundary of 
$\mathcal A \cap B^v(z,\xi)$. For instance, the faces corresponding to Weyl sectors are the vertices of $\mathcal P$ and 
$\mathcal A \cap B^v(z,\xi)$ is the convex hull of these vertices. 

Moreover the nearest point projection onto $B^v(z,\xi)$, restricted to any apartment containing the center $z$, can be described 
explicitly: the preimage of a point $u$ in the boundary $\mathcal P$, more precisely in the face determined by (and orthogonal 
to) a vectorial face $F$ say, consists of the cone $u+F$ \cite[Lemma 3.4.2.2]{Rousseau2023}.

\subsubsection{Polyhedral compactification}
Keeping $z$ and $\xi$ as above, we can choose in addition an  increasing sequence of real numbers $r_k \geqslant 1$, with 
$r_k \to +\infty$, and consider the dilated vectorial balls $B_k = B^v(z,r_k \xi)$. The \emph{polyhedral compactification}~of 
the building is then $\overline{\mathcal{X}}^{\rm pol} = \varprojlim_{k \geqslant 0} B_k$, so that a point $\eta$ in 
$\overline{\mathcal{X}}^{\rm pol}$ is a sequence $(z_k)$ such that $z_k \in B_k$ and ${\rm pr}_k(z_{k+m})=z_k$ for all 
$k,m \geqslant 0$; we denote $z_k = {\rm pr}_k(\eta)$. The projective limit topology on $\overline{\mathcal{X}}^{\rm pol}$ 
is Hausdorff and compact \cite[Theorem 3.4.4.4]{Rousseau2023} and it does not depend on the choice of the center $z$, 
the (regular) vectorial radius $\xi$ or the increasing sequence of real numbers $(r_k)$ going to infinity: this is seen by 
comparing the obtained projective limit compactifications with another one, defined more intrinsically 
\cite[Proposition 3.4.6.3]{Rousseau2023}. 

At last, for each spherical facet in the spherical building at infinity $\scrX^\infty$, there is a map $\iota_F$ embedding the 
outer fa\c cade $\mathcal X(F)$ into $\overline{\mathcal{X}}^{\rm pol}$: recall that a point $\eta$  in $\mathcal X(F)$ is 
the germ at infinity of a cone of direction $F$; then the image $\iota(\eta)$ is the sequence 
$\bigl( {\rm pr}_k (\eta) \bigr)_{k \geqslant 0}$ where, by definition, the point ${\rm pr}_k (\eta)$ in the boundary 
$\partial B_k$ is the image by the nearest point projection ${\rm pr}_k$ of a sufficiently small cone representing the germ 
$\eta$ \cite[\S 3.4.4.3]{Rousseau2023}. It turns out that $\overline{\mathcal{X}}^{\rm pol}$ is the disjoint union of 
the building and of the images of the outer fa\c cades $\mathcal X(F)$ by $\iota_F$ when $F$ varies over the spherical facets 
at infinity \cite[Theorem 3.4.4.4]{Rousseau2023}. 

\subsubsection{Convergence of core sequences in the polyhedral compactification}
Recall that an $(\omega, J, c)$-core sequence $(x_n : n \in \NN)$ with auxiliary sequence $(u_n : n \in \NN)$ defines a 
spherical facet at infinity $F$ of type $J$ and a vertex $x_F$ in the outer fa\c cade $\mathcal X(F)$ 
(Remark \ref{rk - point at infinity for core}). 

\begin{lemma}
	\label{lem:core CV in poly}	
	Let $(x_n)$ be an $(\omega, J, c)$-core sequence with auxiliary sequence $(u_n)$. Let $F$ be the associated spherical 
	facet in the building at infinity $\scrX^\infty$ and let $x_F$ be the associated point at infinity. 
	Then, in the polyhedral compactification, we have 
	\[
		\lim_{n \to \infty} x_n= x_F. 
	\]
\end{lemma}
\begin{proof}
	We first show that in $\overline{\mathcal{X}}^{\rm pol}$ we have $\displaystyle \lim_{n \to \infty} u_n= x_F$, by arguing 
	in an apartment $\mathcal{A}$ containing the sector $[o,\omega]$. In view of the definition of a projective limit, it
	amounts to showing that for each $k \geqslant 0$ we have
	\[
		\lim_{n \to \infty} {\rm pr}_k(u_n) = {\rm pr}_k(x_F).
	\]
	It is enough to check this for large enough $k$. We consider an index $k$ such that the face $\Sigma_k^{\perp F}$ of 
	the polyhedron $\partial B_k$ which is orthogonal to $F$ contains a point $u$ such that $\alpha_j(u) \geqslant c_j$ for 
	each $j \in J$. On the one hand, when $n$ is large enough, we have $u_n \not\in B_k$ and by the above description of 
	the fibers of ${\rm pr}_k$ in $\mathcal A$, the point ${\rm pr}_k(u_n)$ is the unique point $s_k \in \Sigma_k^{\perp F}$ 
	such that $\alpha_j(s_k)=c_j$ for each $j \in J$. On the other hand, the value of ${\rm pr}_k$ at the point at infinity 
	$x_F$ is the (unique) value that ${\rm pr}_k$ takes on the cone $u_n + F$ for $n$ large enough. 
	The implies that for $n$ large enough, we have ${\rm pr}_k(u_n) = {\rm pr}_k(x_F)$, hence the desired convergence. 

	Now we consider the $(\omega, J, c)$-core sequence $(x_n)$ itself. Let $k$ be again a large enough index (in the above sense)
	and let $n$ be large enough too, so that $u_n \not\in B_k$. We pick an apartment $\mathcal{A}'$ containing the origin $o$ 
	and $x_n$. Since $\Omega(o,x_n) \subseteq \Omega(o, u_n)$, the point $u_n$ lies in the convex hull of $o$ and $x_n$ hence 
	in $\mathcal A'$. Therefore we can argue as in the previous paragraph, using the apartment $\mathcal A'$ instead of 
	$\mathcal A$. This implies that ${\rm pr}_k(x_n)$ and ${\rm pr}_k(u_n)$ are both equal to the unique point 
	$s_k \in \Sigma_k^{\perp F}$ such that $\alpha_j(s_k)=c_j$ for each $j \in J$, hence the announced convergence. 
\end{proof}

\subsection{Some lemmas for convergence studies}
\label{ss - lemmas for CV}
We finish with two lemmas that will be useful while studying convergence of core sequences. 

\begin{lemma}
	\label{lem:1}
	Let $F$ be a facet in the building at infinity $\scrX^\infty$. 
	Then the map
	\begin{alignat*}{1}
		\varphi_{F} : R(F) & \to \Omega_{\mathcal{X}(F)} \\
		[\mathcal{S}]_{/\!\!/} & \mapsto [\pi_{F}(\mathcal{S})]_{/\!\!/}
	\end{alignat*}
	is a homeomorphism.
\end{lemma}
\begin{proof}
	We already know that the map $\varphi_{F}$ is surjective since it is obtained from factorizing by parallelism
	equivalence relations (at the levels of the source and of the target) the map obtained from $\pi_{F}$ sending
	surjectively the set of sectors in ${\mathcal X}$ containing a sector face directed by $F$ to the set of sectors in 
	${\mathcal X}(F)$. Now, let ${\mathcal S}$ and ${\mathcal S}'$ be sectors both containing a sector face directed by $F$. 
	here are subsectors ${\mathcal T}$ and ${\mathcal T}'$ of ${\mathcal S}$ and ${\mathcal S}'$ respectively, contained in 
	the same apartment ${\mathcal A}$ and both containing a sector face directed by $F$. We assume that ${\mathcal S}$ and 
	${\mathcal S}'$ are not parallel to one another. Then neither are ${\mathcal T}$ and ${\mathcal T}'$ in ${\mathcal A}$, so 
	there exists $w$ in the spherical Weyl group, non-trivial but stabilizing $F$, sending ${\mathcal T}$ to
	${\mathcal T}'$. After projecting by $\pi_{F}$ this provides two sectors in the apartment 
	$\pi_{F}({\mathcal A})$ which are deduced from one another by the action of a non-trivial element of the Weyl group of 
	${\mathcal X}(F)$, hence not parallel to one another. This proves the injectivity of $\varphi_{F}$. 
	
	To finish the proof, we use compactness of the source $R(F)$ of $\varphi_{F}$ (given by Lemma 
	\ref{lem - residues are closed}). Therefore, in order to conclude that $\varphi_{F}$ is a homeomorphism, it is enough to 
	show that it is continuous. We pick $x_\infty,y_\infty$ in $\mathcal{X}(F) \cap \pi_{F}(V_s)$ and 
	$\omega$ in $\varphi_{F}^{-1}\bigl( \Omega_{\mathcal{X}(F)}(x_\infty,y_\infty) \bigr)$, where
	$\Omega_{\mathcal{X}(F)}(x_\infty,y_\infty)$ is the shadow in the maximal boundary 
	$\Omega_{\mathcal{X}(F)}$ of $\mathcal{X}(F)$ defined by $x_\infty$ and $y_\infty$. By definition of a shadow, we have
	$y_\infty \in [x_\infty,\varphi_{F}(\omega)]$. The sector $[x_\infty,\varphi_{F}(\omega)]$ is contained in
	an apartment of the fa\c cade $\mathcal{X}(F)$, which we can write $\pi_{F}(\mathcal{A})$ for an apartment 
	${\mathcal A}$ in ${\mathcal X}$ such that $\partial_\infty {\mathcal A}$ contains $F$. 
	Let $x \in {\mathcal A} \cap V_s$ be such that $\pi_{F}(x)=x_\infty$. In ${\mathcal A}$, there exists a sector tipped 
	at $x$ whose image by $\pi_{F}$ is in the parallelism class $\varphi_{F}(\omega)$. By injectivity of 
	$\varphi_{F}$ this sector represents $\omega$ so we can write $[x_\infty,\varphi_{F}(\omega)] = 
	\pi_{F}([x,\omega])$. Now, we choose $y \in [x,\omega] \cap V_s$ such that $\pi_{F}(y)=y_\infty$ and we
	consider the shadow $\Omega(x,y)$. Let $\omega' \in R(F) \cap \Omega(x,y)$. We have $y \in [x,\omega']$, so applying
	$\pi_{F}$ we obtain $y_\infty \in [x_\infty,\varphi_{F}(\omega')]$. This shows that the open neighborhood 
	$\Omega(x,y)$ of $\omega$ is contained in
	$\varphi_{F}^{-1}\bigl( \Omega_{\mathcal{X}(F)}(x_\infty,y_\infty) \bigr)$. This proves the continuity of 
	$\varphi_{F}$ since $\omega$ was picked arbitrarily in the preimage
	$\varphi_{F}^{-1}\bigl( \Omega_{\mathcal{X}(F)}(x_\infty,y_\infty) \bigr)$. 
\end{proof}

\begin{lemma}
	\label{lem:6}
	Let $F$ be a facet in the building at infinity $\scrX^\infty$. Let $J$ be its type, $J \subsetneq I_0$.
	\begin{enumerate}[label=(\roman*), ref=\roman*]
		\item
		\label{en:7:1}
		For all $x,x' \in \mathcal{X}$ with $R(F) \cap \Omega(x, x') \neq \varnothing$, we have
		\[
			\sigma_{\mathcal{X}(F)}  \bigl( \pi_F(x),\pi_F(x') \bigr) = P_J \bigl( \sigma(x,x') \bigr)
		\]
		where $\sigma_{\mathcal{X}(F)} $ denotes the vectorial distance of the outer fa\c cade $\mathcal{X}(F)$. 
		\item 
		\label{en:7:2}
		For all $x,x' \in \mathcal{X}$ and each $\omega \in R(F)$, we have
		\[
			h_{\mathcal{X}(F)} \big(\pi_F(x), \pi_F(x'); \vphi_{F}(\omega)\big) = P_J\big(h(x, x'; \omega)\big)
		\]
		where $h_{\mathcal{X}(F)} $ denotes the Busemann function of the outer fa\c cade $\mathcal{X}(F)$. 
	\end{enumerate}
\end{lemma}
\noindent
In the rest of the paper we are mainly interested in outer fa\c cades in order to understand boundaries of various 
compactifications. This is why we did not introduce the retraction $\rho : \mathcal{X} \to \mathcal{X}(F,F')$ above; 
it is defined carefully in \cite[Section 3.3.17]{Rousseau2023}. We use this reference in the proof below. 
\begin{proof}
	In order to prove \eqref{en:7:1}, we pick $x,x' \in \mathcal{X}$ and $\omega \in R(F) \cap \Omega(x, x')$.
	We have $x' \in [x, \omega]$ and the sector $[x, \omega]$ contains a face whose direction is given by $F$.
	There exists an apartment, say $\mathcal{A}$, containing the sector $[x, \omega]$. The apartment $\mathcal{A}$ contains 
	a facet $F'$ which is opposite $F$. We denote by $\mathcal{X}(F,F')$ the inner fa\c cade defined by the couple $(F,F')$: 
	by construction, the apartment $\mathcal{A}$ is an apartment of $\mathcal{X}(F,F')$ so that the retraction 
	$\rho : \mathcal{X} \to \mathcal{X}(F,F')$ induces the identity map on it. The result follows then from the fact that 
	$\pi_F$ is the composition of $\rho$ with the essentialization of map \cite[\S 3.3.17]{Rousseau2023}, and the latter map 
	is performed by $P_J$ on the model of the apartments of $\mathcal{X}$. 

	To show \eqref{en:7:2}, we pick $x,x' \in \mathcal{X}$ and $\omega \in R(F)$. Let $z \in [x,\omega] \cap [x', \omega]$. 
	The point $z$ is chosen in order to write $h(x, x'; \omega) = \sigma(x,z) - \sigma(x',z)$. By the very choice of $z$, 
	we have $\omega \in \Omega(x,z) \cap R(F)$, therefore by \eqref{en:7:1} we have 
	\[
		\sigma_{\mathcal{X}(F)} \bigl( \pi_F(x),\pi_F(z) \bigr) = P_J \bigl( \sigma(x,z) \bigr).
	\]
	Similarly, by replacing $x$ by $x'$ we obtain 
	\[
		\sigma_{\mathcal{X}(F)}  \bigl( \pi_F(x'),\pi_F(z) \bigr) = P_J \bigl( \sigma(x',z) \bigr).
	\]
	Now $\pi_F\big( [x,\omega] \big)$ is a sector in $\mathcal{X}(F)$; more precisely the sector 
	$[\pi_F(x), \varphi_F(\omega)]$. Moreover, 
	\[
		\pi_F(z) \in [\pi_F(x), \varphi_F(\omega)].
	\]
	Similarly by replacing $x$ by $x'$ we obtain
	\[
		\pi_F(z) \in [\pi_F(x'), \varphi_F(\omega)].
	\]
	Therefore, $\pi_F(z) \in [\pi_F(x), \varphi_F(\omega)] \cap [\pi_F(x'), \varphi_F(\omega)]$, so that we can compute
	\[
		h_{\mathcal{X}(F)} \big(\pi_F(x), \pi_F(x'); \vphi_{F}(\omega)\big) 
		= \sigma_{\mathcal{X}(F)}  \bigl( \pi_F(x),\pi_F(z) \bigr) - \sigma_{\mathcal{X}(F)}  \bigl( \pi_F(x'),\pi_F(z) \bigr)
	\]
	which is equal to $P_J \bigl( \sigma(x,z) \bigr) - P_J \bigl( \sigma(x',z) \bigr)$ by the previous arguments. 
	Finally, by linearity of $P_J$ we obtain 
	\[
		h_{\mathcal{X}(F)} \big(\pi_F(x), \pi_F(x'); \vphi_{F}(\omega)\big) 
		= P_J \big( \sigma(x,z) - \sigma(x',z) \big)
		= P_J\big(h(x, x'; \omega)\big),
	\]
	as requested. 
\end{proof}

We finish with a lemma relating core sequences and fa\c cades at infinity. In the following result, we use the fact that 
the polyhedral compactification, when defined as the projective limit of an increasing exhaustion by convex compact subsets, 
does not depend on the point at which the convex compact subsets are centered. 
\begin{lemma}
	\label{lem:4}
	Let $(x_n : n \in \NN)$ be an $(\omega, J, c)$-core sequence with auxiliary sequence $(u_n : n \in \NN)$. 
	Let $F$ be the associated spherical facet in the building at infinity $\scrX^\infty$.
	Then the sequence $(\pi_F(u_n) : n \in \NN)$ is constant in the outer fa\c cade $\mathcal{X}(F)$. Let $x_F$ be its
	value. Then for each $z \in V_g$, there is an extraction $\varphi : \NN \to \NN$ such that
	\[
		\sigma_{\mathcal{X}(F)} \bigl( \pi_F(z), x_F \bigr) = P_J \bigl( \sigma(z, x_{\varphi(n)}) \bigr)
	\]
	for each $n \in \NN$.
\end{lemma}
\begin{proof}
	By Lemma \ref{lem:6}\eqref{en:7:1} we have 
	$P_J \bigl( \sigma(o, u_n) \bigr) = \sigma_{\mathcal{X}(F)} \bigl( \pi_F(o), \pi_F(u_n) \bigr)$. 
	Moreover, by definition of a core sequence, the auxiliary sequence $(u_n)$ satisfies 
	$P_J \bigl( \sigma(o, u_n) \bigr) = P_J \bigl( \sigma(o, x_n) \bigr)$. Therefore for all $n$ we have 
	\[
		\sigma_{\mathcal{X}(F)} (\pi_F(o), x_F) = P_J \bigl( \sigma(o, x_n) \bigr).
	\]
	Now, choosing the vertex $z \in V_g$ as a new origin, by Lemma \ref{lem:2} we have an extraction $\varphi$ and parameters 
	$(\omega', J', c')$ such that $(x_{\varphi(n)})$ is an $(\omega', J', c')$-core sequence 
	with respect to $z$. The argument of the previous paragraph shows that for all $n$ we have
	\[
		\sigma_{\mathcal{X}(F')} (\pi_{F'}(z), x_{F'}) = P_{J'} \bigl( \sigma(z, x_{\varphi(n)}) \bigr)
	\]
	where $F'$ is the spherical facet attached to the parameters $(\omega', J', c')$ and $x_{F'}$ is the corresponding point 
	in the fa\c cade at infinity $\mathcal{X}(F')$. By Lemma \ref{lem:core CV in poly}, the point $x_F$ is the limit of $(x_n)$ 
	in the polyhedral compactification endowed with the projective limit topology provided by vectorial balls centered at $o$; 
	similarly $x_{F'}$ is the limit of the subsequence $(x_{\varphi(n)})$ in the polyhedral compactification endowed with the 
	projective limit topology provided by vectorial balls centered at $z$. But these two topologies coincide 
	\cite[Proposition 3.4.6.3]{Rousseau2023} and are Hausdorff \cite[\S 3.4.4.2]{Rousseau2023}, so $x_F = x_{F'}$. 
	By \cite[\S 3.4.4]{Rousseau2023}, this implies that $F=F'$ and $J=J'$, hence the result by the second displayed formula. 
\end{proof}

\begin{lemma}
	\label{lem:9}
	Let $(x_n)$ be a sequence satisfying one of the following conditions: 
	\begin{itemize} 
		\item $(x_n)$ is an angular $(\omega, \theta)$-sequence such that for all $i \in J$,
		\[
			\liminf_{n \to \infty} \sprod{\sigma(o, x_n)}{\alpha_{i}} = +\infty, \text{or}
		\]
		\item $(x_n)$ is a $(\omega, J, c)$-core sequence.
	\end{itemize}
	Let $(u_n)$ be an auxiliary sequence.
	Let $F$ be the associated spherical facet in the building at infinity $\scrX^\infty$.
	Then the sequence $(\pi_F(u_n) : n \in \NN)$ is constant in the outer fa\c cade $\mathcal{X}(F)$. Let $x_F$ be its
	value. Then for each $y \in V_g$, we have
	\begin{equation}
		\label{eq:78}
		\lim_{n \to \infty} \sigma(o,x_n) - \sigma(y,x_n) = \sigma_{\mathcal{X}(F)} \bigl( \pi_F(o), x_F \bigr) 
		- \sigma_{\mathcal{X}(F)} \bigl( \pi_F(y), x_F \bigr) + Q_J \bigl( h(o,y;\omega) \bigr), 
	\end{equation}
	where in fact the sequence $\sigma(o,x_n) - \sigma(y,x_n)$ is eventually constant. 
\end{lemma}
\begin{proof}
	Let $(x_n')$ be any $(\omega, J, c)$-core subsequence of $(x_n)$. 
	We first note that $\sigma(o, x_n') = \sigma(o, u_n') + \sigma(u_n', x_n')$ because 
	$\Omega(o, x_n') \subseteq \Omega(o, u_n')$. We claim that the following holds true. 
	\begin{claim} 
		For any $y \in V_s$ there is an integer $n_y$ such that $\sigma(y, x_n') = \sigma(y, u_n') + \sigma(u_n', x_n')$ 
		for any $n \geqslant n_y$. 
	\end{claim}
	Since we are using full apartment systems, by \cite[Proposition 3.1.2]{Rousseau2023} there is an apartment $\scrA'$ 
	containing $y$ and a shortening of $u_n' + F_J$, hence $u_n' + F_J$ for $n \geqslant n_y$ for some $n_y$. There is $\omega'$ 
	in the boundary of $\scrA'$ belonging to the $J$-residue of $\omega$ such that the cone $[y,\omega']$ contains $u_n'$ for 
	$n \geqslant n_y$ possibly after increasing $n_y$. The intersection of the cone $[x_n',-\omega]$ with the support of $F_J$ 
	contains $u_n'$, and since $\omega$ and $\omega'$ are $J$-equivalent, $u_n'$ also belongs to the intersection of 
	$[x_n',-\omega']$ with the support of $F_J$. Therefore $[x_n',-\omega']$ contains $[u_n',-\omega']$ which contains $y$.
	This implies the claim. 

	The claim provides the equality: $\sigma(o,x_n') - \sigma(y,x_n') = \sigma(o,u_n') - \sigma(y,u_n')$ for $n \geqslant n_y$, 
	so it remains to compute the limit \eqref{eq:78} with $x_n'$ replaced by $u_n'$. As before, there is an apartment $\scrA'$ 
	containing $y$ and  $u_n'$ for $n \geqslant n_y$. Let $\psi$ be an isomorphism between this apartment and the model $\Sigma$ 
	such that $\psi([y,\omega]) = S_0$. There exists $w \in W_J$ such that $w.\psi(u_n') \in S_0$ for $n \geqslant n_y$.
	Let $\tilde{u}_n$ be the preimage by $\psi$ of $w.\psi(u_n')$. There exists $n_y' \geqslant n_y$ such that for 
	$n \geqslant n_y'$ we have $\tilde{u}_n \in [u_{n_y}',\omega]$. Then $Q_J \bigl( \sigma(u_{n_y}', u_n') \bigr) = 
	Q_J \bigl( \sigma(u_{n_y}', \tilde{u}_n) \bigr)$. Moreover, $\sigma(y, u_n')=\sigma(y, \tilde{u}_n)$. Therefore,
	\begin{align*}
		Q_J\big(\sigma(u_{n_y}', u_n') - \sigma(y, u_n')\big)
		&=
		Q_J\big(\sigma(u_{n_y}', \tilde{u}_n) - \sigma(y, \tilde{u}_n)\big) \\
		&=
		Q_J\big(h(u_{n_y}', y; \omega)\big).
	\end{align*}
	Because $u_{n_y}' \in [o, \omega]$, by the cocycle relation we obtain: 
	\begin{align*}
		Q_J\big(\sigma(o, u_n') - \sigma(y, u_n')\big)
		&=
		Q_J\big(\sigma(o, u_{n_y}')\big) + Q_J\big(\sigma(u_{n_y}', u_n') - \sigma(y, u_n')\big) \\
		&=
		Q_J\big(h(o, u_{n_y}'; \omega\big) + Q_J\big(h(u_{n_y}', y; \omega)\big) \\
		&=
		Q_J\big(h(o, y; \omega)\big).
	\end{align*}
	Finally, to complete the proof we just invoke Lemma \ref{lem:4} for the $P_J$-projections of $\sigma(o, u_n')$ and 
	$\sigma(y, u_n')$. Since the limit is independent of the subsequence $(x_n')$, the lemma follows.
\end{proof}

\subsection{Parabolic subgroups, Levi factors and  their fa\c cades}
\label{sec:14}
In this section we consider the Bruhat--Tits case. In particular, we are going to explain that each outer fa\c cade
$\mathcal{X}(F)$ is the Bruhat--Tits building associated with the semisimple quotient of the parabolic subgroup
stabilizing $F$. Picking an opposite facet $F'$ corresponds to picking a parabolic subgroup
${\rm Stab}(F')$ which is opposite ${\rm Stab}(F)$, {\it i.e.} ${\rm Stab}(F) \cap {\rm Stab}(F')$
is a Levi factor in both parabolic subgroups 

The descriptions of the Furstenberg, of the combinatorial and of the Martin compactifications at the bottom of the spectrum
make appear affine buildings at infinity. As expected, the algebraic group case, and its rich Lie-theoretic combinatorics,
is also the main source of inspiration for the constructions of the involved auxiliary affine buildings. 

Keeping the notation of the previous section, let again ${\mathcal A}({\bf S})$ be the apartment attached to the maximal
$k$-split torus ${\bf S}$. Let $F$ be a facet of the spherical building at infinity $\scrX^\infty$, which we assume to be
contained in the boundary ${\mathcal A}({\bf S})^\infty$ of ${\mathcal A}({\bf S})$. The stabilizer
${\bf P}_F(k) = {\rm Stab}_{{\bf G}(k)}(F)$ consists of the $k$-rational points of a parabolic $k$-subgroup ${\bf P}_F$
containing ${\bf S}$. The maximal $k$-split torus ${\bf S}$ provides a Levi decomposition
${\bf P}_F \simeq {\bf M}_{F,{\bf S}} \ltimes {\rm rad}_{\rm u}({\bf P}_F)$ where ${\bf M}_{F,{\bf S}}$ can be defined
algebraically as the centralizer of a suitable singular subtorus of ${\bf S}$, see \cite[Proposition 20.5]{BorelBook}, or
geometrically as the (Zariski closure of) the stabilizer in ${\bf G}(k)$ of $F$ and its opposite $-F = {\rm opp}_{\bf S}(F)$
with respect to the boundary ${\mathcal A}({\bf S})^\infty$. Note that ${\rm rad}_{\rm u}({\bf P}_F)(k)$ acts simply transitively
on the facets of $\scrX^\infty$ which are opposite $F$ or, equivalently, on Levi factors of ${\bf P}_F$. 

So far, we are dealing with the (unique) spherical building at infinity: this building comes from the combinatorics of the 
spherical Tits system constructed in Borel--Tits theory. It can be geometrically constructed as the boundary at infinity 
\cite[Section II.8]{Bridson1999} of the Gromov compactification of $\scrX$, which is the compactification obtained by the
horospherical process when using the family of usual distances to points for the Bruhat--Tits ${\rm CAT}(0)$-metric 
\cite[Section 11.2]{Abramenko2008}.

As for any compactification, thanks to the Cartan decomposition and its geometric interpretation (Section \ref{sec:4}), 
it is enough to describe the closure of a Weyl sector in order to describe the full space. In the case of the visual boundary,
the parameters of convergence are radial. The points in the boundary of the chosen Weyl sector are in bijective
correspondence with the unit vectors in the sector. 

The other family of parameters, which is useful to describe the three compactifications mentioned at the beginning of this
section, consists of the distances to the walls bounding the chosen Weyl sector. The choice of a spherical facet at infinity
$F$ as before can be seen as a way to select a subfamily of distances: this is explained in Section \ref{sec:3.3} when
introducing the fa\c cade corresponding to $F$ and the inner fa\c cade corresponding to $F$ and $-F$ as before. We want to give
an algebraic group interpretation of these objects. 

More precisely, let us go back to the Levi decomposition
${\bf P}_F \simeq {\bf M}_{F,{\bf S}} \ltimes {\rm rad}_{\rm u}({\bf P}_F)$ associated with the inclusion
$F \subset {\mathcal A}({\bf S})^\infty$. The group ${\bf M}_{F,{\bf S}}$ is a reductive $k$-group, therefore as such it admits
a Bruhat--Tits building, which we denote by $\scrX( {\bf M}_{F,{\bf S}},k)$. The connected center
${\bf Z}_{F,{\bf S}}$ of ${\bf M}_{F,{\bf S}}$ is a torus, the derived subgroup
${\bf G}_{F,{\bf S}} = [{\bf M}_{F,{\bf S}},{\bf M}_{F,{\bf S}}]$ is a semisimple $k$-group and the multiplication map
${\bf Z}_{F,{\bf S}} \times {\bf G}_{F,{\bf S}} \to {\bf M}_{F,{\bf S}}$ is an isogeny: the geometric counterpart of the latter
fact is that the building $\scrX( {\bf M}_{F,{\bf S}},k)$ admits a direct factor isometric to the Euclidean space
$X_*({\bf Z}_{F,{\bf S}}) \otimes_{\mathbb Z} \mathbb R$. In this case, the inner fa\c cade $\mathcal{X}(F,-F)$ associated with
the inclusion $-F \cup F \subset {\mathcal A}({\bf S})^\infty$ in Section \ref{sec:3.3} is nothing else than a natural
${\bf M}_{F,{\bf S}}(k)$-equivariant copy of (the geometric realization of) $\scrX( {\bf M}_{F,{\bf S}},k)$ inside
${\mathcal X}$; in the abstract group-theoretic (and more general) context of valuations of root group data, it had already been
constructed in \cite[Section 7.6]{BruhatTits1972}. As a set, the inner fa\c cade $\mathcal{X}(F,-F)$ is the union of
the apartments $g.{\mathcal A}({\bf S})$ when $g$ runs over ${\bf G}_{F,{\bf S}}(k)$, or equivalently over
${\bf M}_{F,{\bf S}}(k)$ since the elements of the singular torus ${\bf Z}_{F,{\bf S}}(k)$ act trivially on the Bruhat--Tits 
building of ${\bf G}_{F,{\bf S}}$. In fact, in the apartment ${\mathcal A}({\bf S})$ the elements of ${\bf Z}_{F,{\bf S}}(k)$
act as translations along directions which are parallel to the vector subspace $V_F$ of $V$ given by $F$: in particular, they
preserve the distances to any wall in the apartment ${\mathcal A}({\bf S})$ whose direction contains $V_F$, moreover unbounded
sequences of such semisimple matrices can be used to push to infinity a given vertex in ${\mathcal A}({\bf S})$. 

This is a suitable place to mention a group-theoretic interpretation of distances to walls. Let $x$ be a special vertex in
${\mathcal A}({\bf S})$ and let $\omega$ be a chamber at infinity lying in ${\mathcal A}({\bf S})^\infty$ and belonging to
the residue of $F$; the chamber $\omega$ defines a basis $( a_i : i \in I_0)$ of the root system of ${\bf G}$. The Weyl 
cone $[x,\omega]$ is simplicial and there is a bijective correspondence between the simple roots $( a_i : i \in I_0)$ and
the sector panels (codimension $1$ faces) of $[x,\omega]$: the affine subspace spanned by a sector panel is directed by
the kernel of a simple root and each simple root appears in this way. If we denote by $J$ the type of $F$, the subroot system
$\Phi_J$ with a basis $(a_i : i \in J)$ corresponds to the directions of the walls containing $V_F$. Concretely, the sector
$[x,\omega]$ is bounded by two kinds of sector panels: those whose direction is the kernel of a simple root in $\Phi_J$ and
the remaining ones. By construction each element in ${\bf S}(k)$ stabilizes ${\mathcal A}({\bf S})$ and acts as a translation
Section \ref{s:B1}. Among those translations, the semisimple matrices in ${\bf Z}_{F,{\bf S}}(k)$ act with a direction
parallel to the sector face $F$. More precisely, let $a_i$ be a simple root and let $\Pi_i$ be the corresponding sector panel. 
Recall from Section \ref{sec:4} that we have a decreasing filtration of ${\bf U}_{a_i}(k)$ by a countable family of
compact open subgroups indexed by affine linear forms of given vectorial part $a_i$. Then for a special vertex
$y \in [x,\omega]$, the (non-negative) difference between the affine root whose zero-set contains $\Pi_i$ (hence $x$) and
the affine root whose zero-set contains $y$ is a discrete version of the Bruhat--Tits distance for special vertices in
$[x,\omega]$. Recalling that for any $n \in {\bf N}(k)$ we have $n X_\alpha n^{-1} = X_{\alpha \circ \xi(n)}$, we see that
the elements $s \in {\bf Z}_{F,{\bf S}}(k)$ act as translations along directions which are parallel to the vector subspace
$V_F$, since any such $s$ centralizes each root group ${\bf U}_{a_j}(k)$ with $j \in J$. Figure \ref{fig:8} illustrates
the dynamics of semi-simple matrices along a sector face.

\begin{figure}[ht!]
   \includegraphics[width=20em]{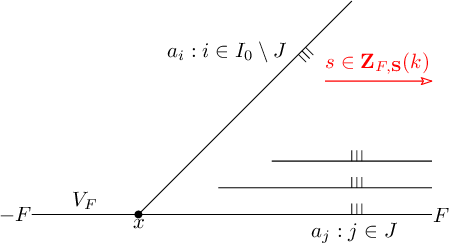}
   \caption{Dynamics of semisimple matrices along a sector face.}
   \label{fig:8}
\end{figure}

Going back to fa\c cades, we consider now the affine building ${\mathcal X}(F)$ associated with the facet $F$ only. 
Strictly speaking, this (essential) building is not contained in $\mathcal X$ but it is a stratum of the boundary of any of
the compactifications mentioned above. In the natural ${\bf G}(k)$-action on these compactifications, the subgroup
${\bf P}_F(k)$ stabilizes the stratum ${\mathcal X}(F)$, which can thus be identified with the Bruhat--Tits building of
the semisimple quotient of ${\bf P}_F$. This is made precise in the theorem below, which also summarizes many available
compactifications and contains the first half of Theorem \ref{th - alg gr} stated in Introduction.

\begin{theorem}
	\label{thm:BerkoMartin}
	Let ${\bf G}$ be a connected semisimple linear algebraic group defined over a non-Archimedean local field $k$ and let
	$\mathcal{X}({\bf G},k)$ be its Bruhat--Tits building. 
	\begin{enumerate}[label=(\roman*), ref=\roman*]
		\item \label{en:6:1}
		The maximal Berkovich compactification $\overline{\mathcal{X}}({\bf G},k)$ is ${\bf G}(k)$-equivariantly homeomorphic 
		to the wonderful compactification, the maximal Satake, as well as the polyhedral compactification of 
		$\mathcal{X}({\bf G},k)$. The closure of the set of vertices in $\overline{\mathcal{X}}({\bf G},k)$ is
		${\bf G}(k)$-equivariantly homeomorphic to the group-theoretic compactification of $\mathcal{X}({\bf G},k)$.
		\item \label{en:6:2}
		For any proper parabolic $k$-subgroup ${\bf P}$, there exists a natural, injective, continuous map
		\[
			\mathcal{X}({\bf P}/{\rm rad}({\bf P}),k) \to \overline{\mathcal{X}}({\bf G},k)
		\]
		whose image lies in the boundary. These maps altogether provide the following stratification
		\[
			\overline{\mathcal{X}}({\bf G},k) 
			= \bigsqcup_{\textrm{${\bf P}$ parabolic $k$-subgroup}} \mathcal{X}({\bf P}/{\rm rad}({\bf P}),k)
		\]
		where the union is indexed by the collection of all parabolic $k$-subgroups in ${\bf G}$.
		\item \label{en:6:3}
		Any two points $x, y$ in $\overline{\mathcal{X}}({\bf G},k)$ lie in a common compactified apartment 
		$\overline{{\mathcal A}}({\bf S})$ and we have
		\[
			{\bf G}(k) = {\rm Stab}_{{\bf G}(k)}(x) \, {\bf N}(k) \, {\rm Stab}_{{\bf G}(k)}(y)
		\]
		where ${\bf N}$ is the normalizer of the maximal split torus ${\bf S}$ defining the apartment ${\mathcal A}({\bf S})$.
	\end{enumerate}
\end{theorem}
\begin{proof}
	Berkovich compactifications of Bruhat--Tits buildings were defined in \cite{Berkovich} in the split case, and in 
	\cite{RTW10} in full generality. The idea is to embed $\mathcal{X}({\bf G},k)$ into the analytic space ${\bf G}^{\rm an}$
	attached to ${\bf G}$ and then into the (compact) analytic space $({\bf G}/{\bf P})^{\rm an}$ attached to a (projective) 
	flag variety of ${\bf G}$; when the parabolic $k$-subgroup ${\bf P}$ is minimal, the closure of the image of the composed
	map is maximal among Berkovich compactifications. The wonderful compactification is obtained by replacing a flag variety 
	${\bf G}/{\bf P}$ by the de Concini--Procesi wonderful compactification $\overline{\bf G}$ and then by using the resulting
	equivariant embedding from $\mathcal{X}({\bf G},k)$ to $\overline{\bf G}^{\rm an}$. It was defined in \cite{RTW17} in the 
	split case, and in \cite{Chanfi} in full generality. The polyhedral compactification in the Bruhat--Tits case was first 
	constructed in \cite{Landvogt}; it coincides with the compactification we presented in 
	Section \ref{ss - polyhedral compactification} \cite[\S 3.4.5.1]{Rousseau2023}. The Satake compactifications were defined 
	in \cite{Werner07} and then revisited in the context of non-Archimedean analytic geometry in \cite{RTW12}. 
	The identifications between Berkovich and Satake compactifications are given by \cite[Theorem 2.1]{RTW12}, while the 
	identification between the wonderful (resp. polyhedral) and the maximal Berkovich compactification is 
	\cite[Corollary 15]{Chanfi} (resp. is in \cite{Werner07} or \cite[Proposition 5.4]{RTW12}). Therefore the description of 
	the closure of vertices (in any of the previous compactifications) by means of the group-theoretic compactification is 
	the content of \cite[Theorem 20]{gure}. This provides \eqref{en:6:1} and once this is obtained, \eqref{en:6:2} and 
	\eqref{en:6:3} follow from  \cite[Theorem 2 of Introduction]{RTW10}. 
\end{proof}

\section{Gromov compactification}
\label{sec:5}
In this section we describe the Gromov compactification that can be constructed for any $\cat$ space $X$. For this purpose
we embed $X$ into the space of continuous functions on $X$. The Gromov boundary corresponds to the Busemann functions. 
The method was originally introduced in \cite[\S3]{Ballman1985} for complete Riemannian manifolds of non-positive
curvature, and it has been extended to $\cat$ spaces in \cite[Chapter II.8]{Bridson1999}. It is often
called the \emph{horofunction procedure}. In particular, it depends on the choice of the metric. In our construction we use 
an intrinsic and natural metric which is well-defined for special vertices and which avoids the use of any geometric realization 
of the buildings.

First, let us observe that 
\[
	d(x, y) = \norm{\sigma(x, y)}, \qquad \text{for } x, y \in V_s, 
\]
defines a metric on special vertices of $\scrX$. Since $\sigma(x, y) = -w_0.\sigma(y, x)$, it is enough to justify the triangle
inequality. Given $x, y, z \in V_s$, let $\omega \in \Omega(x, y)$. Then by \eqref{eq:35} and \eqref{eq:29}, we get
\begin{align}
	\nonumber
	\norm{\sigma(x, y)} 
	&=
	\norm{h(x, y; \omega)} \\
	\nonumber
	&\leqslant
	\norm{h(x, z; \omega)} + \norm{h(z, y; \omega)} \\
	&\leqslant
	\label{eq:61}
	\norm{\sigma(x, y)} + \norm{\sigma(z, y)}.
\end{align}
Analogously, one can show that
\begin{equation}
	\label{eq:69}
	\sprod{\sigma(x, y)}{\tilde{\rho}}
	\leqslant
	\sprod{\sigma(x, y)}{\tilde{\rho}} + \sprod{\sigma(z, y)}{\tilde{\rho}}.
\end{equation}

\begin{lemma}
	\label{lem:8}
	For all $x, y, z \in V_s$,
	\[
		|\sigma(x, y) - \sigma(x, z)| \leqslant |\sigma(y, z)|.
	\]
\end{lemma}
\begin{proof}
	Let $\omega \in \Omega(x, y)$, and let $\scrA$ be any apartment containing $[x, \omega)$. 
	Let $\psi: \scrA \rightarrow \Sigma$ be the isomorphism such that $\psi([x, \omega)) = S_0$. Suppose that $z \in V_s(\scrA)$.
	Let $w \in W$ be such that $z_0 = w.\psi(z) \in S_0$. Then
	\[
		\sigma(x, y) - \sigma(x, z) = \sigma(x, y) - \sigma(x, z_0) = \psi(y) - \psi(z_0).
	\]
	Let $w = r_{i_1} r_{i_2} \dots r_{i_k}$ be a minimal representation of $w$. For $j = 1, \ldots, k$, we set
	\[
		z_j = \psi^{-1}\big(r_{i_j} \psi(z_{j-1})\big).
	\]
	In particular, $\psi(y)$ and  $\psi(z_{j-1})$ are located on the same side of the wall $H_{\alpha_{i_j}; 0}$, which does not
	contain $\psi(z_j)$. Let $v_j$ be the unique point on the line segment between $\psi(y)$ and $\psi(z_j)$ which belongs to 
	$H_{\alpha_{i_j}; 0}$. Since
	\[
		|v_j - \psi(z_{j-1})| = |v_j - \psi(z_j)|
	\]
	we obtain
	\begin{align*}
		|\psi(y) - \psi(z_{j-1})| 
		&\leqslant |\psi(y) - v_j | + |v_j - \psi(z_{j-1})| \\
		&\leqslant |\psi(y) - v_j | + |v_j - \psi(z_j)| \\
		&=|\psi(y) - \psi(z_j)|.
	\end{align*}
	Consequently,
	\[
		|\psi(y) - \psi(z_0)| \leqslant |\psi(y) - \psi(z)|,
	\]
	which completes the proof in the case when $x, y, z \in V_s(\scrA)$.

	Next, assume that $z \notin V_s(\scrA)$. Let $c \in \St(x)$ be a chamber in $\scrA$ having all vertices in 
	$V \cap [x, \omega)$. Let $\rho_{c, \scrA}$ be the retraction of $\scrX$ onto $\scrA$ with center $c$, see 
	\cite[\S2.1.7]{Rousseau2023} for the definition. We set $z_0 = \rho_{c, \scrA}(z)$. Since
	\[
		\sigma(x, z) = \sigma(x, z_0),
	\]
	by the first part of the proof,
	\[
		|\sigma(x, y) - \sigma(x, z)| \leqslant |\sigma(y, z_0)|.
	\]
	Now, we invoke \cite[Theorem 2.1.8]{Rousseau2023} to get $|\sigma(y, z_0)| \leqslant |\sigma(y, z)|$,
	and the lemma follows.
\end{proof}
Now, let us denote by $\calC_*(V_s)$ the quotient space of real-valued $1$-Lipschitz functions on $V_s$ equipped with the 
topology of pointwise convergence by the $1$-dimensional subspace of constant functions. We introduce the action of the 
automorphisms group $\Aut(\scrX)$ on $\calC^*(V_s)$ by setting for $g \in \Aut(\scrX)$,
\[
	g . [f] = [g . f]
\]
where $[f]$ is the equivalence class represented by $f$, and
\[
	(g . f)(y) = f(g^{-1} . y), \qquad y \in V_s.
\]
Now, we define the embedding
\begin{alignat*}{1}
	\iota: V_s &\longrightarrow \calC_*(V_s) \\
	x  &\longmapsto [f_x]
\end{alignat*}
where 
\[
	f_x(y) = - d(y, x), \qquad y \in V_s.
\]
It is easy to check that the map $\iota$ is equivariant and injective, and $\iota(V_s)$ is discrete. Let us denote by
$\clr{\scrX}_V$ the closure of $\iota(V_s)$ in $\calC_*(V_s)$. The space $\calC_*(V_s)$ is metrizable, thus by
Lemma \ref{lem:3} while studying the Gromov boundary we can restrict attention to angular sequences.

The following theorem provides a connection between angular sequences and spherical buildings at infinity, as introduced and 
described in Section \ref{sec:12} (see also \cite[Section 3.2]{Rousseau2023}). To put the result below in perspective recall 
that in suitable (metric) cases, applying the horofunction procedure and taking asymptotic classes of geodesic rays leads to 
the same compactification \cite[Chapter II.8, Theorem 8.13 and Corollary 8.20]{Bridson1999}. In the specific case of affine 
buildings, we see that the convergence parameters for unbounded sequences are angular and compatible with the spherical 
building structure, formulated in terms of residues. 

\begin{theorem}
	\label{thm:11}
	Let $(x_n)$ be an angular $(\omega, \theta)$-sequence. Then
	\[
		\lim_{n \to \infty} [f_{x_n}(y)] = [\sprod{\theta}{h(o, y; \omega)}] \quad \textrm{for any}~y \in V_s.
	\]
	If $(x_n')$ is $(\omega', \theta')$-sequence such that $(f_{x_n'})$ converges to the same limit then
	$\theta = \theta'$ and $\omega' \sim_J \omega$ where $J = \{i \in I_0 : \sprod{\theta}{\alpha_i} = 0\}$.
\end{theorem}
\begin{proof}
	We write
	\begin{align}
		\nonumber
		f_{x_n}(y) - f_{x_n}(o)
		&=
		\frac{\norm{\sigma(o, x_n)}^2 - \norm{\sigma(y, x_n)}^2}{\norm{\sigma(y, x_n)} + \norm{\sigma(o, x_n)}} \\
		\label{eq:51}
		&=
		\frac{2\sprod{\sigma(o, x_n)}{\sigma(o, x_n) - \sigma(y, x_n)} - \norm{\sigma(o, x_n) - \sigma(y, x_n)}^2}
		{\norm{\sigma(y, x_n)} + \norm{\sigma(o, x_n)}}.
	\end{align}
	Since $\norm{\sigma(o, y)} = \norm{\sigma(y, o)}$, by Lemma \ref{lem:8} we have
	\begin{equation}
		\label{eq:42}
		\norm{\sigma(y, x_n) - \sigma(o, x_n)} \leqslant \norm{\sigma(o, y)}.
	\end{equation}
	Therefore, by \eqref{eq:51},
	\[
		\lim_{n \to \infty} f_{x_n}(y) - f_{x_n}(o) = \lim_{n \to \infty} 
		\frac{\sprod{\sigma(o, x_n)}{\sigma(o, x_n) - \sigma(y, x_n)}}{\norm{\sigma(o, x_n)}}.
	\]
	In view of \eqref{eq:42},
	\[
		\lim_{n \to \infty}
		\frac{\sprod{\sigma(o, x_n)}{\sigma(o, x_n) - \sigma(y, x_n)}}{\norm{\sigma(o, x_n)}}
		=
		\lim_{n \to \infty}
		\sprod{\theta}{\sigma(o, x_n) - \sigma(y, x_n)}.
	\]
	Since $(x_n)$ is an angular $(\omega, \theta)$-sequence, the inequality $\sprod{\theta}{\lambda_i} > 0$ implies that 
	\[
		\lim_{n \to \infty}{\sprod{\sigma(o, x_n)}{\alpha_i}} = +\infty,
	\]
	thus by Lemma \ref{lem:9},
	\[
		\lim_{n \to \infty} \sprod{\theta}{\sigma(o, x_n) - \sigma(y, x_n)}
		=
		\sprod{\theta}{h(o, y; \omega)}
	\]
	as claimed.

	Let us now turn to the proof of the second part of the theorem. Suppose that 
	\begin{equation}
		\label{eq:52}
		\sprod{\theta}{h(o, y; \omega)} = \sprod{\theta'}{h(o, y; \omega')},
		\qquad \text{for all } y \in V_s.
	\end{equation}
	Let $\scrA$ be an apartment containing $\omega$ and $\omega'$, and let $o'$ be any good vertex in $\scrA$. By the 
	cocycle relation and \eqref{eq:52}, we get
	\[
		\sprod{\theta}{h(o', y; \omega)} = \sprod{\theta'}{h(o', y; \omega')},
		\qquad\text{for all } y \in V_s.
	\]
	Hence, there is $w \in W$ such that 
	\begin{equation}
		\label{eq:53}
		\sprod{\theta}{\lambda} = \sprod{\theta'}{w^{-1} . \lambda}
		\qquad\text{for all } \lambda \in P.
	\end{equation}
	Let $k = \ell(w)$. We write $w = w_k = w_{k-1} r_{\beta_k}$ with $\ell(w_k) > \ell(w_{k-1})$ and $\beta_k \in \Phi^+$.
	Then
	\begin{align*}
		\sprod{\theta'}{w_k^{-1} . \lambda} 
		&= \sprod{r_{\beta_k} . \theta'}{w_{k-1}^{-1} . \lambda} \\
		&= \sprod{\theta'}{w_{k-1}^{-1} . \lambda} 
		- \sprod{\theta'}{\beta_k\spcheck} \sprod{\beta_k}{w_{k-1}^{-1} . \lambda}.
	\end{align*}
	Hence, arguing by induction we arrive at
	\begin{equation}
		\label{eq:63}
		\sprod{\theta'}{w^{-1} . \lambda} = \sprod{\theta'}{\lambda} - \sum_{j = 1}^k \sprod{\theta'}{\beta_j\spcheck}
		\sprod{\lambda}{w_{j-1} . \beta_j}
	\end{equation}
	where for each $j \in \{1, 2, \ldots, k\}$, we have $w_j = w_{j-1} r_{\beta_j}$ with $\ell(w_j) > \ell(w_{j-1})$ and
	$\beta_j \in \Phi^+$. By \cite[Proposition in Section 5.7]{Humphreys1990}, we have
	\begin{equation}
		\label{eq:64}
		w_{j-1} . \beta_j \in \Phi^+.
	\end{equation}
	Hence, by \eqref{eq:53} and \eqref{eq:63} we get
	\[
		\sprod{\theta}{\lambda} \leqslant \sprod{\theta'}{\lambda}.
	\]
	Since we can swap $\omega$ and $\theta$ with $\omega'$ and $\theta'$, respectively, we conclude that $\theta = \theta'$. Thus,
	by \eqref{eq:53} and \eqref{eq:63},
	\[
		\sum_{j = 1}^k \sprod{\theta}{\beta_j\spcheck} \sprod{\lambda}{w_{j-1} . \beta_j} = 0.
	\]
	Because $\theta \in \Ss^{r-1}_+$, by \eqref{eq:64}, $\sprod{\theta}{\beta_j\spcheck} = 0$ for all $j \in \{1, 2, \ldots, k\}$,
	and the theorem follows.
\end{proof}

Since both $\Ss^{r-1}$ and $\Omega$ are compact, in view of Theorem \ref{thm:11} the set $\clr{\scrX}_V$ is a compact
subset of $\calC_*(V_s)$. It is called the \emph{Gromov compactification} of $\scrX$. The action of the automorphism group 
$\Aut(\scrX)$ is continuous on $\clr{\scrX}_V$. 
At this stage, the maximal boundary $\Omega$, so far a mere compact topological space, can also be seen as the set of chambers of the spherical building at infinity of $\scrX$, which is also the boundary of the Gromov compactification of it. 

\section{Combinatorial compactifications}
\label{s - comb}
We use yet another construction described in \cite{Caprace2011} for a wide class of buildings not necessarily of affine type.
Since we are interested in compactifying special vertices only, we can make the approach more explicit. In particular, we show 
convergence of core sequences and identify when the limits are the same. Thanks to this we immediately conclude in Section 
\ref{s - Furst comp} that the combinatorial and Furstenberg compactifications are $\Aut(\scrX)$-equivariantly isomorphic.

For this purpose, let
\begin{equation}
	\label{eq:74}
	\Gamma(\scrX) 
	=\left\{
	\gamma: V(\scrX) \rightarrow \bigsqcup_{x \in V(\scrX)} \St(x) : \gamma(x) \in \St(x) \text{ for all }
	x \in V(\scrX)
	\right\}.
\end{equation}
Since $V(\scrX)$ is countable and $\St(x)$ is finite for each $x \in V(\scrX)$, by Tychonoff's theorem $\Gamma(\scrX)$
is a compact metrizable space. For an automorphism $g \in \Aut(\scrX)$ and $\gamma \in \Gamma(\scrX)$, we set
\[
	(g . \gamma)(x) = g . \gamma(g^{-1} . x), \qquad\text{for } x \in V(\scrX).
\]
Let us consider the map
\begin{alignat*}{1}
	\iota : V_s & \longrightarrow \Gamma(\scrX) \\
	x & \longmapsto \gamma_x
\end{alignat*}
where
\begin{equation}
	\label{eq:73}
	\gamma_x(y) = \bigcap_{c \in \calC(x)} \proj_{y}(c), \qquad y \in V(\scrX),
\end{equation}
whereas $\calC(x)$ is the set of chambers in $\St(x)$ and $\proj_y(c)$ is the unique chamber in $\St(y)$ that is
the closest to $c$. The map $\iota$ is equivariant, injective and has discrete image. The latter two properties are 
consequences of 
\begin{equation}
	\label{eq:72}
	\gamma_x(y) = \text{maximal simplex in $\St(y)$ with vertices in $[x, y]$}
\end{equation}
which leads to $\gamma_x(y) = x$ if and only $y = x$. In particular, the codimension of $\gamma_x(y)$ is the number of 
independent walls containing $x$ and $y$. For the proof of \eqref{eq:72} see \cite[Lemma 1.1]{Caprace2011}.
The closure of $\iota(V_s)$ in $\Gamma(\scrX)$ is called combinatorial compactification and it is denoted by
$\clr{\scrX}_C$. The map $\gamma_x$ can be interpreted as a discrete vector field. 
This analogy was already mentioned in \cite[Introduction]{Caprace2011} and we make it more concrete below. 
In the following theorem we study the corresponding boundary in terms of core sequences.
\begin{theorem}
	\label{thm:12}
	Suppose that $(x_n)$ is an $(\omega, J, c)$-core sequence. Then for every $y \in V(\scrX)$, the sequence
	$(\gamma_{x_n}(y) : n \in \NN_0)$ is eventually constant, and thus
	\[
		\gamma(y) = 
		\lim_{n \to \infty} \gamma_{x_n}(y) = \bigcup_{m = 1}^\infty \bigcap_{n = m}^\infty \gamma_{x_n}(y).
	\]
	If $(x_n')$ is an $(\omega', J', c')$-core sequence such that $(\gamma_{x_n'})$ converges to the same limit
	then $J' = J$, $\omega' \sim_J \omega$, and $c' = c$.
\end{theorem}
\begin{proof}
	Fix a core sequence $(x_n)$ with auxiliary sequence $(u_n)$. Let
	\[
		Q = \bigcup_{n = 1}^\infty [o, u_n].
	\]
	The set $Q$ is convex. Let $\scrA$ be any apartment containing $Q$. Let $\omega_1 \in \Omega$ be such that
	$Q \subset [o, \omega_1] \subset \scrA$. First, we show that for each $y \in V(\scrA)$, the sequence
	$(\gamma_{u_n}(y) : n \in \NN_0)$ is eventually constant. Indeed, there is $n_0$ such that every half-apartment
	in $\scrA$ containing $y$ with boundary parallel to $\alpha$-wall for $\alpha \in \Phi^+ \setminus \Phi_J$, contains $u_n$
	for all $n \geqslant n_0$. On the other hand if $\alpha \in \Phi_J$, then the half-apartment containing $y$ with the
	boundary parallel to $\alpha$-wall which passes through $u_n$, is independent of $n$. Hence, by \eqref{eq:72}, we easily
	conclude that $(\gamma_{u_n}(y))$ is eventually constant. In fact, if $x \in V(\scrA)$ is such that
	\begin{align*}
		\sprod{h(o, x; \omega_1)}{\alpha_j} &= \sprod{h(o, u_{n_0}; \omega_1)}{\alpha_j}, \qquad \text{for all } j \in J,
	\intertext{and}
		\sprod{h(o, x; \omega_1)}{\alpha_i} &\geqslant \sprod{h(o, u_{n_0};\omega_1 )}{\alpha_i}, 
		\qquad \text{for all } i \in I_0 \setminus J,
	\end{align*}
	then $\gamma_{u_n}(y) = \gamma_x(y)$ for all $n \geqslant n_0$.

	If $\omega_1'$ is the opposite to $\omega_1$ in $\scrA$, $\rho^{o, \omega_1'}_{\scrA}$ restricted to the set 
	\[
		\{x_n\} \cup \bigcap_{\alpha \in \Phi^-\setminus \Phi_J} L^\alpha_{u_n}
	\]
	where for $\alpha \in \Phi^-$ and a good vertex $x \in V(\scrA)$ by $L^\alpha_x$ we denote the half-apartment in $\scrA$
	with the boundary parallel to $\alpha$-wall passing through $x$, is an isomorphism onto its image. Thus the sequence 
	$(\gamma_{x_n}(y))$ is eventually constant.

	Next, let us consider $y \in V(\scrX)$ such that there is no apartment containing $y$ and $Q$. Let $(c_1, \ldots, c_k)$
	be a shortest gallery between $Q$ and $y$. Let $c_j$ be the last chamber in this gallery that can be
	put in one apartment with $Q$. Take $\scrA$ to be an apartment containing both $c_j$ and $Q$, and let $L^\alpha$ be
	the half-apartment in $\scrA$ containing $c_j$ whose boundary contains the panel shared between $c_j$ and $c_{j+1}$.
	Since $c_{j+1}$ cannot be put into the common apartment with $Q$, the boundary $\partial L^\alpha$ intersects $Q$.
	Let $L$ be the half-apartment in $\scrA$ with the boundary $\partial L^\alpha$ which contains $\omega_1$.
	Since $\rho^{o, \omega_1}_{\scrA}(y) \in V(\scrA)$, the sequence
	\[
		\Big(\gamma_{x_n}\big(\rho^{o, \omega_1}_{\scrA}(y)\big) : n \in \NN_0 \Big)
	\]
	is eventually constant. Moreover, $\rho^{o, \omega_1}_{\scrA}$ restricted to $\conv\{L, y\}$ is an isomorphism onto
	its image, thus the sequence $(\gamma_{x_n}(y))$ is eventually constant too. This completes the proof of the 
	first part of the theorem.

	Next, let us define $Q_x$ for $x \in V(\scrX)$ by the formula
	\begin{equation}
		Q_x = \bigcup_{m = 1}^\infty \bigcap_{n = m}^\infty [x, u_n].
		\label{Qdef}
	\end{equation}
	We have the following fact, which can be understood as the construction of a discrete geodesic flow. 
	\begin{claim}
		\label{clm:6}
		For each $x \in V(\scrX)$, $Q_x$ is the minimal set $Q' \subset V(\scrX)$ with the properties:
		\begin{enumerate}[label=(\roman*), ref=\roman*]
			\item
			\label{en:1}
			$x \in Q'$,
			\item
			\label{en:2}
			if a vertex $y \in Q'$ then $\gamma(y) \subset Q'$.
		\end{enumerate}
		Moreover, $Q_x$ consists of all vertices $y \in V(\scrX)$ such that there is a sequence of vertices
		$(v_0, v_1, \ldots, v_n)$, such that $v_0 = x$, $v_n = y$ and $v_{j+1} \in \gamma(v_j)$ for all
		$j \in \{0, 1, \ldots, n-1\}$.
	\end{claim}
	\noindent
	For the proof, let us denote by $\tilde{Q}_x$ the intersection over all subsets of $V(\scrX)$ satisfying \eqref{en:1}
	and \eqref{en:2}. Observe that $\tilde{Q}_x \subseteq Q_x$. To see this it is enough to show that $Q_x$ satisfies
	$\eqref{en:2}$. If $y \in Q_x$ then there is $m \geqslant 1$ such that
	$y \in [x, u_n]$ for all $n \geqslant m$, thus by \eqref{eq:72} we have $\gamma(y) \subset [x, u_n]$ for all $n \geqslant m$.
	Next, to show that $Q_x \subseteq \tilde{Q}_x$, we prove that for each $y \in Q_x$ there is a sequence of vertices
	$(v_0, \ldots, v_m)$ such that $v_0 = x$, $v_m = y$ and $v_{j+1} \in \gamma(v_j)$ for all $j \in
	\{0, 1, \ldots, m-1\}$. We proceed by induction on the length of a minimal gallery between $x$ and $y$. It is trivial
	if the length equals $1$. Suppose that it holds true for all vertices $x$ and $y$ at a minimal gallery between them 
	of length $k \geqslant 2$. Let $x'$ be any vertex in $[x, y]$ other than $x$ and $y$. Since $y \in Q_x$, there is $n_0$ 
	such that $y \in [x, u_n]$ for all $n \geqslant n_0$. Hence, $x' \in [x, u_n]$ for all $n \geqslant n_0$, and so 
	$x' \in Q_x$. Since the length of a minimal gallery between $x$ and $x'$ is smaller than $k$, by the inductive hypothesis, 
	there is a sequence of vertices $(v'_0, \ldots, v'_{m'})$ such that $v'_0 = x$, $v'_{m'} = x'$, and 
	$v'_{j+1} \in \gamma(v'_j)$ for all $j \in \{0, \ldots, m'-1\}$. Moreover, $y \in [x', u_n]$ for all $n \geqslant n_0$ thus 
	$y \in Q_{x'}$. Since the length of a minimal gallery between $x'$ and $y$ is smaller than $k$, by the inductive hypothesis, 
	there is a sequence of vertices $(v''_0, \ldots, v''_{m''})$ such that $v''_0 = x'$, $v''_{m''} = y$, and 
	$v''_{j+1} \in \gamma(v''_j)$ for all $j \in \{0, \ldots, m''-1\}$. Therefore, the desired sequence is 
	$(v'_0, \ldots, v'_{m'}, v''_1, \ldots, v''_{m''})$ which completes the proof of the claim.

	In particular, by Claim \ref{clm:6}, the set $Q_x$ depends only on a vertex $x$ and the map $\gamma$. To finish the proof 
	of the theorem, let us assume that there are two core sequence $(x_n)$ and $(x_n')$ such that
	\[
		\lim_{n \to \infty} \gamma_{x_n} = \lim_{n \to \infty} \gamma_{x_n'} = \gamma.
	\]
	Thanks to Claim \ref{clm:6}, 
	\[
		\bigcup_{n = 1}^\infty [o, u_n] = Q_o = \bigcup_{n = 1}^\infty [o, u_n'],
	\]
	which easily leads to $c = c'$, $J = J'$ and $\omega' \sim_J \omega$.
\end{proof}

\section{Harmonic measures}
\label{s - max boundary}
In this section we introduce harmonic measures which we use for Furstenberg compactifications: there is one such probability
measure on the maximal boundary for each special point in the building. Then we study big cells in the maximal boundary from
topological and measure-theoretic view points. 

\subsection{Construction of harmonic measures}
\label{ss - harm meas}
The following proposition is well-known (see e.g. \cite{park2,mz1}). It introduces the family of harmonic measures that are
used to define the Furstenberg compactification. The subtlety is that we cannot stand by group actions and integration on
homogeneous spaces to define these measures. 

\begin{proposition}
	\label{prop:2}
	For every special vertex $x \in V_s$, there is a unique Borel probability measure $\nu_x$ on $\Omega$ such that
	for all $y, y' \in V_s$, if $\sigma(x, y) = \sigma(x, y')$ then
	\begin{equation}
		\label{eq:75}
		\nu_x\big(\Omega(x, y) \big) = \nu_x\big(\Omega(x, y')\big).
	\end{equation}
	Moreover, for $x, y \in V_s$ the measures $\nu_x$ and $\nu_y$ are mutually absolutely continuous. 
	When $x, y \in V_g$
	then the Radon--Nikodym derivative equals
	\begin{equation}
		\label{eq:12}
		\frac{{\rm d} \nu_y}{{\rm d} \nu_x} (\omega) = \chi\big(h(x, y; \omega)\big),
		\qquad \omega \in \Omega.
	\end{equation}
\end{proposition}

The proof below gives in fact all Radon--Nikodym derivatives, but the formulation uses notions from
Section \ref{sec:8} such as the permutation $\varepsilon$. 

\begin{proof}
	Let us consider a map $\Lambda_x$ defined on the set of characteristic functions
	$\big\{\ind{\Omega(x, y)} : y \in V_s\big\}$ by the formula
	\begin{equation}
		\label{eq:76}
		\Lambda_x\big(\ind{\Omega(x, y)}\big) = \frac{1}{\#\{y' \in V_s : \sigma(x, y') = \sigma(x, y)\}}.
	\end{equation}
	Recall that for every $z \in V_s$, such that $\sigma(x, z) = \sigma(x, y) + \sigma(y, z)$, we have
	\[
		\# \{ z \in V_s : \sigma(x, z) = \sigma(x, z') \}
		=
		\# Z' \cdot \# \{y' \in V_s : \sigma(x, y') = \sigma(x, y)\}
	\]
	where $Z' = \{z' \in V_s : \sigma(x, z') = \sigma(x, z), \sigma(y, z') = \sigma(y, z) \}$. Since
	\[
		\Omega(x, y) = \bigsqcup_{z' \in Z'} \Omega(x, z'),
	\]
	we obtain
	\[
		\Lambda_x\big(\ind{\Omega(x, y)}\big) = \sum_{z' \in Z} \Lambda_x\big(\ind{\Omega(x, z')}\big).
	\]
	Consequently, $\Lambda_x$ extends to the linear operator acting on locally constant functions on $\Omega$.
	Since locally constant functions on $\Omega$ are dense in the space of continuous functions on $\Omega$, the operator 
	has a unique extension to a positive linear operator $\tilde{\Lambda}_x$
	defined on all continuous complex-valued functions on $\Omega$. Hence, by the Riesz--Markov--Kakutani theorem, there is a 
	unique regular Borel measure $\nu_x$ on $\Omega$ such that
	\[
		\tilde{\Lambda}_x(f) = \int_\Omega f(\omega) \nu_x ({\rm d} \omega)
	\]
	for any continuous function $f: \Omega \rightarrow \CC$. 

	To prove \eqref{eq:12}, we fix $x, y \in V_g$. Let $\omega \in \Omega$. By \cite[Lemma 3.13]{park2}, there is
	a good vertex $z \in [x, \omega] \cap [y, \omega]$, such that $\Omega(x, z) = \Omega(y, z)$. We may assume that
	$h(x, z; \omega)$ and $h(y, z; \omega)$ are strongly dominant co-weights. Then, by \eqref{eq:8} and \eqref{eq:15},
	\[
		\nu_x\big(\Omega(x, z)\big) = 
		\frac{1}{W(q^{-1})}
		\chi\big(-h(x, z; \omega)\big) = 
		\frac{1}{W(q^{-1})} \chi\big(-\sigma(x,z)\big)
	\]
	and
	\begin{align*}
		\nu_y\big(\Omega(x, z)\big) 
		&=\nu_y\big(\Omega(y, z)\big) \\
		&= \frac{1}{W(q^{-1})}
		\chi\big(-h(y, z; \omega)\big) = 
		\frac{1}{W(q^{-1})} \chi\big(-\sigma(y,z)\big).
	\end{align*}
	Hence,
	\begin{align*}
		\frac{\nu_y\big(\Omega(x, z)\big)}{\nu_x\big(\Omega(x, z)\big)}
		&=
		\chi\big(-h(y, z; \omega)\big) \chi\big(h(x, z; \omega)\big) \\
		&=
		\chi\big(h(x, y; \omega)\big)
	\end{align*}
	where we have used the cocycle relation \eqref{eq:35}. 
	
	If $x \in V_g$ and $z \in V_g^\varepsilon$, we 
	set
	\[
		Z' = \big\{z' \in V_g : \exists \omega \in \Omega(x, z) \text{ such that } h(z,z';\omega) = \tfrac{1}{2} \lambda_r 
		\big\}.
	\]
	Let $c$ be the unique chamber in $\St(z)$ which is the closest to $x$. Let $w_0$ and $w_{0 r}$ be the longest element
	in $W$ and $W_{\lambda_r}$, respectively. Each vertex $z' \in Z'$ belongs to a certain chamber from $\St(z)$ which
	is opposite to $c$. There are $q_{\varepsilon; w_0}$ chambers with this property. However, for a fixed $z' \in Z'$, 
	there are $q_{\varepsilon; w_{0r}}$ distinct chambers sharing vertices $z'$ and $z$ that are opposite to $c$. Hence,
	\[
		\# Z' = \frac{q_{\varepsilon; w_0}}{q_{\varepsilon; w_{0 r}}}.
	\]
	Since
	\[
		\Omega(x, z) = \bigsqcup_{z' \in Z'} \Omega(x, z'),
	\]
	we get
	\begin{align*}
		\nu_x(\Omega(x, z)) &= \sum_{z' \in Z'} \nu_x(\Omega(x, z')) \\
		&= 
		\frac{q_{\varepsilon; w_0}}{q_{\varepsilon; w_{0 r}}}
		\frac{1}{W(q^{-1})}
		\chi\big(-\tfrac{1}{2} \lambda_r \big) \chi(-h(x, z; \omega)).
	\end{align*}
	Therefore, for $y \in V_g^\varepsilon$ and $x \in V_g$ we obtain
	\[
		\frac{\nu_y\big(\Omega(x, z)\big)}{\nu_x\big(\Omega(x, z)\big)}
		=
		\frac{W(q^{-1})}{W(q_\varepsilon^{-1})} \frac{q_{\varepsilon; w_{0r}}}{q_{\varepsilon; w_{0}}}
		\chi\big(\tfrac{1}{2} \lambda_r)
		\chi\big(h(x, y; \omega)\big).
	\]
	This completes the proof.
\end{proof}

\begin{remark}
\label{rem:5}
	We can weaken the hypothesis in Proposition \ref{prop:2}, by imposing \eqref{eq:75} for all $y \in V_g$ having
    fixed type, say $\tau(y) = j \in I$. To see this, let $z \in V_g$ with $\tau(z) \neq j$. There is $y_0 \in V_g$ with 
    $\tau(y_0) = j$, such that
    \[
        \Omega(x, y_0) \subset \Omega(x, z).
    \]
    Fix $\omega_0 \in \Omega(x, y_0)$. Then
    \[
        Y = \big\{y \in V_g : \exists \omega \in \Omega(x, z) \text{ such that } h(x, y;\omega) = h(x, y_0; \omega_0) \big\}.
    \]
    Observe that each vertex $y \in Y$ has type $j$, and
    \[
        \Omega(x, z) = \bigsqcup_{y \in Y} \Omega(x, y).
    \]
    Now, the linear operator $\Lambda$ defined in \eqref{eq:76} for $y \in V_g$, $\tau(y) = j$, can be uniquely extended to
    $\ind{\Omega(x, z)}$, for any $z \in V_g$. The rest of the poof is unchanged.
\end{remark} 

The measures $(\nu_x : x \in V_g)$ naturally appear while studying harmonic functions with respect to
vertex averaging operators, see Section \ref{ss - Martin embeddings}. To be more precise, for $\lambda \in P^+$, 
and a function $F: V_g \rightarrow \CC$, we set
\begin{equation}
	\label{eq:66}
	A_\lambda F(x) = \frac{1}{N_\lambda} \sum_{y \in V_\lambda(x)} F(y), \qquad x \in V_g.
\end{equation}
Then for any $f \in L^\infty(\Omega)$, the function $F: V_g \rightarrow \CC$ defined by 
\[
	F(x) = \int_\Omega f(\omega) \nu_x({\rm d} \omega), \qquad x \in V_g,
\]
satisfies
\[
	A_\lambda F(x) = F(x),
\]
for all $x \in V_g$ and $\lambda \in P^+$. We have a similar characterization for vertices in $V_g^\varepsilon$.

\subsection{Disintegration of harmonic measures}
Let us fix $J \subset I_0$. We denote by $\Omega^J$ the set of all spherical residues of type $J$ and thus we have a quotient 
map $\pi^J : \Omega \twoheadrightarrow \Omega^J$ defined by $\omega \mapsto \res_J(\omega)$. To each $J$-residue $R$ is 
associated a facet $F$ of type $J$ in the spherical building at infinity: the intersection of all closed chambers at infinity
in $R$ is equal to $F$.  Moreover, according to Section \ref{sec:3.3}, to each $J$-residue $R$ we can
associate an affine building $\scrX_R = \scrX(F)$ and we have a homomorphism of buildings
\[
	\pi_R: \scrX \rightarrow \scrX_R, 
\]
together with a homeomorphism
\[
	\vphi_R: R \rightarrow \Omega_R
\]
identifying the residue $R$ with the maximal boundary $\Omega_R = \Omega_{\scrX(F)}$ of the outer fa\c cade $\scrX_R$
(Lemma \ref{lem:1}). 

Now pick in addition $x, y \in V_g$ and let $\omega \in \Omega(x,y)$. If $y'$ is another vertex, we say that $y'$ is 
\emph{$J$-related to $y$} (with respect to $R$) if there is an apartment containing $[x, \omega]$ and $y'$ such that $y'$
belongs to the orbit of the subgroup of the spherical Weyl group fixing the sector face of type $J$ contained in $[x, \omega]$
(this subgroup is also the subgroup fixing the facet $F$ in the corresponding apartment at infinity). We have then
$\sigma(x,y')=\sigma(x,y)$. We denote by $[y]_J$ the set of vertices which are $J$-related to $y$ with respect to $R$. 
The projection $\pi_R |_{[y]_J}$ provides a bijection between $[y]_J$ and the vectorial sphere in $\scrX_R$ centered at
$\pi_R(x)$ and of radius $P_J\sigma(x,y)$; we denote by $N^J_{P_J \sigma(x, y)}$ the common cardinality of these two sets. 

We introduce on $\Omega^J$ a probability measure $\eta^J_x$ by setting
\begin{equation}
	\label{eq:37}
	\eta^J_x\big(\Omega(x, y)^J \big) = \frac{N^J_{P_J \sigma(x, y)}}{N_{\sigma(x, y)}}
\end{equation}
where
\[
	\Omega(x, y)^J = \Big\{R \in \Omega^J : \Omega(x, y) \cap R \neq \varnothing \Big\}. 
\]
Analogously to the proof of Proposition \ref{prop:2} one can show that the measures $\eta^J_x \in \mathcal{P}(\Omega^J)$ are
well-defined for each $J \subset I_0$ and each $x \in V_g$.  Having in mind the partition 
\[
	\Omega = \bigsqcup_{R \in \Omega^J} R = \bigsqcup_{R \in \Omega^J} \varphi_R^{-1}(\Omega_R), 
\]
we want to disintegrate each harmonic measure $\nu_x \in \mathcal{P}(\Omega)$ by integrating first on each
$J$-residue -- seen as the maximal boundary of the corresponding outer fa\c cade -- against a suitable harmonic measure of 
the fa\c cade and by using the measure $\eta^J_x$ for the second integration step. 

\begin{proposition}
\label{prop - disintegration}
	With the above notation, we have the equality of measures $\eta^J_x = (\pi^J)_*\nu_x$ in $\mathcal{P}(\Omega^J)$. 
	Moreover for each Borel set $A \subset \Omega$, we have the disintegration formula 
	\begin{equation}
		\label{eq:19}
		\nu_x(A) = \int_{A^J} \nu_{\pi_R(x)}\big(\vphi_R(R \cap A)\big) \: \eta_x^J({\rm d} R)
	\end{equation}
	where $\nu_{\pi_R(x)}$ is the harmonic measure on the maximal boundary $\Omega_R = \vphi_R(R)$ of the outer fa\c cade
	$\scrX_R$ attached to the vertex $\pi_R(x)$. 
\end{proposition}
\begin{proof} 
	In order to show the equality $\eta^J_x = (\pi^J)_*\nu_x$ in $\mathcal{P}(\Omega^J)$, it is enough to prove that for any 
	$y \in V_g$ we have 
	\begin{equation}
		\label{eq:24}
		((\pi^J)_*\nu_x)\big(\Omega(x, y)^J \big) = \frac{N^J_{P_J \sigma(x, y)}}{N_{\sigma(x, y)}},
	\end{equation}
	that is 
	\[
		\nu_x\big((\pi^J)^{-1}(\Omega(x, y)^J)\big) = \frac{N^J_{P_J \sigma(x, y)}}{N_{\sigma(x, y)}}.
	\]
	We first note that 
	\[
		\Omega(x, y)^J 
		= \big\{R \in \Omega^J : \Omega(x, y) \cap R \neq \varnothing \big\}
		= \big\{\res_J(\omega) : \omega \in \Omega(x, y) \big\}, 
	\]
	which implies that 
	\[
		(\pi^J)^{-1}(\Omega(x, y)^J) = \bigcup_{\omega \in \Omega(x,y)} \res_J(\omega). 
	\]
	We claim that
	\begin{equation}
		\label{eq:2}
		(\pi^J)^{-1}(\Omega(x, y)^J) = \bigsqcup_{y' \in [y]_J} \Omega(x,y').
	\end{equation}
	Since $\# [y]_J = N^J_{P_J \sigma(x, y)}$ and each shadow $\Omega(x, y')$ with $y' \in [y]_J$ has the same $\nu_x$-mass,
	the equality \eqref{eq:2} immediately leads to \eqref{eq:24}. To prove \eqref{eq:2}, we notice that if 
	$\omega' \in \res_J(\omega)$ with $\omega \in \Omega(x,y)$, then the sectors $[x,\omega]$ and $[x,\omega']$ share their 
	vectorial face of type $J$. Hence, the stabilizer in the Weyl group of that vectorial face, $W_J$ say, sends $y$ to a vertex 
	$y'$ contained in $[x,\omega']$ and thus $y' \in [y]_J$. Conversely, if $\omega'$ belongs to a shadow $\Omega(x,y')$ with
	$y' \in [y]_J$, then we can pick an apartment containing $[x,\omega']$ and $y$. In this apartment $W_J$ sends the latter
	sector to a $J$-adjacent one which contains $y$, so that denoting this sector by $[x,\omega]$ we see that
	$\omega' \in \res_J(\omega)$ with $\omega \in \Omega(x,y)$, as desired. 

	We turn now to the proof of the disintegration formula. Let us consider the Borel measure $\mu$ on $\Omega$ defined by 
	\[
		\mu\big(A\big) 
		=
		\int_{A^J} \nu_{\pi_R(x)}\big( \vphi_R(R \cap A) \big) \: \eta_x^J({\rm d} R)
	\]
	where $ A^J = \big\{R \in \Omega^J : R \cap A \neq \varnothing \big\}$. Since for each $x, y \in V_g$ and
	$R \in \Omega(x, y)^J$, by Lemma \ref{lem:6} we have 
	\[
		\sigma_R\big(\pi_R(x), \pi_R(y)\big) = P_J \sigma(x, y),
	\]
	thus
	\begin{align*}
		\mu(\Omega(x,y)) 
		&= 
		\int_{\Omega(x, y)^J} \nu_{\pi_R(x)}\Big(\Omega_R\big(\pi_R(x), \pi_R(y)\big) \Big) \eta^J_x({\rm d} R) \\
		&=
		\int_{\Omega(x, y)^J} \Big(N_{\sigma_R(\pi_R(x), \pi_R(y))}^J\Big)^{-1} \, \eta^J_x({\rm d} R)  \\
		&=
		\big(N^J_{P_J \sigma(x, y)}\big)^{-1} \int_{\Omega(x, y)^J} \eta^J_x({\rm d} R) \\
		&=
		\big(N_{\sigma(x, y)}\big)^{-1} \\
		&= 
		\nu_x(\Omega(x,y)). 
	\end{align*}
	Finally, by Proposition \ref{prop:2} characterizing harmonic measures, we have $\mu = \nu_x$. 
\end{proof} 

\subsection{Big cells}
In the study of convergence of unbounded sequences of harmonic measures, we will make use of analogues of contraction arguments 
in the case of symmetric spaces and Bruhat--Tits buildings. The subspaces of $\Omega$ on which the arguments are valid 
are analogues of the so-called big cells in flag varieties. The following fact saying that
big cells are co-null in maximal boundaries, has already been proved in \cite[Proposition 2.13]{HHL}; 
however our proof differs from the one given in [loc. cit.] in the sense that it stays inside the building,
without using points at infinity.

\begin{theorem}
	\label{thm:1}
	For each $\omega_0 \in \Omega$ and $x \in V_s$, we have
	$\nu_x\big(\Omega'(\omega_0)\big) = 1$. 
\end{theorem}
\begin{proof} 
	Without loss of generality we assume that $x \in V_g$. Given
	$y \in V_g$, we denote by $\Omega_y(\omega_0)$ the set of all $\omega \in \Omega$ such that there is an apartment $\scrA$
	containing both sectors $[y, \omega_0]$ and $[y, \omega]$. Then
	\[
		\Omega = \bigcup_{y \in V_g} \Omega_y(\omega_0),
	\]
	and so
	\begin{equation}
		\label{eq:6}
		\Omega \setminus \Omega'(\omega_0)
		= 
		\bigcup_{y \in V_g} \Omega_y(\omega_0) \setminus \Omega'_y(\omega_0)
	\end{equation}
	where we have set $\Omega'_y(\omega_0) = \Omega'(\omega_0) \cap \Omega_y(\omega_0)$. Since the set of vertices is countable, 
	it is therefore enough to show that for all $y \in V_g$,
	\begin{equation}
		\label{eq:36}
		\nu_x\big(\Omega_y(\omega_0) \setminus \Omega'_y(\omega_0) \big) = 0.
	\end{equation}
	Since the measures $\nu_x$ are mutually absolutely continuous, it is sufficient to consider $x = y = o$. 
	
	Next, let us observe that
	\begin{equation}
		\label{eq:18}
		\Omega_o(\omega_0) \setminus \Omega'_o(\omega_0) =
		\bigcap_{n = 0}^\infty \bigcup_{\stackrel{w \in W}{w \neq w_0}} \bigcup_{y \in B^{n\rho}_w} \Omega(o, y).
	\end{equation}
	where for $w \in W$ and $\lambda \in P^+$, by $B^\lambda_w$ we denote the set of all $x \in V_g$ such that
	$h(o, x; \omega_0) = w . \lambda$ and there is an apartment containing $[o, \omega_0]$ and $x$. 
	It is easy to verify the inclusion $\subseteq$. To show the reverse one, let us consider $\omega$ belonging to 
	the right-hand side of \eqref{eq:18}. Then there are sequences $(y_n : n \in \NN)$ and $(w_n : n \in \NN)$ such that
	for each $n \in \NN$, we have $w_n \in W \setminus \{w_0 \}$, $y_n \in B^{n\rho}_{w_n}$ and $\omega \in \Omega(o, y_n)$.
	Therefore,
	\[
		\omega \in \bigcap_{n = 1}^\infty \Omega(o, y_n).
	\]
	In particular, $[o, y_n] \subset [o, y_{n+1}] \subset [o, \omega]$. 
	Now, let $\Lambda_n$ be the set of all sectors $[o, \omega']$, with $\omega' \in \Omega_o'(\omega_0)$, such that
	the apartment $[\omega_0, \omega']$ contains $y_n$. Since $y_n \in B^{n\rho}_{w_n}$, each set $\Lambda_n$ is non-empty. 
	Moreover, $\Lambda_{n+1} \subset \Lambda_n$ because the convex hull of $o$ and $y_{n+1}$ contains $y_n$.
	Hence, by compactness of $\Omega$, there is $\omega'_0$ such that the apartment $[\omega_0, \omega_0']$ contains
	the convex hull of $o$ and $y_n$ for all $n \in \NN$, thus it contains $[o, \omega]$. This proves \eqref{eq:18}.

	Notice that for a fixed $n \in \NN$, the sets $\Omega(o, y)$ are disjoint provided that
	$\sigma(o, y) = n\rho$, thus by \eqref{eq:18}
	\begin{align*}
		\nu_o\big(\Omega_o(\omega_0) \setminus \Omega'_o(\omega_0)\big)
		&=
		\lim_{n \to \infty}
		\nu_o\Big(\bigcup_{\stackrel{w \in W}{w \neq w_0}} \bigcup_{y \in B^{n\rho}_w} \Omega(o, y)\Big) \\
		&=
		\lim_{n \to \infty}
		\sum_{\stackrel{w \in W}{w \neq w_0}} \sum_{y \in B^{n\rho}_w} \nu_o\big(\Omega(o, y)\big).
	\end{align*}
	Therefore to finish the proof of the theorem we show the following claim.
	\begin{claim}
		\label{clm:2}
		For each $w \in W$, $w \neq w_0$, we have
		\[
			\lim_{n \to \infty} \frac{\# B^{n\rho}_w}{N_{n\rho}} = 0.
		\]
	\end{claim}
	\noindent
	For the proof, let us notice that for a given $y \in B^{(n+1)\rho}_w$ there is exactly one $x \in B^{n\rho}_w$ such 
	there is an apartment containing $[o, \omega_0]$, $x$ and $y$. Therefore, the problem is reduced to estimating
	the number of $y$'s that corresponds to a given $x$. To do so, we observe that there is a wall passing through $x$ so 
	that $o$ and $\omega_0$ are on its opposite sides. Let $\beta \in \Phi^+$ be the corresponding root. Since a convex
	hull of $[o, \omega_0]$ and $x$ is contained in the intersection over all half-apartments having $x$ on
	the boundary and containing the sector $[o, \omega_0]$, the intersection of the link of $x$ with the convex hull
	contains at least two chambers that are $\beta$ adjacent. Let $c_0$ and $c_1$ be the chambers with a vertex $x$ that
	are the closest to $o$ and $\omega_0$, respectively. Each minimal gallery between $x$ and $y$ starts with a certain
	chamber $c'$ having a vertex $x$ and which is opposite to $c_0$. All minimal galleries have the same type $f$.
	For a given $y$ the chamber $c'$ is unique. Moreover, for each $c'$ there is the minimal gallery between $c_0$ and
	$c'$ which contains $c_1$. Therefore, there are at most $q_\beta^{-1} q_{w_0}$ possible choices for $c'$, and hence
	there are at most $q_\beta^{-1} q_{w_0} q_{w_f}$ vertices $y \in B^{(n+1)\rho}_w$ such that the apartment containing 
	$y$ and $[o, \omega_0]$ also contains $x$. Consequently,
	\[
		\#B^{(n+1)\rho}_w \leqslant \# B^{n\rho}_w \cdot q_\beta^{-1} \chi(\rho).
	\]
	In view of \eqref{eq:8}, we get
	\[
		\frac{\#B^{(n+1)\rho}_w}{N_{(n+1)\rho}} \leqslant q^{-1}_\beta \frac{\#B^{n\rho}_w}{N_{n\rho}},
	\]
	which complete the proof of the claim.

	Now, using Claim \ref{clm:2} and Proposition \ref{prop:2}, we get
	\[
		\nu_o\big(\Omega_o(\omega_0) \setminus \Omega'_o(\omega_0)\big)
		=
		\sum_{\stackrel{w \in W}{w \neq w_0}} \lim_{n \to \infty} \frac{\#B^{n\rho}_w}{N_{n\rho}} = 0
	\]
	which proves \eqref{eq:36}, and the theorem follow.
\end{proof}

\begin{corollary}
	For each $\omega_0 \in \Omega$, the big cell $\Omega'(\omega_0)$ is dense in $\Omega$.
\end{corollary}
\begin{proof}
	Indeed, otherwise there are $\omega' \in \Omega \setminus \Omega'(\omega_0)$ and its neighborhood $\Omega(x, y)$
	for certain $x \neq y$, $x, y \in V_g$, such that $\Omega(x, y) \subset \Omega \setminus \Omega'(\omega_0)$. 
	Hence, by Theorem \ref{thm:1}
	\[
		0 < \nu_x\big(\Omega(x, y)\big) \leqslant \nu_x\big(\Omega \setminus \Omega'(\omega_0) \big) = 0
	\]
	which leads to contraction.
\end{proof}

\section{Furstenberg compactification}
\label{s - Furst comp} 
In this section we describe Furstenberg compactifications. It originates in the study of harmonic functions
on semisimple Lie groups, see \cite{Furstenberg1963}. For Bruhat--Tits buildings associated with semisimple groups
over local fields, it has been constructed in \cite{gure}. Our approach is purely geometric and provides the maximal
Furstenberg compactification for a large class of affine buildings including exotic ones.

Let $\calP(\Omega)$ be the space of Borel probability measures on $\Omega$ endowed with the weak-$*$ topology. 
An automorphism $g \in \Aut(\scrX)$ acts on a measure $\nu \in \calP(\Omega)$ by pushing forward, that is 
\begin{align*}
	(g . \nu)(A) &= (g_* \nu)(A) \\
	&= \nu\big(g^{-1} . A\big)
\end{align*}
for any Borel set $A$. 

We define
\begin{alignat*}{1}
	\iota: V_s & \longrightarrow \calP(\Omega) \\
	x & \longmapsto \nu_x
\end{alignat*}
where $\nu_x$ is the harmonic measure defined in Proposition \ref{prop:2}. The map $\iota$ gives an equivariant 
embedding of $V_s$ into the space of Borel probability measures on $\Omega$. Moreover, $\iota(V_s)$ is discrete in
$\calP(\Omega)$ since there is $\epsilon > 0$ such that the set
\[
	\bigcap_{j = 1}^r \bigcap_{y_j \in V_{\lambda_j}(x)} \left\{\nu \in \calP(\Omega) : 
	\nu(\Omega(x, y_j)) > (1-\epsilon) \nu_x(\Omega(x, y_j))\right\}
\]
is open and its intersection with $\iota(V_s)$ gives the singleton $\nu_x$. 
Let $\clr{\scrX}_{F}$ be the closure of $\iota(V_s)$ in $\calP(\Omega)$. Then $\clr{\scrX}_{F}$
equipped with the induced topology is a compact Hausdorff space called the Furstenberg compactification of $\scrX$.
Let us describe the structure of $\clr{\scrX}_{F}$. Since $\calP(\Omega)$ is metrizable, it is sufficient to consider sequences
approaching infinity $(x_n)$ such that $(\iota(x_n))$ converges in $\calP(\Omega)$. Furthermore, in view of Lemma \ref{lem:2},
we restrict our attention to core sequences.

\subsection{Degenerations of harmonic measures}
\label{ss - limit harm}
In order to state our main theorem on convergence of sequences of harmonic measures (which is Theorem \ref{th - harm conv}
in Introduction), it is convenient to use the notions and terminology introduced in Section \ref{sec:6}, including
fa\c cades indexed by spherical facets, {\it i.e.} residues at infinity. We use freely the notation of Section \ref{sec:3.3}.
Given $R$ a $J$-residue in the building at infinity we define the map
\[
	\begin{alignedat}{1}
		\phi_R: \Omega &\longrightarrow \Omega_R \\
		\omega &\longmapsto \vphi_R\big(\proj_R (\omega)\big). 
	\end{alignedat}
\]
where $\proj_R(\omega)$ is the unique chamber closest to $\omega$ in $R$.
\begin{theorem}
	\label{thm:5}
	Suppose that $(x_n)$ is an $(\omega, J, c)$-core sequence. Let $R$ be the $J$-residue in $\scrX^\infty$ containing
	$\omega$ and let $F$ be the corresponding spherical facet. We denote by $\scrX(F)$ the fa\c cade associated with $F$,
	and by $x_F$ the vertex in $\scrX(F)$ defined by $(x_n)$. Then the sequence of harmonic measures
	$(\nu_{x_n})$ weakly converges in $\calP(\Omega)$ to the measure $\mu$ characterized by the following two conditions: 
	\begin{enumerate}[label=(\roman*), ref=\roman*]
	\item \label{en:5:1}
	the support ${\rm supp}(\mu)$ is equal to the residue $\res_J(\omega)$;
	\item \label{en:5:2} 
	the measure $(\phi_{R})_*\mu$ is the harmonic measure on the maximal boundary of $\scrX(F)$ attached to the vertex
	$x_F$. 
	\end{enumerate} 
	Moreover, if $(x_n')$ is an $(\omega', J', c')$-core sequence such that $(\nu_{x_n'})$ weakly converges to the same limit
	$\mu$ as above, then we have $J' = J$, $\omega' \sim_J \omega$, and $c = c'$.
\end{theorem}

\begin{proof}
	The proof uses the weak-$*$ compactness of the set $\calP(\Omega)$ of probability measures on the compact metrizable maximal
	boundary $\Omega$. Convergence is obtained by proving uniqueness of the cluster value in two steps. The first one shows that
	any cluster value $\mu$ of a sequence of harmonic measures as above has its support contained in the residue $R$, which
	allows us then to push $\mu$ by the homeomorphism $\varphi_{R}$ in order to prove, in the second step, that
	$(\phi_{R})_*\mu$ must be the announced harmonic measure on the  fa\c cade $\scrX(F)$. 

	Without loss of generality we assume that $(x_n) \subset V_g$ (see Section \ref{sec:8}). Let $\mu$ be any 
	cluster point of the sequence $(\nu_{x_n})$. By selecting a further subsequence we can guarantee that for each 
	$j \in I_0 \setminus J$, a sequence of probability measures $\big(\eta_{x_n}^{I_0 \setminus \{j\}}\big)$ weakly converges, 
	see \eqref{eq:37} for the definition.

	Before we embark on the proof, let us observe the following fact.
	\begin{claim}
		\label{clm:8}
		For each $\omega_1 \in \Omega$, $i \in I_0$ and $n \in \NN_0$, if
		\[
			\sprod{h(o, u_{n+2}; \omega_1)}{\alpha_i} \geqslant \sprod{h(o, u_{n+1}; \omega_1)}{\alpha_i}
		\]
		then
		\[
			\sprod{h(o, u_{n+1}; \omega_1)}{\alpha_i} \geqslant \sprod{h(o, u_n; \omega_1)}{\alpha_i}.
		\]
	\end{claim}
	For the proof, let us denote by $\scrA$ an apartment containing $[u_n, \omega_1]$. Since $u_{n+1} \in [u_n, u_{n+2}]$,
	\[
		\sprod{h(u_n, u_{n+1}; \omega_1)}{\alpha_i} < 0, \quad\text{and}\quad
		\sprod{h(u_{n+1}, u_{n+2}; \omega_1)}{\alpha_i} \geqslant 0,
	\]
	cannot hold true because the $\alpha_i$-wall passing through $\rho^{u_n, \omega_1}_\scrA(u_{n+1})$ cannot have
	$\rho^{u_n, \omega_1}_\scrA(u_{n+2})$ on the same side as $[u_n, \omega_1]$.

	As a consequence of Claim \ref{clm:8}, if $i \in I_0 \setminus J$, then the sequence 
	$(\sprod{h(o, u_{n}; \omega_1)}{\alpha_i} : n \in \NN_0)$ cannot be eventually constant. Furthermore, either it is
	unbounded, or there is $n_0 \geqslant 0$ such that for all $n > n_0$,
	\[
		\sprod{h(o, x_n; \omega_1)}{\alpha_i} < \sprod{h(o, u_{n_0}; \omega_1)}{\alpha_i}
	\]
	and
	\[
		\sprod{h(o, u_{n_0}; \omega_1)}{\alpha_i} = \sup_{n \in \NN} {\sprod{h(o, u_n ; \omega_1)}{\alpha_i}}.
	\]

	\vspace*{2ex}
	\noindent
	{\bf Step 1.}~In this step we show that the support of $\mu$ is contained in $R$. 
	\noindent
	Let $\omega_0 \in \Omega \setminus R$. Then there is $i \in I_0 \setminus J$ such that
	\begin{equation}
		\label{eq:20}
		\sup_{n \in \NN_0}{\sprod{h(o, u_n; \omega_0)}{\alpha_i}} < \infty.
	\end{equation}
	We are going to construct $U \subset \Omega$, an open neighborhood of $\omega_0$, such that $\mu(U) = 0$.

	In view of \eqref{eq:20} and Claim \ref{clm:8}, there are $N \geqslant 1$ and $n_0 \in \NN$ such that
	\[
		\sprod{h(o, u_{n_0+1} ; \omega_0)}{\alpha_i} <
		\sprod{h(o, u_{n_0} ; \omega_0)}{\alpha_i} = 
		\sup_{n \geqslant 1}{\sprod{h(o, u_n; \omega_0)}{\alpha_i}} \leqslant N.
	\]
	Let
	\[
		U = \left\{\omega_1 \in \Omega :
	        \begin{aligned}
				\sprod{h(o, u_{n_0+1}; \omega_1)}{\alpha_i} &< \sprod{h(o, u_{n_0}; \omega_1)}{\alpha_i} \\
				\sprod{h(o, u_{n_0-1}; \omega_1)}{\alpha_i} &\leqslant \sprod{h(o, u_{n_0}; \omega_1)}{\alpha_i} \\
				\sprod{h(o, u_{n_0}; \omega_1)}{\alpha_i} &\leqslant N
				\end{aligned}
			\right\}.
	\]
	By Claim \ref{clm:8}, for all $\omega_1 \in U$, we have
	\[
		\sup_{n \geqslant n_0}{\sprod{h(o, u_n; \omega_1)}{\alpha_i}} 
		=
		\sprod{h(o, u_{n_0}; \omega_1)}{\alpha_i}
		\leqslant N.
	\]
	Since the Busemann function takes values in $\frac{1}{2}P$, the set $U$ is an open neighborhood of $\omega_0$. 
	Furthermore, as a sublevel set of a continuous function it is also closed. Hence,
	\[
		\mu(U) = \lim_{n \to \infty} \nu_{x_n}(U).
	\]
	Next, we are going to compute the limit. Setting $U' = U \cap \Omega'(\omega)$, by co-nullity of big cells 
	(see Theorem \ref{thm:1}) we get
	\[
		\nu_{x_n}(U)
		=
		\nu_{x_n}(U')
		= \int_{(U')^{\{i\}}}
		\nu_{\pi_{R_1}(x_n)}\big(\vphi_{R_1}(R_1 \cap U')\big) \: \eta^{\{i\}}_{x_n}({\rm d} R_1)
	\]
	where in the second equality we have used \eqref{eq:19}. Recall that we know explicitly Radon--Nikodym derivatives
	between harmonic measures (Proposition \ref{prop:2}), that is for each $R_1 \in (U')^{\{i\}}$, we have
	\begin{align*}
		\nu_{\pi_{R_1}(x_n)} \big(\vphi_{R_1}(R_1 \cap U')\big)
		&= 
		\int_{\vphi_{R_1}(R_1 \cap U')}
		\chi^{\{i\}}\big(h_{R_1}(\pi_{R_1}(o), \pi_{R_1}(x_n); \omega_1)\big) 
		\nu_{\pi_{R_1}(o)}({\rm d} \omega_1) \\
		&=
		\int_{\vphi_{R_1}(R_1 \cap U')}
		\chi\Big(P_{\{i\}} h(o, x_n; \vphi_{R_1}^{-1}(\omega_1))\Big)
		\nu_{\pi_{R_1}(o)}({\rm d} \omega_1)
	\end{align*}
	where the last equality is a consequence of Lemma \ref{lem:6}.

	By Claim \ref{clm:8}, for $\omega_1 \in U$ and $n > n_0$,
	\begin{align*}
		\sprod{h(o, x_n; \omega_1)}{\alpha_i}
		&<
		\sprod{h(o, u_{n_0}; \omega_1)}{\alpha_i} \leqslant N,
	\end{align*}
	thus
	\[
		\sup_{\omega_1 \in \vphi_{R_1}(R_1 \cap U) }
		{\sup_{n > n_0}
		{\chi\Big(P_{\{i\}} h(o, x_n; \vphi_{R_1}^{-1}(\omega_1))\Big)}} \leqslant N.
	\]
	Our next aim is to show that for every sequence $(\omega^n : n \in \NN) \subset U'$ tending to $\omega_1 \in U'$, 
	we have
	\begin{equation}
		\label{eq:38}
		\lim_{n \to \infty} \chi\Big(P_{\{i\}} h(o, x_n; \vphi_{R_1}^{-1}(\omega^n))\Big)=0.
	\end{equation}
	Since
	\[
		\lim_{n \to \infty} \sprod{h(o, u_n; \omega_1)}{\alpha_i} = -\infty,
	\]
	by Claim \ref{clm:8}, for each $A > 0$ there is $n' \geqslant 1$ such that for all $n > n'$,
	\[
		\sprod{h(o, u_n; \omega_1)}{\alpha_i} 
		<
		\sprod{h(o, u_{n'}; \omega_1)}{\alpha_i}
		\leqslant -A.
	\]
	By continuity, there is an open set $B \subset \Omega$, $\omega_1 \in B$, such that for all $\omega_2 \in B$,
	\[
		\sprod{h(o, u_{n'}; \omega_2)}{\alpha_i} \leqslant -A.
	\]
	Let
	\[
		V = \big\{\omega_2 \in B \cap U' : 
		\sprod{h(o, u_{n'+1}; \omega_2)}{\alpha_i} < \sprod{h(o, u_{n'}; \omega_2)}{\alpha_i} \big\}.
	\]
	The set $V$ is an open neighborhood of $\omega_1$ such that for all $\omega_2 \in V$, and $n > n'$,
	\[
		\sprod{h(o, x_n; \omega_2)}{\alpha_i} < \sprod{h(o, u_{n'}; \omega_2)}{\alpha_i} \leqslant -A
	\]
	which leads to \eqref{eq:38}. Now, by the Lebesgue's dominated convergence theorem for weakly convergent probability 
	measures \cite[Theorem 3.5]{Serfozo1982}, we obtain
	\[
		\mu(U) 
		= \lim_{n \to \infty}
		\int_{(U')^{\{i\}}}
		\int_{\vphi_{R_1}(R_1 \cap U')}
		\chi\Big(P_{\{i\}} h(o, x_n; \vphi_{R_1}(\omega_1))\Big)
		\nu_{\pi_{R_1}(o)}({\rm d} \omega_1)
		\: \eta^{\{i\}}_{x_n}({\rm d} R_1)
		=0
	\]
	proving that $U \cap \supp \mu = \varnothing$. Consequently, $\supp \mu \subseteq R$.

	\vspace*{2ex}
	\noindent
	{\bf Step 2.}~In this step, we show that any cluster value $\mu$ as in Step 1, which we write $\mu = \lim \nu_{x_n}$ 
	for simplicity, is such that the measure $(\phi_R)_*\mu$ is the harmonic measure of $\scrX(F)$ attached to the vertex
	$x_F$.

	First, let us consider the map
	\[
		\begin{alignedat}{1}
			\pi^J: \Omega &\longrightarrow \Omega^J \\
			\omega &\longmapsto \res_J(\omega).
		\end{alignedat}
	\]
	Since $(\pi^J)_* \nu_x = \eta_x^J$ by Proposition \ref{prop - disintegration}, we get
	\begin{align}
		\nonumber
		\lim_{n \to \infty} \eta_{x_n}^J 
		&=
		\lim_{n \to \infty} (\pi^J)_*\nu_{x_n} \\
		\label{eq:43}
		&=
		(\pi^J)_* \mu = \delta_{R}.
	\end{align}
	Let $x_0 = x_F$. Let $y_0$ be any special vertex of $\scrX(F)$. 
	Let $\omega_R' \in \Omega_R(x_0,y_0)$. There is an apartment containing the cone $u_1 + F$ and the germ $y_0$ 
	\cite[\S 3.1]{Rousseau2023}. Then there exists $\omega'$ such that the sector $[u_1,\omega']$ projects onto
	$[x_0, \omega_R]$. Let us pick $y' \in [u_1,\omega']$ such that $\pi_R(y') = y_0$. By construction
	$\omega' \in \Omega(u_1,y')$; moreover by convexity any $\omega_R'' \in \Omega_R(x_0,y_0)$ will be the image of some 
	$\omega'' \in \Omega(u_1,y')$. 
	
	Now our aim is to compute $((\phi_{R})_*\mu)(\Omega_{R}(x_0,y_0))$. First for $j \in J$, we set
	\[
		V_j = \Big\{\omega_1 \in \Omega:
		\sprod{h(u_1, u_2; \omega_1)}{\alpha_j} \leqslant \sprod{h(u_1, u_1; \omega_1)}{\alpha_j} = 0
		\Big\}.
	\]
	Since each $V_j$ is clopen and contains $R$, the set
	\[
		U = \Omega(u_1, y') \cap \bigcap_{j \in J} V_j
	\]
	is clopen and
	\[
		\phi_{R}(U) = \Omega_{R}(x_0, y_0).
	\]
	Moreover, we have
	\begin{align*}
		((\phi_{R})_* \mu) \big(\Omega_{R}(x_0, y_0)\big)
		&=
		\mu \big(\phi_{R}^{-1} (\Omega_{R}(x_0, y_0)) \big) \\
		&=
		\mu \big(\phi_{R}^{-1} (\Omega_{R}(x_0, y_0)) \cap R \big)
		=
		\mu \big(U \cap R \big)
		=
		\mu (U).
	\end{align*}
	Hence, we can write
	\begin{align*}
		((\phi_{R})_* \mu) \big(\Omega_{R}(x_0, y_0)\big)
		&=
		\lim_{n \to \infty}
		\nu_{x_n}(U).
	\end{align*}
	Now, by \eqref{eq:19}
	\[
		\nu_{x_n}(U) = \int_{U^J} \nu_{\pi_{R_1}(x_n)}\big(\vphi_{R_1}(U \cap R_1)\big) \nu_{x_n}^J({\rm d} R_1).
	\]
	In view of Proposition \ref{prop:2},
	\[
		\nu_{\pi_{R_1}(x_n)}\big(\vphi_{R_1}(U \cap R_1)\big)
		=
		\int_{\vphi_{R_1}(U \cap R_1)} \chi^J\big(h_{R_1}(\pi_{R_1}(u_1), \pi_R(x_n); \omega_1)\big) 
		\nu_{\pi_R(u_1)}({\rm d} \omega_1).
	\]
	To complete the proof, we need the following fact.
	\begin{claim}
		\label{clm:1}
		For all $\lambda \in P^+$,
		\[
			\chi(P_J \lambda) = \prod_{\alpha \in \Phi_J^+} \tau_\alpha^{\sprod{\lambda}{\alpha}}. 
		\]
	\end{claim}
	To see this it is enough to notice that for all $w \in W_J$, and $\alpha \in \Phi^+ \setminus \Phi_J$,
	\[
		w.\alpha \in \Phi^+ \setminus \Phi_J,
		\qquad\text{and}\qquad
		\tau_{w.\alpha} = \tau_{\alpha}.
	\]
	Next, thanks to Claim \ref{clm:1}, for every $\omega_1 \in U$,
	\begin{align*}
		\chi^J\big(h_{R_1}(\pi_{R_1}(u_1), \pi_{R_1}(x_n); \omega_1)\big) 
		&=
		\chi\big(P_J h(u_1, x_n; \omega_1)\big) \\
		&=
		\prod_{\alpha \in \Phi^+_J} \tau_\alpha^{\sprod{h(u_1, x_n; \omega_1)}{\alpha}}
		\leqslant 1,
	\end{align*}
	therefore
	\[
		\nu_{\pi_{R_1}(x_n)}\big(\vphi_{R_1}(U \cap R_1)\big) \leqslant \nu_{\pi_{R_1}(u_1)}\big(\vphi_{R_1}(U \cap R_1)\big),
	\]
	and so
	\[
		\nu_{x_n}(U) \leqslant \int_{U^J}
		\nu_{\pi_{R_1}(u_1)}\big(\vphi_{R_1}(U \cap R_1)\big)
		\eta_{x_n}^J({\rm d} R_1).
	\]
	In view of \eqref{eq:43}, we get
	\begin{align*}
		(\phi_{R})_* \mu \big(\Omega_{R}(x_0, y_0)\big)
		&\leqslant
		\nu_{x_0} \big(\Omega_{R}(x_0, y_0)\big).
	\end{align*}
	Since $y_0$ was arbitrary and both $(\phi_{R})_* \mu$ and $\nu_{x_0}$ are probability measures, we must have
	$(\phi_{R})_* \mu = \nu_{x_0}$.

	To show the uniqueness statement, let us suppose that there are two core sequences $(x_n)$ and $(x_n')$ of type
	$(\omega, J, c)$ and $(\omega', J', c')$, respectively, converging to the same limit, say $\mu$. 
	By \eqref{en:5:1}, we have $\res_J(\omega) = {\rm supp}(\mu) = \res_{J'}(\omega')$, and since types of residues
	are well-defined we deduce that $J' = J$ and $\omega' \sim_J \omega$. Moreover, in view of the description of
	$(\varphi_J)_*\mu$ as a harmonic measure, and since Proposition \ref{prop:2} implies that the vertex defining such a
	measure is well-defined, we have $x_{F_J} = x_{F_{J'}}$ and therefore the parameters $c$ and $c'$
	must be the same. This completes the proof.
\end{proof}

\begin{theorem}
	\label{thm:3}
	Let $\scrX$ be a thick regular locally finite affine building. The closure of the collection of harmonic measures
	on $\scrX$ in the space of probability measures $\calP(\Omega)$ on the maximal boundary $\Omega$ endowed with 
	the weak-$*$ topology is $\Aut(\scrX)$-equivariantly isomorphic to the polyhedral or to the combinatorial compactification
	of $\scrX$. More precisely, the maximal boundary of each affine building at infinity, or stratum, can be seen as a residue
	in $\Omega$ and any cluster value of any unbounded sequence of harmonic measures in $\scrX$ is a harmonic measure on a
	well-defined stratum. 
\end{theorem}
\begin{proof}
	In the three considered compactifications, core sequences converge: for the polyhedral compactification it follows from 
	Lemma \ref{lem:core CV in poly}, for the Caprace--L\'ecureux compactification it follows from Theorem \ref{thm:12}
	and for the Furstenberg compactification it follows from Theorem \ref{thm:5}. As a consequence, combining a standard 
	topological argument (see \emph{e.g.}~the domination criterion given by \cite[Lemma 3.28]{gjt} in both directions between
	two compactifications) and the uniqueness assertions in these results provide the identifications between the three
	compactifications, hence the first part of Theorem \ref{thm:3}. In view of Lemma \ref{lem:1}, we can identify
	the maximal boundary of $\scrX(F)$ with $\res_J(\omega)$, thus the rest of Theorem \ref{thm:3} is a consequence of 
	\eqref{en:5:1} and \eqref{en:5:2} in Theorem \ref{thm:5}.
\end{proof}

\subsection{Furstenberg compactification for Bruhat--Tits buildings}
\label{ss - Furst group}
We provide here the Bruhat-Tits aspect of the previous section. 
The measure-theoretic compactification procedures are among the oldest ones for Riemannian symmetric spaces
\cite{Furstenberg1963}. The idea, due to H.~Furstenberg, is beautiful: it consists in using probability measures on the (maximal)
boundary $\Omega$ by seeing it as a homogeneous space for as many compact subgroups as possible in the ambient Lie group.
In order to be more precise, we need to combine Iwasawa decompositions of ${\bf G}(k)$ and Levi decompositions of parabolic
$k$-subgroups. For any chamber at infinity $\omega \in \Omega$, we saw that have $\Omega \simeq {\bf G}(k)/{\bf P}_\omega(k)$.
We can furthermore choose a maximal $k$-split torus ${\bf S}$ such that the spherical apartment $\mathcal{A}({\bf S})^\infty$ 
contains $\omega$ and we can pick a special vertex $x$ in the affine apartment ${\mathcal A}({\bf S})$. While we have an Iwasawa
decomposition (Section \ref{sec:4}) 
\[
	{\bf G}(k) = K_x \, {\bf S}(k) \, {\bf U}^\omega(k),
\]
we also have a Levi decomposition (Section \ref{sec:14})
\[
	{\bf P}_\omega(k) = 
	\bigl([{\bf Z}_{\bf G}({\bf S}),{\bf Z}_{\bf G}({\bf S})](k) \cdot {\bf S}(k) \bigr) \ltimes  {\bf U}^\omega(k).
\]
We used here the decomposition of the centralizer ${\bf Z}_{\bf G}({\bf S})$ into its part (derived subgroup)
$[{\bf Z}_{\bf G}({\bf S}),{\bf Z}_{\bf G}({\bf S})]$ whose $k$-rational points fix pointwise the apartment
${\mathcal A}({\bf S})$ and the part ${\bf S}$ whose $k$-rational points provide the translations of the affine Weyl group.

Putting together these decompositions and then varying ${\bf S}$ and $x \in {\mathcal A}({\bf S})$, we see that $\Omega$ is 
acted upon continuously and transitively by the stabilizer $K_x$ of any special vertex $x$. It follows then from general
integration theory on homogeneous spaces \cite[VII \S 2 6, Th\'eor\`eme 3]{Bourbaki2004} that there is a unique $K_x$-invariant
probability measure on $\Omega$, which we denote by $\mu_x$; we call $\mu_x$ the \emph{homogeneous measure} associated with 
the special vertex $x$. As a result, if we also use the construction from Proposition \ref{prop:2} we can associate i
to each special vertex $x$ the harmonic measure $\nu_x$ and the homogeneous measure $\mu_x$. Here is the Bruhat--Tits version of
Theorem \ref{th - harm conv} in Introduction. 

\begin{theorem}
\label{th - Comp Furstenberg}
	Let ${\bf G}$ be a simply connected semisimple algebraic group defined over a non-Archimedean local field $k$ and let 
	${\mathcal X}({\bf G},k)$ be the associated Bruhat--Tits building. We denote by $\Omega$ the maximal boundary of
	${\mathcal X}({\bf G},k)$ and by $\calP(\Omega)$ the set of probability measures on it, endowed with the weak-$*$ topology.
	For any special vertex $x \in {\mathcal X}({\bf G},k)$, the harmonic measure $\nu_x$ and the homogeneous measure $\mu_x$
	coincide, therefore the Furstenberg compactification is also the closure of the set of homogeneous measures in 
	$\calP(\Omega)$.
\end{theorem}

\begin{remark}
	The proof below is valid in the more general case when ${\bf G}(k)$ is replaced by a type-preserving and strongly transitive 
	automorphism group $G$ acting on a locally finite affine building $\scrX$. 
\end{remark}

\begin{proof} 
	By the uniqueness of homogeneous measures checked above, it is enough to show that for any $x \in V_s$ the probability
	measure $\nu_x$ is $K_x$-invariant. Let us pick $x \in V_s$ and $k \in K_x$; we thus need to show that $k_*\nu_x=\nu_x$, 
	which we do by checking harmonicity of $k_*\nu_x$ (using the uniqueness given by Proposition \ref{prop:2}).
	Let $y,z \in V_s$ be such that $\sigma(x,y) = \sigma(x,z)$. By the geometric interpretation of the Cartan decomposition 
	(Section \ref{sec:4}), the group $K_x$ acts transitively on the sectors tipped at $x$, and since it preserves types 
	(because ${\bf G}$ is assumed to be simply connected) we have $\sigma(x,k.y)=\sigma(x,k.z)$. By harmonicity of $\nu_x$, this
	implies that $\nu_x\bigl( \Omega(x,k.y) \bigr) = \nu_x\bigl( \Omega(x,k.z) \bigr)$. But $\Omega(x,k.y) = \Omega(k.x,k.y)
	= k .\Omega(x,y)$, and similarly $\Omega(x,k.z) = \Omega(k.x,k.z) = k. \Omega(x,z)$, so the previous equality says that
	for any $y,z \in V_s$ such that $\sigma(x,y) = \sigma(x,z)$, we have $(k^{-1}{}_*\nu_x)\bigl( \Omega(x,y) \bigr) 
	= (k^{-1}{}_*\nu_x)\bigl( \Omega(x,z) \bigr)$; this is the requested harmonicity, hence the first statement. 
\end{proof}
One complementary question is to try to attach a natural measure on $\Omega$ to an arbitrary point of the building.
The problem is the lack of transitivity of the action on $\Omega$ for an arbitrary facet stabilizer. More precisely, when
we described geometrically the Cartan decomposition in Section \ref{sec:4}, we saw that the first step (the one using
foldings given by root group actions, see Figure \ref{fig:10}) could show that if $c$ is an alcove in a given
apartment $\mathcal{A}$, then the Iwahori subgroup ${\rm Stab}_{{\bf G}(k)}(c)$ acts on ${\mathcal X}({\bf G},k)$ with
$\mathcal{A}$ as a fundamental domain; this is the geometric counterpart to the Bruhat--Tits decomposition
${\bf G}(k) = \bigsqcup_{w \in W^a} {\rm Stab}_{{\bf G}(k)}(c) w {\rm Stab}_{{\bf G}(k)}(c)$, 
see \cite [Proposition 4.2.1]{BruhatTits1972}. By approximating geodesic rays by geodesic segments and passing to the limit,
the outcome is that ${\rm Stab}_{{\bf G}(k)}(c)$ acts on $\Omega$ with a fundamental set of representatives given by the
chambers at infinity lying in $\mathcal{A}^\infty$. Keeping the group ${\rm Stab}_{{\bf G}(k)}(c)$, which does not contain any
lift of elements of the vectorial part of $W^a$, we cannot do better than this because the second step is not available.
This shows that the set of ${\rm Stab}_{{\bf G}(k)}(c)$-orbits on $\Omega$ is indexed by the spherical Weyl group of ${\bf G}$
over $k$ (the combinatorial counterpart here is \cite[Th\'eor\`eme 5.1.3 (vi)]{BruhatTits1972}). In particular, we cannot see
$\Omega$ as a homogeneous space for the Iwahori subgroup ${\rm Stab}_{{\bf G}(k)}(c)$ and deduce that it carries a unique
invariant probability measure: there is a simplex of possible choices according to the mass given to each orbit.
Note that the problem still holds for non special vertices, whose stabilizers in $W^a$ do not act transitively on 
$\mathcal{A}^\infty$. This explains, at least in the Bruhat--Tits case, why we only considered the set of special vertices when
defining Furstenberg compactifications. We intend to go back to this problem in a subsequent work. 

\section{Martin compactification}
\label{sec:7}
Our aim is to construct Martin compactifications for the set of special vertices of affine buildings. It relies on 
the asymptotics of the Green's function recently obtained by the second author, see \cite{tr}. It covers a large class of
random walks on good vertices. Consequently, we can achieve our program for all affine buildings with reduced root systems.
For non-reduced case, by changing the origin $o$, we can compactify only half of the special vertices at the same time. To
complete the program in the non-reduced case, we introduce a certain \emph{distinguished} random walk described in Appendix 
\ref{ap:1}.

Let us start by recalling the basics of random walks on discrete structures, and then we immediately specialize the situation 
to affine buildings by describing the asymptotics of ground spherical functions. This enables us to provide the uniqueness 
statements for limits of Martin kernels, which lead to the identifications announced in Theorem \ref{th - Martin} of 
Introduction. The last two parts contain the convergence theorems which define and describe the Martin compactifications 
of affine buildings, at and above the bottom of the spectrum. 

\subsection{Martin embeddings}
\label{ss - Martin embeddings}
We say that a random walk is \emph{isotropic} if the transition probabilities $p(x, y)$ only depend on $\sigma(x,y)$, 
{\it i.e.} are constant on
\[
	\big\{(x',y') \in V_g \times V_g: \sigma(x', y') = \sigma(x, y) \big\}.
\]
We set 
\begin{align*}
	p(0; x, y) &= \delta_x(y), \\
	p(n; x, y) &= \sum_{z \in V_p} p(n-1; x, z) p(z, y), \qquad n \geqslant 1.
\end{align*}
A random walk is \emph{irreducible} if for each $x, y \in V_g$ there is $n \in \NN$ such that $p(n; x, y) > 0$.
It is \emph{aperiodic} if for every $x \in V_g$,
\[
	\gcd\big\{n \in \NN : p(n; x, x) > 0\big\} = 1.
\]
And lastly, a walk has \emph{finite range} if for every $x \in V_g$, we have
\[
	\# \{y \in V_g : p(x, y) > 0 \big\} < \infty.
\]
For each isotropic finite range random walk on good vertices of $\scrX$ with transition density $p$, we define the
corresponding operator acting on $f: V_g \rightarrow \CC$ as
\begin{equation}
	\label{eq:21}
	A f(x) = \sum_{y \in V_g} p(x, y) f(y).
\end{equation}
Let $\mathscr{A}_0 = \cspan\{A_\lambda : \lambda \in P^+\}$ be the commutative $\star$-subalgebra of the
algebra of bounded linear operators on $\ell^2(V_g)$ where $A_\lambda$ are defined in \eqref{eq:66}. The multiplicative
functionals on $\scrA_0$ are given in terms of Macdonald spherical functions. The latter are defined for $\lambda \in P^+$, as
\begin{equation}
	\label{eq:22}
	P_\lambda(z) = \frac{\chi^{-\frac{1}{2}}(\lambda)}{W(q^{-1})} 
	\sum_{w \in W} e^{\sprod{w . z}{\lambda}} \bfc(w . z),
	\qquad z \in \CC^r,
\end{equation}
where
\begin{align}
	\nonumber
	\bfc(z) &= \prod_{\alpha \in \Phi^+} 
	\frac{1-\tau_\alpha^{-1} \tau_{\alpha/2}^{-1/2} e^{-\sprod{z}{\alpha\spcheck}}} 
	{1 - \tau_{\alpha/2}^{-1/2} e^{-\sprod{z}{\alpha\spcheck}}} \\
	\label{eq:10}
	&=
	\prod_{\alpha \in \Phi^{++}}
	\frac{\Big(1-\tau_{2\alpha}^{-1} \tau_\alpha^{-\frac{1}{2}} e^{-\frac{1}{2} \sprod{z}{\alpha\spcheck}}\Big)
	\Big(1 + \tau_\alpha^{-\frac{1}{2}} e^{-\frac{1}{2} \sprod{z}{\alpha\spcheck}}\Big)}
	{1 - e^{-\sprod{z}{\alpha\spcheck}}}.
\end{align}
The values of $P_\lambda$ when the denominator of the $\bfc$-function equals zero can be obtained by taking appropriate limits.
The mapping 
\[
	h_z(A_\lambda) = P_\lambda(z), \qquad z \in \CC^r, \lambda \in P^+
\]
extends to a multiplicative functional, still denoted by $h_z$, on $\scrA_0$. Moreover, all multiplicative functions on $\scrA_0$
are of this form. For more details about spherical harmonic analysis on $\scrX$ we refer the interested reader to \cite{macdo0}
and \cite{park2}.

In fact the formula \eqref{eq:22} defines Macdonald spherical functions for a given root system $\Phi$ and parameters
$(\tau_\alpha : \alpha \in \Phi)$ invariant under the action of the Weyl group $W$. In particular, the definition is valid
without any underlying building, see e.g. \cite{Matsumoto77}.

We henceforth fix an isotropic finite range random walk on good vertices of $\scrX$, with transition function $p$. Since
the walk has finite range, there are a finite set $\calV \subset P$ and positive real numbers $\{c_v:v \in \calV\}$,
such that
\begin{equation}
	\label{eq:68}
    \kappa(z) := \varrho^{-1} h_z(A) = \sum_{v \in \calV} c_v e^{\sprod{z}{v}}, \qquad z \in \CC^r
\end{equation}
where
\begin{equation}
	\label{eq:56}
	\varrho = h_0(A)
\end{equation}
with $A$ given by \eqref{eq:21}. 
Let us observe that $\varrho$ is the spectral radius of $A$. Indeed, by the Gelfand
theorem, $\varrho = \sup_{z \in M_2} |h_z(A)|$ where $M_2$ denotes the spectrum of the commutative $C^\star$-algebra $\scrA_2$, 
and since $A$ is a finite convex combination of $A_\lambda$, we get
\begin{align}
	\nonumber
	h_0(A) \leqslant \sup_{z \in M_2} |h_z(A)| 
	&\leqslant \sum_\lambda a_\lambda \sup_{z \in M_2} |h_z(A_\lambda)| \\
	\label{eq:60}
	&= \sum_\lambda a_\lambda h_0(A_\lambda)
	= h_0(A)
\end{align}
where the penultimate equality follows by \cite[Theorem 6.5]{park2}.

For each $\zeta \geqslant \varrho$, the \emph{Green function} $G_\zeta$ is defined as
\[
	G_\zeta(x, y) = \sum_{n \geqslant 0} \zeta^{-n} p(n; x, y), \qquad x, y \in V_g.
\]
Without loss of generality we assume that the random walk is aperiodic. Indeed, otherwise we consider
\[
	\tilde{p}(x, y) = \tfrac{1}{2} \delta_x(y) + \tfrac{1}{2} p(x, y).
\]
Then
\[
	G_\zeta(x, y) = \frac{\zeta}{\zeta+1} \tilde{G}_{\frac{\zeta+1}{2}}(x, y).
\]
Next, let us observe that for each $y \in V_g$, the function
\[
	V_g \ni x \mapsto G_\zeta(x, y)
\]
is \emph{$\zeta$-harmonic}, that is
\[
	\sum_{v \in V_g} p(x, v) G_\zeta(v, y) = \zeta G_\zeta(x, y).
\]
Recall that a function $f : V_g \rightarrow \RR$, is called \emph{$\zeta$-superharmonic} if $A f \leqslant \zeta f$.
Let us denote by $\calB_\zeta(V_g)$ the set of positive $\zeta$-superharmonic functions
on good vertices of $\scrX$, normalized to take value $1$ at the origin $o$. The set $\calB_\zeta(V_g)$ endowed with
the topology of pointwise convergence is a compact second countable Hausdorff space, thus it is metrizable. For an
automorphism $g \in \Aut(\scrX)$ and a function $f \in \calB_\zeta(V_g)$ we set
\begin{equation}
	\label{eq - proj action}
	(g . f)(x) = \frac{f(g^{-1} . x)}{f(g^{-1} . o)}
\end{equation}
provided that $f(g^{-1} . o) \neq 0$. Let us define the map
\begin{alignat*}{1}
	\iota: V_g &\longrightarrow \calB_\zeta(V_g) \\
	y &\longmapsto K_\zeta(\,\cdot\,, y) 
\end{alignat*}
where for $x, y \in V_g$ we have set
\[
	K_\zeta(x, y) = \frac{G_\zeta(x, y)}{G_\zeta(o, y)}.
\]
Since the random walk is transient, the map $\iota$ is injective. Moreover, for $g \in \Aut(\scrX)$, we have
\begin{align*}
	g . \iota(y) 
	= g . K_\zeta(\cdot,y) 
	= \frac{K_\zeta(g^{-1}\cdot,y)}{K_\zeta(g^{-1} . o, y)} 
	= \frac{K_\zeta(\cdot,g . y)}{K_\zeta(o, g . y)} &= K_\zeta(\cdot,g . y) = \iota(gy)
	\end{align*}
thus $\iota$ is equivariant. Notice that $\iota(V_g)$ is also discrete, because for $y, y' \in V_g$ with $y \neq y'$, we have 
\begin{align*}
	K_\zeta(y, y) - K_\zeta(y, y') 
	&= \frac{G_\zeta(y, y) G_\zeta(o, y') - G_\zeta(o, y) G_\zeta(y, y')}{G_\zeta(o, y) G_\zeta(o, y')} \\
	&= \frac{(G_\zeta(o, o) - 1) G_\zeta(o, y') + G_\zeta(o, y') 
	- G_\zeta(o, y) G_\zeta(y, y')}{G_\zeta(o, y) G_\zeta(o, y')}\\
	&\geqslant \frac{G_\zeta(o, o)-1}{G_\zeta(o, y)}.
\end{align*}
Let $\clr{\scrX}_{M, \zeta}$ be the closure of $\iota(V_g)$ in $\calB_{\zeta}(V_g)$. The space $\calB_\zeta(V_g)$
is metrizable, thus by Lemma \ref{lem:2}, while studying $\clr{\scrX}_{M, \zeta}$ we restrict attention to core
sequences. By \cite[Proposition 6.4]{gjt}, the group $\Aut(\scrX)$ acts continuously on $\clr{\scrX}_{M, \zeta}$.

\subsection{Asymptotic behavior of ground state spherical functions}
Before embarking on the computing the Martin kernels, we need the following auxiliary result. Given a subset $J \subseteq I_0$, 
let us define
\[
	\Xi_J(\mu) = 
	\lim_{\theta \to 0} 
	\frac{1}{|W_J|}
	\sum_{w \in W_J}
    e^{\sprod{w . \theta}{\mu}} \bfc_J(w . \theta), \qquad \mu \in P 
\]
where $\bfc_J$ denotes $\bfc$-function for the root system $\Phi_J$, that is
\[
	\bfc_J(z) = \prod_{\alpha \in \Phi_J^+} 
	\frac{1-\tau_\alpha^{-1} \tau_{\alpha/2}^{-1/2} e^{-\sprod{z}{\alpha\spcheck}}} 
	{1 - \tau_{\alpha/2}^{-1/2} e^{-\sprod{z}{\alpha\spcheck}}}.
\]
Let $\bfb_J(z) = e^{\sprod{z}{\rho_J}} \bfc_J(z) \Delta_J(z)$ where
$\rho_J = \frac{1}{2} \sum_{\alpha \in \Phi^{++}_J} \alpha\spcheck$, and
\[
	\Delta_J(z) = \prod_{\alpha \in \Phi^{++}_J} \Big(e^{\frac{1}{2}\sprod{z}{\alpha\spcheck}} - 
	e^{-\frac{1}{2} \sprod{z}{\alpha\spcheck}} \Big).
\]
If $J = I_0$ we drop the index $J$ from the notation. By \eqref{eq:22} it stems
\begin{equation}
	\label{eq:41}
	\Xi(\mu) = \frac{W(q^{-1})}{|W|} \chi^{\frac{1}{2}}(\mu) P_\mu(0).
\end{equation}
We start by the following lemma.
\begin{lemma}
	\label{lem:7}
	For each $\lambda \in P$,
	\[
		\Xi(\lambda) = 
		\frac{1}{a}
		\bigg\{
		\prod_{\alpha \in \Phi^{++}}
		\sprod{\nabla}{\alpha\spcheck}
		\bigg\}
		\bigg\{
		e^{\sprod{\theta}{\lambda+\rho}} \bfb(\theta)
		\bigg\}_{\theta = 0}
	\]
	where
	\[
		a = \bigg\{\prod_{\alpha \in \Phi^{++}} \sprod{\nabla}{\alpha\spcheck}\bigg\}
		\big\{\Delta(\theta)\big\}_{\theta = 0}.
	\]
\end{lemma}
\begin{proof}
	First, let us write
	\[
		\sum_{w \in W} e^{\sprod{w .  \theta}{\lambda}} \bfc(w . \theta) 
		=
		\sum_{w \in W} e^{\sprod{w . \theta}{\lambda+\rho}} 
		\bfb(w . \theta) \frac{1}{\Delta(w . \theta)}.
	\]
	Since $\Delta(\theta)$ is a $W$-anti-invariant exponential polynomial, we have
	\[
		\sum_{w \in W} e^{\sprod{w . \theta}{\lambda}} \bfc(w . \theta)
		=
		\frac{1}{\Delta(\theta)} 
		\sum_{w \in W}
		(-1)^{\ell(w)}
		e^{\sprod{w . \theta}{\lambda+\rho}} 
		\bfb(w . \theta).
	\]
	Now, by multiple applications of L'H\^opital's rule we get (see e.g. \cite{a1})
	\begin{align*}
		\lim_{\theta \to 0}
		\frac{a}{\abs{W}}
		\sum_{w \in W} e^{\sprod{w . \theta}{\lambda}} \bfc(w . \theta) 
		&=\bigg\{
		\prod_{\alpha \in \Phi^{++}}
		\sprod{\nabla}{\alpha\spcheck}
		\bigg\}
		\bigg\{
		\frac{1}{|W|}\sum_{w \in W} (-1)^{\ell(w)} 
		e^{\sprod{w. \theta}{\lambda+\rho} \bfb(w. \theta)}
		\bigg\}_{\theta = 0}
		\\
		&=
		\bigg\{
		\prod_{\alpha \in \Phi^{++}}
		\sprod{\nabla}{\alpha\spcheck}
		\bigg\}
		\bigg\{
		e^{\sprod{\theta}{\lambda+\rho}} \bfb(\theta)
		\bigg\}_{\theta = 0}
	\end{align*}
	where the last equality follows by $W$-anti-invariance of the differential operator. This completes the proof
	of the lemma.
\end{proof}

We are now ready to prove the asymptotic formula for the auxiliary functions $\Xi_J$, which will be used in the next section 
to describe the asymptotics of the ground spherical functions. The following proposition is motivated by 
\cite[Remark 2.2.13]{aj}.
\begin{proposition}
	\label{prop:1}
	Let $J \subsetneq I_0$. Suppose that $(\gamma_n : n \in \NN)$ is a sequence of dominant co-weights such that
	\begin{align}
		\label{eq:2a}
		\sup_{n \in \NN}
		\sprod{\gamma_n}{\alpha} &< \infty, \quad\text{for all } \alpha \in \Phi_J^+,\\
	\intertext{and}
		\label{eq:2b}
		\lim_{n \to \infty}
		\sprod{\gamma_n}{\alpha} &= +\infty, \quad\text{for all } \alpha \in \Phi^+ \setminus \Phi^+_J.
	\end{align}
	Then there is a positive constant $a_J$, such that
	\begin{equation}
		\label{eq:1}
		\Xi(\gamma_n) = \bigg(\prod_{\alpha \in \Phi^{++} \setminus \Phi_J} \sprod{\gamma_n}{\alpha\spcheck}
		\bigg)
		\Xi_{J}(\gamma_n) \big(a_J + o(1)\big).
	\end{equation}
\end{proposition}

\begin{proof}
	The proof is in two steps. 

	\vspace*{2ex}
	\noindent
	{\bf Step 1.}
	Let us first prove \eqref{eq:1} assuming that the sequence $(\gamma_n : n \in \NN)$ satisfies 
	$\lim_{n \to \infty} \norm{\gamma_n} = +\infty$, and
	\begin{equation}
		\label{eq:3}
		\liminf_{n \to \infty} \left\langle \frac{\gamma_n}{\norm{\gamma_n}}, \alpha_j \right\rangle =0,
		\qquad\text{if and only if } j \in J,
	\end{equation}
	instead of \eqref{eq:2a} and \eqref{eq:2b}. Let $P_J$ and $Q_J$ be the projections defined by \eqref{eq:77}. 
	We set
	\begin{equation}
		\label{eq:49}
		\tilde{a}_J = \bigg\{\prod_{\alpha \in \Phi^{++}_J} \sprod{\nabla}{\alpha\spcheck}\bigg\}
		\big\{\Delta_J(\theta)\big\}_{\theta = 0}.
	\end{equation}
	Since for each $\lambda \in P$,
	\[
		\sprod{\theta}{\lambda + \rho} = 
		\sprod{\theta}{P_J(\lambda+\rho)} + \sprod{\theta}{Q_J(\lambda+\rho)}
	\]
	by Lemma \ref{lem:7}, $\Xi(\lambda)$ is equal to the sum over $A \subseteq \Phi^{++}$ of terms
	\begin{equation}
		\label{eq:17}
		\frac{1}{\tilde{a}_{I_0}}
		\prod_{\alpha \in A} \sprod{Q_J(\lambda+\rho)}{\alpha\spcheck}
		\bigg\{
		\prod_{\alpha \in \Phi^{++} \setminus A} \sprod{\nabla}{\alpha\spcheck}
		\bigg\}
		\bigg\{
		e^{\sprod{\theta}{P_J(\lambda+\rho)}} \bfb(\theta)
		\bigg\}_{\theta = 0}.
	\end{equation}
	Observe that \eqref{eq:17} equals zero if $A \cap \Phi^{++}_J \neq \varnothing$.
	Moreover, if $A \subseteq \Phi^{++}\setminus \Phi^{++}_J$, then the term \eqref{eq:17} is
	$\calO\big(\norm{\lambda}^{\abs{A}}\big)$.
	Hence, taking $\lambda = \gamma_n$, we get
	\[
		\lim_{\theta \to 0}
		\Xi(\gamma_n)
		=
		\frac{1}{\tilde{a}_{I_0}}
		\bigg(
		\prod_{\alpha \in \Phi^{++}\setminus\Phi_J} \sprod{\gamma_n}{\alpha\spcheck}
		\bigg)
		\bigg\{
		\prod_{\alpha \in \Phi^{++}_J} \sprod{\nabla}{\alpha\spcheck}
		\bigg\}
		\left\{
		e^{\sprod{\theta}{P_J(\gamma_n+\rho)}} \bfb(\theta)
		\right\}_{\theta = 0}
		\big(1 + o(1)\big).
	\]
	Since $\bfb / \bfb_J$ is $W_J$-invariant, for each $\lambda \in P$, we have
	\begin{align*}
		\frac{\bfb(0)}{\bfb_J(0)}
		\Xi_J(\lambda)
		&=
		\lim_{\theta \to 0}
		\frac{1}{|W_J|} \sum_{w \in W_J}
		e^{\sprod{w . \theta}{P_J(\lambda)}} \bfc_J(w . \theta) \frac{\bfb(w . \theta)}{\bfb_J(w . \theta)} \\
		&=
		\lim_{\theta \to 0}
		\frac{1}{|W_J|} \sum_{w \in W_J}
		e^{\sprod{w . \theta}{P_J(\lambda+\rho)}} \bfb(w . \theta) 
		\frac{1}{\Delta_J(w . \theta)} \\
		&=
		\lim_{\theta \to 0}
		\frac{1}{\Delta_J(\theta)}
		\frac{1}{|W_J|} \sum_{w \in W_J}
		(-1)^{\ell(w)}
		e^{\sprod{w . \theta}{P_J(\lambda+\rho)}} \bfb(w . \theta) \\
		&=
		\frac{1}{\tilde{a}_J}
		\bigg\{
		\prod_{\alpha \in \Phi^{++}_J} \sprod{\nabla}{\alpha\spcheck}
		\bigg\}
		\left\{
		e^{\sprod{\theta}{P_J(\lambda+\rho)}} \bfb(\theta)
		\right\}_{\theta = 0}
	\end{align*}
	where the last equality follows by multiple applications of L'H\^opital's rule. Hence, we obtain
	\[
		\lim_{\theta \to 0}
		\Xi(\gamma_n)
		=
		\frac{\tilde{a}_J}{\tilde{a}_{I_0}}
		\cdot
		\frac{\bfb(0)}{\bfb_J(0)}
		\bigg(
		\prod_{\alpha \in \Phi^{++}\setminus\Phi_J} \sprod{\gamma_n}{\alpha\spcheck}
		\bigg)
		\Xi_J(\gamma_n) \big(1 + o(1)\big)
	\]
	which completes the proof of the first step.

	\vspace*{2ex}
	\noindent
	{\bf Step 2.}
	Let $(\gamma_n : n \in \NN)$ be a sequence satisfying \eqref{eq:2a} and \eqref{eq:2b}. The proof is by
	induction over the rank. If $r = 2$, the conclusion easily follows by Step 1. Suppose that Proposition \ref{prop:1} 
	holds true for all root systems of rank smaller than $r$. We define
	\[
		J_1 = \left\{j \in I_0 : \liminf_{n \to \infty} 
		\left\langle \frac{\gamma_n}{\norm{\gamma_n}}, \alpha_j \right\rangle 
		= 0 \right\}.
	\]
	Notice that $J \subseteq J_1$. In view of Step 1, we have
	\begin{equation}
		\label{eq:4}
		\Xi(\gamma_n) =
		\frac{\tilde{a}_{J_1}}{\tilde{a}_{I_0}} \cdot \frac{\bfb(0)}{\bfb_{J_1}(0)}
		\bigg(
		\prod_{\alpha \in \Phi^{++} \setminus \Phi_{J_1}} \sprod{\gamma_n}{\alpha\spcheck}
		\bigg)
        \Xi_{J_1}(\gamma_n) \big(1 + o(1)\big).
	\end{equation}
	This completes the proof in the case $J = J_1$. If $J \subsetneq J_1$, we consider $J_1$ in
	place of $I_0$. By the inductive hypothesis we have
	\[
		\Xi_{J_1}(\gamma_n) =
		\frac{\tilde{a}_{J_1}}{\tilde{a}_{J_1}} \cdot \frac{\bfb(0)}{\bfb_{J}(0)}
		\bigg(
		\prod_{\alpha \in \Phi^{++}_{J_1} \setminus \Phi_J}
		\sprod{\gamma_n}{\alpha\spcheck}
		\bigg)
       	\Xi_{J}(\gamma_n) \frac{\tilde{a}_J \bfb_{J_1}(0)}{\tilde{a}_{J_1} \bfb_J(0)} \big(1 + o(1)\big),
	\]
	which together with \eqref{eq:4} finishes the proof.
\end{proof}

\subsection{Uniqueness of limit functions}
\label{sec:15}
We introduce here partial ground state spherical functions attached to subsets of simple roots. 
These functions will appear naturally when constructing the Martin boundary. This fits well with the fact that fa\c cades
at infinity (as defined in Section \ref{sec:3.3}) can be used to describe Martin compactifications. 
The partial ground spherical functions are combined with partial horospherical functions: the presence of the latter factors
is explained by the fact that the corresponding affine buildings of smaller rank lie at infinity. The difference between the
analytic behaviors of the factors enables us to provide a precise parametrization of the functions in the Martin boundaries. 

Let us define the $J$-\emph{ground state} spherical function as
\[
	\bfPhi_J(\lambda) = \frac{|W_J|}{W_J(q^{-1})} \chi_J^{-\frac{1}{2}} \big(\lambda \big) \Xi_J(\lambda),
	\qquad \lambda \in P^+.
\]
\begin{theorem}
	\label{thm:6}
	Let $J, J' \subsetneq I_0$, $\omega, \omega' \in \Omega$, $y \in [o, \omega]$, $y' \in [o, \omega']$.
		If for all $x \in V_g$, we have
	\begin{equation}
		\label{eq:54}
		\frac{\bfPhi_{J}(\sigma(x, y))}{\bfPhi_{J}(\sigma(o, y))}
		\chi^{\frac{1}{2}}\big(Q_J h(o, x; \omega)\big)
		=
		\frac{\bfPhi_{J'}(\sigma(x, y'))}
		{\bfPhi_{J'}(\sigma(o, y'))}
		\chi^{\frac{1}{2}}\big(Q_{J'} h(o, x; \omega')\big)
	\end{equation}
	then $J' = J$, $\omega' \sim_J \omega$, and $\pi_{F}(y) = \pi_{F}(y')$ where $\pi_F : \scrX \to \scrX(F)$ is the projection 
	to the fa\c cade at infinity $\scrX(F)$, and $F$ is the spherical facet at infinity 
	corresponding to the residue $\res_J(\omega)$. 
\end{theorem}
Note that the ratio of values of partial spherical ground functions above corresponds to the fact that, in the group case,
the function space used to define Martin compactifications is acted upon by a projective action (see \cite[p. 101, paragraph
before Proposition 6.4]{gjt}). The corresponding uniqueness statement in the case of symmetric spaces is
\cite[Theorem 7.22]{gjt}: this result is more analytic and group-theoretic in nature since it establishes the uniqueness of
eigenfunctions of Laplace operators satisfying some invariance under the action of a well-chosen unipotent subgroup.
\begin{proof}
	In view of the cocycle relation, for each $x, x' \in V_s$, we have
	\begin{align*}
		&
		\frac{\bfPhi_{J}(\sigma(x', y))}{\bfPhi_{J}(\sigma(o, y))}
		\chi^{\frac{1}{2}}\big(Q_J h(o, x'; \omega)\big) \\
		&\qquad=
		\frac{\bfPhi_{J}(\sigma(x', y))}{\bfPhi_{J}(\sigma(x, y))}
		\chi^{\frac{1}{2}}\big(Q_J h(x, x'; \omega)\big)
		\frac{\bfPhi_{J}(\sigma(x, y))}{\bfPhi_{J}(\sigma(o, y))}
		\chi^{\frac{1}{2}}\big(Q_J h(o, x; \omega)\big).
	\end{align*}
	Consequently, \eqref{eq:54} implies that for all $x, x' \in V_s$,
	\begin{equation}
		\label{eq:54'}
		\frac{\bfPhi_{J}(\sigma(x', y))}{\bfPhi_{J}(\sigma(x, y))}
		\chi^{\frac{1}{2}}\big(Q_J h(x, x'; \omega)\big)
		=
		\frac{\bfPhi_{J'}(\sigma(x', y'))}{\bfPhi_{J'}(\sigma(x, y'))}
		\chi^{\frac{1}{2}}\big(Q_{J'} h(x, x'; \omega')\big).
	\end{equation}
	Let $\scrA$ be an apartment containing $\omega$ and $\omega'$ in its boundary. Let $\tilde{\omega}$ be opposite to 
	$\omega$ such that $[\omega, \tilde{\omega}] = \scrA$. We select $x \in V_s(\scrA) \cap [y', \tilde{\omega}] \cap 
	[y, \tilde{\omega}]$. For $\lambda \in P^+$ and $n \in \NN$, by $x_{\lambda, n}$ we denote the unique point in
	$[x, \tilde{\omega}]$ such that $\sigma(x_{\lambda; n}, x) = n \lambda$. Then
	\begin{align*}
		\sigma(x_{\lambda; n}, y)
		&=
		\sigma(x_{\lambda; n}, x) + \sigma(x, y) \\
		&
		=
		n\lambda + \sigma(x, y),
	\end{align*}
	and
	\[
		\sigma(x_{\lambda; n}, y') = n \lambda + \sigma(x, y').
	\]
	Next, there is $w \in W$ such that
	\[
		h(x, x_{\lambda; n}; \omega) = - n \lambda,
		\quad\text{and}\quad
		h(x, x_{\lambda; n}; \omega') = -n w.\lambda.
	\]
	Therefore taking $x = x_{\lambda; n}$ in \eqref{eq:54'}, we obtain
	\begin{equation}
		\label{eq:45}
		\frac{\bfPhi_J(\sigma(x, y)+n\lambda)}{\bfPhi_J(\sigma(x, y))}
		\chi^{\frac{1}{2}}\big(n Q_J \lambda\big)
		=
		\frac{\bfPhi_{J'}(\sigma(x, y')+n\lambda)}{\bfPhi_{J'}(\sigma(x, y'))}
		\chi^{\frac{1}{2}}\big(n Q_{J'} w. \lambda\big).
	\end{equation}
	Suppose, contrarily to our claim, that there is $j \in J \setminus J'$. Select
	$\lambda = \lambda_j$. Then the left hand-side of \eqref{eq:45} takes the form
	\[
		\frac{\Xi_J(\sigma(x, y) + n \lambda_j)}{\Xi_J(\sigma(x, y))}
		\chi^{\frac{1}{2}} \big(-n \lambda_j\big)
	\]
	which has a factor that is a non-trivial polynomial in $n$, while the right hand-side equals
	\[
		\chi^{\frac{1}{2}}\big(-n Q_{J'} w. \lambda_j \big).
	\]
	Therefore, \eqref{eq:45} cannot be satisfied for $n$ sufficiently large. Hence $J = J'$ and \eqref{eq:45} takes
	the form
	\begin{equation}
		\label{eq:46}
		\frac{\Xi_J(\sigma(x, y) + n \lambda)}{\Xi_J(\sigma(x, y))}
		\cdot
		\frac{\Xi_J(\sigma(x, y'))}{\Xi_J(\sigma(x, y')+n\lambda)}
		=
		\chi^{\frac{1}{2}}\big(n Q_J(\lambda - w^{-1}. \lambda)\big).
	\end{equation}
	However, \eqref{eq:46} cannot be satisfied unless both sides are constant sequences for $n$ sufficiently large. 
	Checking the right hand-side of \eqref{eq:46} we conclude that it has to be constant equal to $1$ for $n \geqslant 1$.
	Since for each $\lambda \in P^+$,
	\[
		\chi(Q_J \lambda) = e^{\sprod{\eta}{Q_J \lambda}},
	\]
	we conclude that 
	\begin{equation}
		\label{eq:47}
		\sprod{\eta}{Q_J \lambda} = \sprod{\eta}{Q_J w^{-1}. \lambda}.
	\end{equation}
	First, let us show that \eqref{eq:47} entails that $w \in W_J$. Let $k = \ell(w)$ and suppose that 
	$w = w_k= w_{k-1} r_{\beta_k}$ is such that $\ell(w_k) > \ell(w_{k-1})$. We have
	\begin{align*}
		\sprod{\eta}{Q_J w_k^{-1} . \lambda} = \sprod{Q_J \eta}{w_k^{-1}. \lambda} 
		&=
		\sprod{Q_J \eta}{r_{\beta_k} w_{k-1}^{-1}. \lambda} \\
		&=
		\sprod{r_{\beta_k} Q_J \eta}{w_{k-1}^{-1}. \lambda} \\
		&=
		\sprod{Q_J \eta}{w_{k-1}^{-1}. \lambda} - \sprod{Q_J \eta}{\beta_k\spcheck} \sprod{\beta_k}{w_{k-1}^{-1}. \lambda}.
	\end{align*}
	Using induction on $k$, we get
	\begin{equation}
		\label{eq:59}
		\sprod{\eta}{Q_J w_k^{-1} . \lambda}
		=
		\sprod{\eta}{Q_J \lambda}
		-
		\sum_{i = 1}^k  \sprod{Q_J \eta}{\beta_i \spcheck} \sprod{\beta_i}{w_{i-1}^{-1}. \lambda}.
	\end{equation}
	Therefore,
	\begin{equation}
		\label{eq:14}
		\sum_{i = 1}^k  \sprod{Q_J \eta}{\beta_i \spcheck} \sprod{\beta_i} {w_{i-1}^{-1}. \lambda} = 0.
	\end{equation}
	For each $i \in \{1, 2, \ldots, k\}$, $w_i = w_{i-1} r_{\beta_i}$ and $\ell(w_i) > \ell(w_{i-1})$, hence
	by \cite[Section 5.7, Proposition]{Humphreys1990} we have
	\begin{equation}
		\label{eq:31}
		w_{i-1} . \beta_i \in \Phi^+.
	\end{equation}
	Moreover, in view of \eqref{eq:32}, $Q_J \eta \in S_0$, and $\sprod{Q_J \eta}{\alpha_j} = 0$, for all 
	$j \in J$. Hence, \eqref{eq:14} implies that $\beta_i \in \Phi_J^+$, for all $i \in \{1, \ldots, k\}$. Consequently, 
	$w \in W_J$ and so $\omega' \sim_J \omega$.

	To complete the proof, we need to show that $\pi_F(y) = \pi_F(y')$. Since $\omega$ and $\omega'$ are not opposite
	there is an apartment $\scrA$ containing both $\omega$ and $\omega'$ on its boundary, such that $y \in V_s(\scrA)$.
	Select $\omega'' \in \Omega$ such that $[y, \omega''] \subset \scrA$ and $y \in [y', \omega'']$. For $\lambda \in P^+$ 
	and $n \in \NN$, we take $x_{\lambda; n} \in [y, \omega'']$ such that
	\[
		\sigma(x_{\lambda; n}, y) = n\lambda.
	\]
	Then
	\[
		\sigma(x_{\lambda; n}, y') = n \lambda + \sigma(y, y').
	\]
	By taking $x = y$ and $x' = x_{\lambda; n}$ in \eqref{eq:46}, we obtain
	\begin{equation}
		\label{eq:44}
		\Xi_J(n \lambda) = 
		\frac{\Xi_{J}(\sigma(y, y') + n \lambda)}{\Xi_{J}(\sigma(y, y'))},
	\end{equation}
	for all $n \in \NN$ and $\lambda \in P^+$. Now, by Lemma \ref{lem:7}, we get
	\begin{align*}
		\tilde{a}_J \Xi_J(\sigma(y, y') + n \lambda)
		&=
		\bfb_J(0) 
		\prod_{\alpha \in \Phi_J^{++}} \sprod{\sigma(y, y') + \rho_J + n \lambda}{\alpha\spcheck} \\
		&\phantom{=}+
		\sum_{\beta \in \Phi_J^{++}}
		\sprod{\nabla}{\beta\spcheck} \bfb_J(0) \cdot
		\prod_{\stackrel{\alpha \in \Phi_J^{++}}{\alpha \neq \beta}} 
		\sprod{\sigma(y, y') + \rho_J + n \lambda}{\alpha\spcheck} +
		\text{ lower powers of $n$}
	\end{align*}
	where $\tilde{a}_J$ is given by the formula \eqref{eq:49}. Hence,
	\begin{align*}
		\tilde{a}_J \Xi_J(\sigma(y, y') + n \lambda)
		&=
		n^{|\Phi^{++}_J|} \bfb_J(0) \prod_{\alpha \in \Phi_J^{++}} \sprod{\lambda}{\alpha\spcheck} \\
		&\phantom{=}
		+
		n^{|\Phi^{++}_J|-1} \bfb_J(0) \sum_{\beta \in \Phi^{++}_J}
		\sprod{\sigma(y, y') + \rho_J}{\beta\spcheck}
		\prod_{\stackrel{\alpha \in \Phi_J^{++}}{\alpha \neq \beta}}
	    \sprod{\lambda}{\alpha\spcheck} \\
		&\phantom{=}
		+n^{|\Phi^{++}_J|-1}
		\sum_{\beta \in \Phi^{++}_J} 
		\sprod{\nabla}{\beta\spcheck} \bfb_J(0) \cdot 
		\prod_{\stackrel{\alpha \in \Phi_J^{++}}{\alpha \neq \beta}}
		\sprod{\lambda}{\alpha\spcheck}
		+\text{ lower powers of $n$}.
	\end{align*}
	For $\lambda = \rho$, by comparing the leading terms in \eqref{eq:44}, we immediately get
	\[
		\Xi_J(\sigma(y, y')) = 1.
	\]
	Now, the equality of the following terms implies that for all $\lambda \in P^+$, we have
	\begin{equation}
		\label{eq:48}
		\sum_{\beta \in \Phi^{++}_J} \sprod{\sigma(y, y')}{\beta\spcheck}
		\prod_{\stackrel{\alpha \in \Phi_J^{++}}{\alpha \neq \beta}}
		\sprod{\lambda}{\alpha\spcheck}
		=
		0.
	\end{equation}
	We notice that for each $j \in J$, if $\alpha \in \Phi^{++}_J$ satisfies $\sprod{\alpha}{\rho-\lambda_j} = 0$, 
	then $\alpha = \alpha_j$. Thus, taking $\lambda = \rho - \lambda_j$ in \eqref{eq:48}, we obtain
	\[
		\sprod{\sigma(y, y')}{\alpha_j\spcheck} 
		\prod_{\stackrel{\alpha \in \Phi_J^{++}}{\alpha \neq \alpha_j}}
		\sprod{\rho-\lambda_j}{\alpha\spcheck}
		= 0.
	\]
	Hence, for each $\alpha \in \Phi_J^+$,
	\[
		\sprod{ \sigma_J(\pi_F(y), \pi_F(y'))}{\alpha} 
		\leq
		\sprod{ \sigma(y, y') }{\alpha} = 0
	\]
	which implies that $\pi_F(y) = \pi_F(y')$ and the theorem follows.
\end{proof}

Here is a variant of the previous uniqueness result which takes into account an additional radial parameter; it will be
used to describe the Martin boundary above the bottom of the spectrum. Recall that the sector face $[o, F_J]$ is the subset
of $\frakA$ consisting of the vectors such that $\sprod{x}{\alpha} \geqslant 0$ for all $\alpha \in \Phi^+$ and
$\sprod{x}{\alpha} = 0$ for all $\alpha \in \Phi_J$. 
\begin{theorem}
	\label{thm:2}
	Let $J, J' \subsetneq I_0$, $\omega, \omega' \in \Omega$, $s, s' \in S_0$, and
	$y \in [o, \omega]$, $y' \in [o, \omega']$. If for all $x \in V_g$, we have
	\begin{align*}
		&
		\frac{\bfPhi_{J}(\sigma(x, y))}{\bfPhi_{J}(\sigma(o, y))}
		\chi^{\frac{1}{2}}\big(Q_J h(o, x; \omega)\big)
		e^{\sprod{s}{Q_J h(o, x; \omega)}} \\
		&\qquad\qquad=
		\frac{\bfPhi_{J'}(\sigma(x, y'))}
		{\bfPhi_{J'}(\sigma(o, y'))}
		\chi^{\frac{1}{2}}\big(Q_J h(o, x; \omega')\big)
		e^{\sprod{s'}{Q_J h(o, x; \omega')}}
	\end{align*}
	then $J' = J$, $\omega' \sim_J \omega$, $Q_J s = Q_J s'$ and $\pi_F(y) = \pi_F(y')$ where $\pi_F : \scrX \to \scrX(F)$ 
	is the projection to the fa\c cade at infinity $\scrX(F)$, and $F$ is the spherical facet at infinity 
	corresponding to the residue $\res_J(\omega)$. 
\end{theorem}
\begin{proof}
	By the same line of reasoning as in the proof of Theorem \ref{thm:6}, we can show that $J = J'$.
	Moreover, there is $w \in W$ such that for all $\lambda \in P^+$,
	\begin{equation}
		\label{eq:34}
		\chi^{\frac{1}{2}}\big(Q_J \lambda \big)
		e^{\sprod{s}{Q_J \lambda}}
		=
		\chi^{\frac{1}{2}}\big(Q_J w^{-1}.\lambda \big)
		e^{\sprod{s'}{Q_J w^{-1}.\lambda}}.
	\end{equation}
	Therefore, there is $w \in W$ such that
	\[
		Q_J(s + \eta) = w. Q_J(s' + \eta).
	\]
	We are going to conclude that $Q_J s = Q_J s'$ and $w \in W_J$. Using the notation from the proof of Theorem \ref{thm:6},
	we write
	\begin{align*}
		w_i.Q_J(s'+\eta) 
		&= w_{i-1} r_{\beta_i} Q_J(s'+\eta) \\
		&= w_{i-1} Q_J(s'+\eta) - \sprod{Q_J(s'+\eta)}{\beta_i\spcheck} w_{i-1}.\beta_i.
	\end{align*}
	Hence, we get
	\[
		w.Q_J(s' + \eta) = 
		Q_J(s'+\eta) 
		-\sum_{i = 1}^k
		\sprod{Q_J(s'+\eta)}{\beta_i\spcheck} w_{i-1}.\beta_i.
	\]
	Therefore, by \eqref{eq:34} we obtain
	\[
		Q_J(s-s') = 
		-\sum_{i = 1}^k \sprod{Q_J(s'+\eta)}{\beta_i\spcheck} w_{i-1}\beta_i.
	\]
	Since $s' + \eta \in S_0$, in view of \eqref{eq:31} we conclude that $Q_J(s-s')$ is a non-negative combination
	of positive roots. Since we can swap $J$, $\omega$, $s$ and $y$ with $J'$, $\omega'$, $s'$ and $y'$,
	respectively, we deduce that $Q_J(s'-s)$ is also non-negative combination of positive roots. Consequently, $Q_J s = Q_J s'$,
	and
	\[
		\sum_{i = 1}^k \sprod{Q_J(s'+\eta)}{\beta_i\spcheck} w_{i-1}\beta_i = 0.
	\]
	Since $Q_J \eta \in S_0$ and $\sprod{Q_J \eta}{\alpha_j} = 0$ for all $j \in J$, we conclude that
	$\beta_i \in \Phi_J^+$ for all $i \in \{1, 2, \ldots, k\}$. Consequently, $w \in W_J$ and $\omega' \sim_J \omega$.

	Now, by the same line of reasoning as in the proof Theorem \ref{thm:6} we show that $\pi_F(y) = \pi_F(y')$,
	which completes the proof.
\end{proof}

\subsection{Martin compactification for $\zeta = \varrho$}
\label{ss - Martin bottom}
In this section we describe the Martin compactification at the bottom of the spectrum corresponding to the isotropic
finite range random walk on good vertices of the building $\scrX$ chosen in Section \ref{ss - Martin embeddings}.
We set $\zeta = \varrho$ where $\varrho$ is defined in \eqref{eq:56}.

As for Furstenberg compactifications, it is convenient to use the notions and terminology introduced in Section
\ref{sec:6}, including fa\c cades indexed by spherical facets at infinity. We again use the notation of Section
\ref{sec:3.3}. 
\begin{theorem}
	\label{thm:4}
	Suppose that $(y_n)$ is an $(\omega, J, c)$-core sequence. We denote by $F$ the spherical facet at infinity corresponding
	to the residue $\res_J(\omega)$ and by $\pi_F : \scrX \to \scrX(F)$ the projection to the fa\c cade at infinity
	$\scrX(F)$. Then for all $x \in V_g$, we have
	\begin{equation}
		\label{eq:30}
		\lim_{n \to \infty}
		K_\varrho(x, y_n)
		=
		\frac{\bfPhi_J(\sigma_{\scrX(F)}(x_F, y_F))}{\bfPhi_J(\sigma_{\scrX(F)}(o_F, y_F))}
		\chi^{\frac{1}{2}}\big(Q_J h(o, x; \omega)\big)
	\end{equation}
	where $x_F = \pi_F(x)$, $y_F = \lim_{n \to \infty} \pi_F(y_n)$ (limit of a constant sequence). 
	If $(y_n')$ is an $(\omega', J', c')$-core sequence such that
	$(K_\varrho(\:\cdot\:, y_n'))$ converges to the same limit, then $J' = J$, $\omega' \sim_J \omega$, and $c' = c$.
\end{theorem}

\begin{proof}
	Let $\gamma_n = \sigma(o, y_n)$ and $\eta_n = \sigma(x, y_n)$. By \cite[Theorem 6]{tr}, we have
	\[
		G_\varrho(o, y_n) = P_{\gamma_n}(0) \norm{\gamma_n}^{-r - 2 \abs{\Phi^{++}} + 2} \big(D_0 + o (1)\big)
	\]
	and
	\[
		G_\varrho(x, y_n) = P_{\eta_n}(0) \norm{\eta_n}^{-r - 2 \abs{\Phi^{++}} + 2} \big(D_0 + o (1)\big)
	\]
	where $D_0$ is a certain positive constant. Hence, by \eqref{eq:41},
	\[
		K_\varrho(x, y_n) = \frac{\calP(\eta_n)}{\calP(\gamma_n)} 
		\bigg(\frac{\norm{\gamma_n}}{\norm{\eta_n}} \bigg)^{r + 2 \abs{\Phi^{++}}-2} 
		\chi^{-\frac{1}{2}}\big(\eta_n-\gamma_n\big)
		\big(1+o(1)\big).
	\]
	By Lemma \ref{lem:8}, for all $\alpha \in \Phi^+$, we have
	\[
		|\sprod{\gamma_n}{\alpha} - \sprod{\eta_n}{\alpha}| \leqslant
		|\gamma_n - \eta_n||\alpha| \leqslant |\sigma(o, x)| |\alpha|.
	\]
	Consequently, for all $\alpha \in \Phi^+ \setminus \Phi_J$,
	\begin{equation}
		\label{eq:55}
		\frac{\sprod{\gamma_n}{\alpha}}{\sprod{\eta_n}{\alpha}} = 1 + o(1).
	\end{equation}
	Moreover, by the triangle inequality and Lemma \ref{lem:8}
	\[
		\big|\norm{\gamma_n} - \norm{\eta_n}\big| \leqslant |\gamma_n - \eta_n| \leqslant |\sigma(o, x)|,
	\]
	thus
	\begin{equation}
		\label{eq:57}
		\norm{\gamma_n} = \norm{\eta_n}(1+o(1)).
	\end{equation}
	In view of Lemma \ref{lem:4}, by taking a subsequence, we can assume that for each $\alpha \in \Phi_J$,
	\[
		\sprod{\sigma_{\scrX(F)}(x_F, y_F)}{\alpha} = \lim_{n \to \infty} \sprod{\eta_n}{\alpha},
		\qquad\text{and}\qquad
		\sprod{\sigma_{\scrX(F)}(o_F, y_F)}{\alpha} = \lim_{n \to \infty} \sprod{\gamma_n}{\alpha}.
	\]
	Therefore, by Proposition \ref{prop:1},
	\begin{align*}
		\frac{\calP(\eta_n)}{\calP(\gamma_n)} 
		&= \frac{\calP_J(\eta_n)}{\calP_J(\gamma_n)}
		\bigg(
		\prod_{\alpha \in \Phi^{++}\setminus \Phi_J} 
		\frac{\sprod{\eta_n}{\alpha\spcheck}}{\sprod{\gamma_n}{\alpha\spcheck}}
		\bigg)
		\big(1 + o(1)\big) \\
		&=
		\frac{\calP_J(\sigma_{\scrX(F)}(x_F, y_F))}{\calP_J(\sigma_{\scrX(F)}(o_F, y_F))}
		\big(1+ o(1)\big).
	\end{align*}
	Finally, we invoke Lemma \ref{lem:9} and Lemma \ref{lem:6}\eqref{en:7:2} to get
	\begin{align*}
		\chi^{-\frac{1}{2}}(\eta_n - \gamma_n)
		&=
		\frac{\chi^{-\frac{1}{2}}(P_J \eta_n)}{\chi^{-\frac{1}{2}}(P_J \gamma_n)}
		\chi^{-\frac{1}{2}} \big(Q_J(\eta_n - \gamma_n)\big) \\
		&=
		\frac{\chi^{-\frac{1}{2}}(\sigma_{\scrX(F)}(x_F, y_F))}{\chi^{-\frac{1}{2}}(\sigma_{\scrX(F)}(o_F, y_F))}
		\chi^{\frac{1}{2}} \big(Q_J(h(o, x; \omega))\big)
		(1+o(1))
	\end{align*}
	which completes the proof of \eqref{eq:30}. The uniqueness part of the theorem follows from Theorem 
	\ref{thm:6}.
\end{proof}

At this stage, we can prove Theorem \ref{th - Martin}\eqref{en:8:1} in the case when $\zeta=\varrho$. 
Recall that when the finite root system associated with the building $\scrX$ is reduced, all special vertices are good. 

\begin{theorem}
	\label{thm:7}
	Let $\scrX$ be a thick regular locally finite affine building. 
	Then for any isotropic irreducible finite range random walk on $\scrX$, the Martin compactification 
	$\clr{\scrX}_{M, \varrho}$ is $\Aut(\scrX)$-equivariantly isomorphic to the Furstenberg (measure-theoretic) 
	or the Caprace--L\'ecureux (combinatorial) compactification of the set $V_g$ of good vertices.
	Moreover, if there exists a locally compact closed subgroup of
	$\Aut(\scrX)$, say $G$, acting strongly transitively and type-preserving on the building $\scrX$, then the Guivarc'h 
	(group-theoretic) compactification $\clr{\scrX}_G$ of the set $V_g$ of good vertices is $G$-isomorphic to the Martin 
	compactification $\clr{\scrX}_{M, \varrho}$.
\end{theorem}
\begin{proof}
	The proof mainly consists in putting together previous results. In all the considered compactifications, core sequences 
	converge: for the Martin compactification it follows from Theorem \ref{thm:4}, for the Furstenberg compactification it 
	follows from Theorem \ref{thm:5}, and for the Caprace--L\'ecureux compactification it follows from Theorem \ref{thm:12}.
	As a consequence, combining a standard topological argument (see \emph{e.g.}~the domination criterion given by 
	\cite[Lemma 3.28]{gjt} in both directions between two compactifications) and the uniqueness assertions in the previous 
	theorems provide the identifications between the first three compactifications. Moreover, since the full automorphism 
	$\Aut(\scrX)$ acts continuously on each compactification and permutes the set of core sequences, the identifications are 
	equivariant, whatever the size of $\Aut(\scrX)$. At last, for the identification with the Guivarc'h compactification when 
	$\Aut(\scrX)$ acts strongly transitively and type-preserving on $\scrX$, it remains to apply 
	Theorem \cite[Theorem II]{Caprace2011}. 
\end{proof}
	
Let us remark that the automorphism group of an irreducible affine building always acts strongly transitively on
$\scrX$ if the rank of the affine building is at least $4$, see \cite[page 274]{tits}.

\subsection{Martin compactification for $\zeta > \varrho$}
\label{sec:7.3}
In this section we describe the Martin compactification above the bottom of the spectrum corresponding to the isotropic
finite range random walk on good vertices of the building $\scrX$ chosen in Section \ref{ss - Martin embeddings}.
Recall the definition of $\varrho$ given in \eqref{eq:56} and $\kappa$ in \eqref{eq:68}. 
This is the place where we have to use angular core sequences of good vertices defined at the end of Section 
\ref{sec:11}. In order to describe the Martin kernel, let us define
\[
	\scrC = \big\{x \in \frakA : \kappa(x) = \zeta \varrho^{-1} \big\}.
\]
Notice that $\scrC$ is the boundary of a convex body such that for each $x \in \scrC$, the gradient $\nabla \kappa(x)$
is well defined. Hence, for each $\theta \in \Ss^{r-1}$, a unit sphere in $\frakA$ centered at the origin, there is a 
unique point $s_\theta \in \scrC$, such that
\begin{equation}
	\label{eq:23}
	\nabla \kappa(s_\theta) = \norm{\nabla \kappa(s_\theta)} \theta.
\end{equation}
Moreover, if $\theta \in \Ss^{r-1}_+$ then $s_\theta \in S_0$. In the next theorem we describe the Martin kernel for
$\zeta > \varrho$.
\begin{theorem}
	\label{thm:10}
	Let $(y_n)$ be an angular core $(\omega, J, c, \theta)$-sequence. We denote by $F$ the spherical facet at infinity
	corresponding to the residue $\res_J(\omega)$ and by $\pi_F : \scrX \to \scrX(F)$ the projection to the fa\c cade at 
	infinity $\scrX(F)$. Then for all $x \in V_g$, we have 
	\begin{equation}
		\label{eq:33}
		\lim_{n \to \infty} K_\zeta(x, y_n)
		=
		\frac{\bfPhi_J(\sigma_{\scrX(F)}(x_F, y_F))}{\bfPhi_J(\sigma_{\scrX(F)}(o_F, y_F))}
		\chi^{\frac{1}{2}}\big(Q_J h(o, x; \omega)\big)
		e^{\sprod{s_\theta}{Q_J h(o, x; \omega)}}
	\end{equation}
	where $x_F = \pi_F(x)$, $y_F = \lim_{n \to \infty} \pi_F(y_n)$ (limit of a constant sequence). 
	If $(y'_n : n \in \NN)$ is an angular core
	$(\omega', J', c', \theta')$-sequence such that $(K_\zeta(\: \cdot \:, y_n'))$ converges to the same limit then $J' = J$,
	$\omega' \sim_J \omega$, $c' = c$ and $Q_J \theta' = Q_J \theta$.
\end{theorem}

\begin{proof}
	First, let us introduce some notation. For $s \in \mathfrak{a}$, we define a quadratic form on $\mathfrak{a}$,
	\[
		B_s(y, y) = \frac{1}{2} \sum_{v, v' \in \calV} \frac{c_v e^{\sprod{s}{v}}}{\kappa(s)} \cdot
		\frac{c_{v'} e^{\sprod{s}{v'}}}{\kappa(s)} \sprod{y}{v-v'}^2, \qquad y \in \mathfrak{a},
	\]
	where $\kappa$ is given by \eqref{eq:68}. Let
	\[
		J_1 = \left\{j \in I_0 : \sprod{u}{\alpha_j} = 0 \right\}.
	\]
	In particular, $J \subseteq J_1 \subsetneq I_0$. For $\theta \in \Ss^{r-1}$, we set
	\[
		\calR(\theta) = \sqrt{2\pi}
		\norm{\nabla \log \kappa(s_\theta) }^{\frac{r-3}{2} + |\Phi^{++}_{J_1}|} 
		\big(B_{s_\theta}(\theta, \theta)\big)^{-\frac{1}{2}} \calQ(s_\theta)
	\]
	where for $s \in \mathfrak{a}$,
	\[
		\calQ(s) = 
		\bigg(\frac{1}{2\pi}\bigg)^r 
		\int_\mathfrak{a} e^{-\frac{1}{2} B_s(z, z)} |\pi_{J_1}(z)|^2 {\: \rm d} z
		\cdot
		\frac{1}{|\mathbf{b}_{J_1}(0)|^2}
		\cdot
		\prod_{\alpha \in \Phi^+\setminus\Phi_{J_1}}
		\frac{1 - \tau_{\alpha/2}^{-1/2} e^{-\sprod{s}{\alpha\spcheck}}}
		{1 - \tau_{\alpha}^{-1} \tau_{\alpha/2}^{-1/2} e^{-\sprod{s}{\alpha\spcheck}}}
	\]
	and
	\[
		\pi_{J_1}(s) = \prod_{\alpha \in \Phi^{++}_{J_1}} \sprod{s}{\alpha\spcheck}.
	\]
	Observe that $\calR(\theta) \neq 0$. Indeed, since
	\[
		\prod_{\alpha \in \Phi^+\setminus\Phi_{J_1}}
		\frac{1 - \tau_{\alpha/2}^{-1/2} e^{-\sprod{s}{\alpha\spcheck}}}
		{1 - \tau_{\alpha}^{-1} \tau_{\alpha/2}^{-1/2} e^{-\sprod{s}{\alpha\spcheck}}}
		=
		\prod_{\alpha\in\Phi^{++}\setminus\Phi_{J_1}}
		\frac{1-e^{-\sprod{s}{\alpha\spcheck}}}
		{\big(1 - \tau_{2\alpha}^{-1} \tau_\alpha^{-1/2} e^{-\sprod{s}{\alpha\spcheck}/2 }\big)
		\big(1+\tau_\alpha^{-1/2} e^{-\sprod{s}{\alpha\spcheck}/2}\big)},
	\]
	thus the equality $\calR(\theta) = 0$ implies that there is $\alpha \in \Phi^{++}\setminus \Phi_{J_1}$ such that
	$\sprod{s_\theta}{\alpha\spcheck} = 0$. Hence, by \eqref{eq:23}, we would have $\sprod{\theta}{\alpha\spcheck} = 0$, 
	which is impossible.

	Let $\gamma_n = \sigma(o, y_n)$ and $\eta_n = \sigma(x, y_n)$. By Lemma \ref{lem:4}, by taking a subsequence,
	we can assume that
	\begin{equation}
		\label{eq:39}
		\lim_{n \to \infty} \sprod{\gamma_n}{\alpha} = \sprod{\sigma_{\scrX(F)}(o_F, y_F)}{\alpha},
		\quad\text{and}\quad
		\lim_{n \to \infty} \sprod{\eta_n}{\alpha} = \sprod{\sigma_{\scrX(F)}(x_F, y_F)}{\alpha},
	\end{equation}
	for all $\alpha \in \Phi_J^+$. We set
	\[
		a_n = \frac{\gamma_n}{\norm{\gamma_n}}, \qquad\text{and}\qquad
		b_n = \frac{\eta_n}{\norm{\eta_n}}.
	\]
	By \eqref{eq:55} and \eqref{eq:57}, we have
	\[
		\lim_{n \to \infty} b_n = \lim_{n \to \infty} a_n = \theta.
	\]
	Since $\calR$ is a continuous function on $\Ss^{r-1}$ and $\calR(\theta) \neq 0$, we obtain
	\[
		\calR(a_n) = \calR(\theta) (1 + o(1)),
		\qquad\text{and}\qquad
		\calR(b_n) = \calR(\theta) (1+ o(1)).
	\]
	Now, by \cite[Theorem 5]{tr}, we get
	\[
		G_\zeta(o, y_n) = \norm{\gamma_n}^{-\frac{r-1}{2} - |\Phi_{J_1}^{++}|} 
		\chi^{-\frac{1}{2}}(\gamma_n)
		\calP_{J_1}(\gamma_n)
		\calR(a_n) e^{-\sprod{s_{a_n}}{\gamma_n}} \big(1 + o(1)\big)
	\]
	and
	\[
		G_\zeta(x, y_n) = \norm{\eta_n}^{-\frac{r-1}{2} - |\Phi_{J_1}^{++}|} 
		\chi^{-\frac{1}{2}}(\eta_n)
		\calP_{J_1}(\eta_n)
		\calR(b_n) e^{-\sprod{s_{b_n}}{\eta_n}} \big(1 + o(1)\big).
	\]
	Hence,
	\[
		K_\zeta(x, y_n) = 
		\bigg(\frac{\norm{\gamma_n}}{\norm{\eta_n}}\bigg)^{\frac{r-1}{2} + |\Phi_{J_1}^{++}|}
		\frac{\calP_{J_1}(\eta_n)}{\calP_{J_1}(\gamma_n)} 
		\chi^{-\frac{1}{2}}(\eta_n- \gamma_n)
		e^{-\sprod{s_{b_n}}{\eta_n} + \sprod{s_{a_n}}{\gamma_n}} \big(1 + o(1)\big).
	\]
	By Proposition \ref{prop:1}, we have
	\[
		\frac{\calP_{J_1}(\eta_n)}{\calP_{J_1}(\gamma_n)} = \frac{\calP_J(\eta_n)}{\calP_J(\gamma_n)}
		\bigg(\prod_{\alpha \in \Phi^{++}_{J_1} \setminus \Phi_J} \frac{\sprod{\eta_n}{\alpha}}
		{\sprod{\gamma_n}{\alpha}} \bigg) \big(1+o(1)\big)
		=
		\frac{\calP_J(\eta_n)}{\calP_J(\gamma_n)}
		\big(1+o(1)\big).
	\]
	Moreover, by Lemmas \ref{lem:9} and \eqref{eq:39},
	\[
		\chi^{-\frac{1}{2}}(\eta_n - \gamma_n)
		=
		\frac{\chi^{-\frac{1}{2}}(\sigma_{\scrX(F)}(x_F, y_F))}{\chi^{-\frac{1}{2}}(\sigma_{\scrX(F)}(o_F, y_F))}
		\chi^{\frac{1}{2}}\big(Q_J h(o, x; \omega)\big)(1+o(1))
	\]
	where $Q_J = {\rm id} - P_J$, and $P_J$ are given by \eqref{eq:77}. Therefore,
	\[
		K_\zeta(x, y_n) = \frac{\bfPhi_J(\sigma_{\scrX(F)}(x_F, y_F))}{\bfPhi_J(\sigma_{\scrX(F)}(o_F, y_F))}
		 \chi^{\frac{1}{2}}\big(Q_J h(o, x; \omega)\big)
		e^{-\sprod{s_{b_n}}{\eta_n} + \sprod{s_{a_n}}{\gamma_n}}
		\big(1+o(1)\big).
	\]
	Next, we show the following claim.
	\begin{claim}
		\label{clm:4}
		\begin{equation}
		\label{eq:28}
		\sprod{s_{a_n}}{\gamma_n} - \sprod{s_{b_n}}{\eta_n} = \sprod{s_u}{Q_J(\gamma_n - \eta_n)}
		+o(1).
	\end{equation}
	\end{claim}
	\noindent
	To see this, we write
	\begin{align}
		\nonumber
		\sprod{s_{a_n}}{\gamma_n} - \sprod{s_{b_n}}{\eta_n} 
		&=
		\norm{\gamma_n} \sprod{s_{a_n}}{a_n} - \norm{\eta_n} \sprod{s_{b_n}}{b_n} \\
		\label{eq:27}
		&=
		\big(\norm{\gamma_n} -  \norm{\eta_n}\big)  \sprod{s_{a_n}}{a_n}
		+ \norm{\eta_n} \big(\sprod{s_{a_n}}{a_n} - \sprod{s_{b_n}}{b_n}\big).
	\end{align}
	We start by considering the second term in \eqref{eq:27}. Let us denote by $\calM$ the interior of the convex hull of 
	$\calV$. For $\xi \in \Ss^{r-1}$, we set $t_0 = \min\{t > 0 : t^{-1} \xi \in \calM\}$ and define a function on 
	$(t_0, \infty)$,
	\begin{equation}
		\label{eq:25}
		\psi_\xi(t) = t \big(\log(\zeta^{-1} \rho) - \phi(t^{-1} \xi)\big)
	\end{equation}
	where 
	\[
		\phi(\delta) = \min\big\{\sprod{x}{\delta} - \log \kappa(x) :  x \in \frakA\big\},
		\qquad
		\delta \in \calM.
	\]
	For the properties of $\phi$, see \cite[Section 2.1]{tr}. The function $\psi_\xi$ attains its unique maximum at
	$t_\xi > t_0$. In particular, $\psi'_\xi(t_\xi) = 0$ and $\psi_\xi(t_\xi) = -\sprod{s_\xi}{\xi}$.
	Thus the gradient of the function $\Ss^{r-1} \ni \xi \mapsto \psi_\xi(t_\xi)$ equals 
	$-\nabla \phi\big(t_{\xi}^{-1} \xi\big) = -s_\xi$. Hence, by the Taylor's formula we obtain
	\begin{equation}
		\label{eq:26}
		\begin{aligned}
		\sprod{s_{a_n}}{a_n} - \sprod{s_{b_n}}{b_n}
		&=
		\psi_{b_n}(b_{v_n}) - \psi_{a_n}(t_{a_n})\\
		&=
		-\sprod{s_{a_n}}{b_n - a_n} + \calO\big(\norm{b_n-a_n}^2\big).
		\end{aligned}
	\end{equation}
	Now, we compute
	\[
		\norm{\eta_n} \big(a_n - b_n\big) = \gamma_n  - \eta_n + a_n \big(\norm{\eta_n} - \norm{\gamma_n}\big).
	\]
	Since by Lemma \ref{lem:8}, $\norm{\eta_n - \gamma_n}$ is bounded and $(a_n)$ approaches $u$, we obtain
	\begin{align*}
		\norm{\eta_n} - \norm{\gamma_n} 
		&= \frac{\norm{\eta_n}^2 - \norm{\gamma_n}^2}{\norm{\eta_n} + \norm{\gamma_n}} \\
		&= \frac{2 \sprod{\eta_n - \gamma_n}{\gamma_n} + \norm{\eta_n - \gamma_n}^2}{\norm{\eta_n} + \norm{\gamma_n}}
		= \sprod{\eta_n - \gamma_n}{u} + o(1).
	\end{align*}
	In particular, $\norm{\eta_n} \norm{a_n - b_n}$ is bounded. Hence,
	\begin{align*}
		\norm{\eta_n} \sprod{s_{a_n}}{a_n - b_n} 
		&= \sprod{s_{a_n}}{\gamma_n - \eta_n} + \sprod{s_{a_n}}{a_n}\big(\norm{\eta_n} - \norm{\gamma_n}\big)
		\\
		&= \sprod{s_u}{\gamma_n - \eta_n} + \sprod{s_{a_n}}{a_n}\big(\norm{\eta_n} - \norm{\gamma_n}\big)
		+ o(1),
	\end{align*}
	which together with \eqref{eq:26} implies that
	\[
		\norm{\eta_n} \left(\sprod{s_{a_n}}{a_n} - \sprod{s_{b_n}}{b_n}\right)
		= \sprod{s_u}{\gamma_n - \eta_n} + \sprod{s_{a_n}}{a_n}\big(\norm{\eta_n} - \norm{\gamma_n}\big)
		+ o(1).
	\]
	Therefore, by \eqref{eq:27}, we obtain
	\begin{align*}
		\sprod{s_{a_n}}{\gamma_n} - \sprod{s_{b_n}}{\eta_n}
		&=
		\sprod{s_\theta}{\gamma_n - \eta_n} + o(1)
	\end{align*}
	proving \eqref{eq:28}.

	Now, Claim \ref{clm:4} together with Lemma \ref{lem:9} implies that
	\begin{align*}
		\sprod{s_{a_n}}{\gamma_n} - \sprod{s_{b_n}}{\eta_n} 
		&=
		\sprod{s_\theta}{\gamma_n - \eta_n} + o(1) \\
		&=
		\sprod{s_\theta}{h(o, x; \omega)} + o(1)
	\end{align*}
	where we have also used $Q_J(s_\theta) = s_\theta$. This establishes the limit \eqref{eq:33}. 
	The uniqueness part of the theorem follows from Theorem \ref{thm:2}.
\end{proof}
	
At this stage, we can prove Theorem \ref{th - Martin}\eqref{en:8:1} in the case when $\zeta>\varrho$. 
Recall that when the finite root system associated with the building $\scrX$ is reduced, all special vertices are good. 
Recall also that thanks to Theorem \ref{thm:7}, the Martin compactification $\clr{\scrX}_{M, \varrho}$ below is equivariantly 
isomorphic to the combinatorial or to the measure-theoretic compactification. 

\begin{theorem}
\label{thm:8}
	Let $\scrX$ be a thick regular locally finite affine building. Then for any isotropic irreducible finite range random walk 
	on $\scrX$ and for any $\zeta > \varrho$, the Martin compactification $\clr{\scrX}_{M,\zeta}$ is $\Aut(\scrX)$-isomorphic 
	to the compactification $\clr{\scrX}_{M, \varrho} \vee \clr{\scrX}_V$, where $\clr{\scrX}_{M, \varrho}$ is the Martin 
	compactification at the bottom of the spectrum and $\clr{\scrX}_V$ is the Gromov compactification.
\end{theorem}
\begin{proof}
	Recall that the join $\clr{\scrX}_{M, \varrho} \vee \clr{\scrX}_V$ is the compactification obtained by taking the closure 
	of the image of the diagonal embedding of $V_g$ in the product $\clr{\scrX}_{M, \varrho} \times \clr{\scrX}_V$,
	\cite[3.45]{gjt}. As a consequence, the image of an unbounded sequence converges in the join if, and only if, it converges 
	in each of the two factors of the topological product. In view of Theorems \ref{thm:11}, \ref{thm:4} and \ref{thm:10}, we 
	conclude first that angular core $(\omega, J, c, u)$-sequences converge both in the Martin compactification 
	$\clr{\scrX}_{M,\zeta}$ for $\zeta > \varrho$ and in the join $\clr{\scrX}_{M, \varrho} \vee \clr{\scrX}_V$, and then that 
	the uniqueness statements in the latter results provide the identification (see \emph{e.g.}~the domination criterion given 
	by \cite[Lemma 3.28]{gjt} in both directions between two compactifications). This identification is equivariant because 
	$\Aut(\scrX)$ acts continuously on the compactifications and permutes angular core sequences. 
\end{proof}

\subsection{Martin compactifications for Bruhat--Tits buildings}
We finally consider the Bruhat--Tits context.
In other words, we present the non-Archimedean counterpart to the study done on Riemannian symmetric spaces by
Y.~Guivarc'h and collaborators, see \cite{gjt} and \cite{MR1832435}. In these references, the Archimedean case of
potential-theoretic compactifications is fully treated in the following sense:
\begin{enumerate}[label=(\roman*), ref=\roman*]
	\item \label{en:1:1}
	Martin compactifications of symmetric spaces are defined, both by means of differential operators ({\it i.e.} eigenfunctions 
	of Laplace operators) and via random walks \cite[Chapters VI--VIII and XIV]{gjt}; 
	\item \label{en:1:2}
	for a given symmetric space, the Martin compactification at the bottom of the spectrum is shown to be equivariantly
	homeomorphic to the maximal Satake (representation-theoretic), the maximal Furstenberg (measure-theoretic) or the Guivarc'h
	compactifications (group-theoretic) \cite[Theorems 2.13 and  3.20]{MR1832435}, see also \cite{Moore64};
	\item \label{en:1:3}
	Martin compactifications above the bottom of the spectrum are shown to be equivariantly homeomorphic to the join of
	the Gromov compactification with any compactification discussed before in \eqref{en:1:2} \cite[Theorems 8.2 and 8.21]{gjt};
	\item \label{en:1:4} 
	Martin compactifications at the bottom of the spectrum are used to parametrize geometrically two classes of remarkable
	subgroups, namely maximal distal and maximal amenable subgroups \cite[Theorem 2.14]{MR1832435}, see also \cite{Moore80};
	\item \label{en:1:5}
	an integral formula for eigenfunctions of the Laplace operator is given by means of suitable Poisson kernels
	\cite[Theorems 13.1 and 13.28]{gjt} and an analogous result is given from the viewpoint of random walks
	\cite[Theorem 13.33]{gjt}. 
\end{enumerate}
We consider now the Bruhat--Tits analogues of these results. These problems were mentioned, together with some hints,
in \cite[Chapter XV]{gjt} and \cite[\S 4]{MR1832435}. We wish to explain here where the intuitions there could be implemented
and where we took another path. 

Of course, the use of techniques from partial differential equations is not directly efficient when dealing with buildings
instead of Riemannian symmetric spaces. The viewpoint of random walks together with non-Archimedean harmonic analysis as
developed in \cite{macdo0} becomes the main tool. In the probabilistic part of their work, Guivarc'h--Ji--Taylor use the notion
of a \emph{well-behaved} measure on a symmetric space $X=G/K$, or more precisely on the connected semisimple Lie group $G$:
a positive measure on $G$ is called well-behaved if it has a continuous density (with respect to the Haar measure) and if its
support $S$, assumed to be compact, satisfies $G = \bigcup_{n \geqslant 0} S^n$. If we are given a bi-$K$-invariant well-behaved
probability measure $p$, then the convolution operator associated with $p$ provides a generalization of the Laplace operator
\cite[Proposition 1]{gu}. The associated random walk has finite range whenever $p$ has compact support and is irreducible
whenever the probability measure $p$ is well-behaved; moreover the trick in \ref{ss - Martin embeddings} modifies, if necessary,
the random walk attached to $p$ in such a way that it becomes aperiodic but still provides the same Martin boundary. To sum up,
a compactly supported bi-$K$-invariant well-behaved probability measure on $\mathbf{G}(k)$ defines a random walk for which
Theorems \ref{thm:4} and \ref{thm:10} provide explicit descriptions of Martin boundaries by means of core sequences (at the
bottom and above the bottom of the spectrum, respectively). This settles \eqref{en:1:1} above and allows us to identify
the Martin compactification at the bottom of the spectrum according to Theorem \ref{th - Comp Furstenberg} above, which settles
\eqref{en:1:2}. While \eqref{en:1:4} was established in \cite{gure}, we intend to go back to \eqref{en:1:5}, namely integral
representation of harmonic functions, in a subsequent work. Finally, for \eqref{en:1:3}, we have the following statement
which contains the second half of Theorem \ref{th - alg gr} of the introduction. 

\begin{theorem}
\label{thm:BrTMartinAbove}
	Let $\mathbf{G}$ be a semisimple simply connected algebraic group over $k$, a locally compact non-Archimedean valued field,
	and let $\scrX$ be its Bruhat--Tits building. We choose a good vertex in $\scrX$ and denote by $K$ its stabilizer.
	Let us pick a compactly supported bi-$K$-invariant well-behaved probability measure on $\mathbf{G}(k)$. 
	Then, for every $\zeta$ above the spectral radius $\rho$ of the measure, the corresponding Martin compactification
	$\clr{\scrX}_{M, \zeta}$ is the join of $\clr{\scrX}_{M, \varrho}$ with the Gromov compactification. 
\end{theorem}

\begin{proof}
	In view of Theorem \ref{thm:8}, it is enough to see that the operator associated to a bi-$K$-invariant well-behaved 
	probability measure on $\mathbf{G}(k)$ is an averaging operator as in Section \ref{ss - Martin embeddings}. 
	Let $\varphi(g) {\rm d}g$ be a compactly supported bi-$K$-invariant well-behaved probability measure on $\mathbf{G}(k)$. 
	Then $\varphi$ is a compactly supported bi-$K$-invariant continuous function on $G$. Let $o$ be the good vertex in $\scrX$
	whose stabilizer is the subgroup $K$ and let $V_o$ be the set of the (good) vertices of the same type as $o$. Each
	$v \in V_o$ can be written as $g . o$ for some $g \in \mathbf{G}(k)$ which well-defined up to right multiplication by
	elements in $K$. We wish to introduce the operator $A$ acting on suitable functions $f : V_o \to \mathbb{C}$ by the formula
	(see \cite[Remark 2 p.171]{gjt}) 
	\[
		Af(v) = Af(g . o) =  \int_G f(gh . o) \varphi(h) \, {\rm d}h.
	\]
	By left-invariance of the Haar measure, we can also write 
	\[
		Af(g . o) = \int_G f(h . o) \varphi(g^{-1}h) \, {\rm d}h,
	\]
	which, since $\varphi$ is left $K$-invariant, shows that the definition does not depend on the element
	$g \in \mathbf{G}(k)$ such that $v=g . o$. 

	We need to show that the operator $A$ is as in Section \ref{ss - Martin embeddings}. The function $\varphi$, being compactly
	supported and bi-$K$-invariant, is a (finite) linear combination of characteristic functions of double classes modulo $K$.
	We are thus reduced to the situation where $\varphi = \frac{1}{{\rm vol}(K \xi^{-1}(t) K)} \ind{K \xi^{-1}(t) K}$ for some
	$t \in Y^+$ in the notation of Section \ref{sec:4}, and where ${\rm vol}$ denotes the volume of with respect to
	the Haar measure. We have thus
	\[
		Af(g . o) = \frac{1}{{\rm vol}(K \xi^{-1}(t) K)} \int_{K \xi^{-1}(t) K} f(gh . o) {\: \rm d}h.
	\]
	The group $K$ acts transitively on the vertices at vectorial distance $t$ from $o$, so when $h$ runs over $K \xi^{-1}(t) K$,
	the vertices $(gh) . o$ describe the combinatorial sphere $V_t(v)$ centered at $v=g . o$ and of vectorial radius $t$. 
	Moreover, we have an identification of (finite) $K$-homogeneous spaces $K \xi^{-1}(t) K /K \simeq K / {\rm Stab}_K(t . o)$
	on which the invariant measure is the counting measure. Altogether, this provides the integration formula
	\begin{align*}
		\int_{K \xi^{-1}(t) K} f(gh . o) \, {\rm d}h 
		&= 
		\int_{{\rm Stab}_K(t . o)} 
		\bigg( \int_{K/{\rm Stab}_K(t . o)} 
		f(g\bar k . (t . o) ) \, {\rm d}\bar k \bigg) {\: \rm d}k \\
		&= \int_{{\rm Stab}_K(t . o)} \bigg( \sum_{v' \in V_t(v)} f(v') \bigg) {\: \rm d}k. 
	\end{align*}
	Taking $f$ to be constant equal to $1$ gives: ${\rm vol}(K \xi^{-1}(t) K) 
	= {\rm vol}({\rm Stab}_K(t . o)) \times |V_t(v)|$, thus going back to the previous expression for $Af(g . o)$,
	we obtain
	\[
		Af(g . o) = \frac{1}{|V_t(v)|} \sum_{v' \in V_t(v)} f(v'),
	\]
	which shows that $A$ is indeed an averaging operator, as desired. In general, $\varphi$ is a finite linear combination
	of characteristic functions of double classes modulo $K$, and therefore $A$ is a finite linear combination of averaging
	operators as in Section \ref{ss - Martin embeddings}. 
\end{proof} 

In the proofs here, we have used more general results from Sections \ref{ss - Martin bottom} and \ref{sec:7.3}.
Recall that we do not require group action in order to define and understand Martin compactifications of affine buildings.

We would like to conclude by discussing in slightly more details the differences with the hints and intuitions provided in 
\cite{gjt} and \cite{MR1832435}. More precisely, not only the use of group actions is crucial in the latter references, but even
in the Bruhat--Tits framework chosen here, some differences with the case of symmetric spaces should be mentioned. 

The main difference is the fact that for affine buildings, thanks to \cite{tr} which provides us exact asymptotics of the
Green's functions on affine buildings, we could perform exact computations of limits of Martin kernels. One consequence is that,
uniformly with respect to any chosen procedure of compactification, we can use the same parametrizing system for limits,
whatever the target space of the embedding map, so that finally we can identify or describe compactifications by arguments using
these parameters only (which are basically: radial directions and distances to sector panels of a given Weyl sector). 
This is what we do at the end of each section dealing with a given type of compactification, Theorem \ref{thm:7} at the end of 
Section \ref{ss - Martin bottom} (Martin compactification at the bottom of the spectrum), Theorem \ref{thm:8} at the end of
Section \ref{sec:7.3} (Martin compactification above the bottom of the spectrum) and Theorem \ref{thm:9} at the end of 
Appendix \ref{ap:1} (the case of non-reduced root systems). In the case of symmetric spaces, the available asymptotics due 
to Anker--Ji \cite{aj} (Green kernels) and Anker \cite{a1} (ground state spherical function) are good enough to describe 
the Martin compactifications in \cite{gjt} and \cite{MR1832435}, but the computation of limits is not direct. This is related 
to the well-known fact that Harish-Chandra's integral formula for spherical functions remains an integral formula in Archimedean 
case, while it can be made algebraic in the non-Archimedean case \cite[Chapter IV]{macdo0}: this is, so to speak, the analytic 
approach to the theory of Macdonald spherical functions \cite{MacdonaldSym}. Note also that the idea to develop an abstract 
harmonic analysis dealing with (Iwahori--)Hecke algebras, and avoiding automorphism groups as much as possible, goes back to 
H. Matsumoto \cite{Matsumoto77}. 

In order to make the comparison with Archimedean case in more details, let us separate two steps in the study of Martin
compactifications: the computation of limits of Martin kernels first, and then the description of the boundaries. Already at the
bottom of the spectrum, the computation of limits is not purely analytic \cite[Proposition 7.26]{gjt} since it is based on
an argument of uniqueness of cluster value (by compactness), which itself uses a characterization of limit functions by means
of conditions mixing harmonicity properties and knowledge of stabilizers \cite[Theorem 7.22]{gjt} (the convergence of measures
for the Furstenberg compactifications uses a similar uniqueness argument based on the knowledge of the support and of part of
the stabilizer of the limit measure). Still at the bottom of the spectrum, the description of the Martin compactification
\cite[Theorem 7.33]{gjt} uses in a crucial way the group action since the identification with other compactifications is based
on an explicit comparison of stabilizers and of complete sets of representatives. Above the bottom of the spectrum, the
computation of limits of Martin kernels \cite[Theorem 8.2]{gjt} is not direct since it uses the Anker--Ji and Anker asymptotics,
which are given up to multiplicative constants. One then knows that a cluster value of Martin kernels attached to
a core sequence is a multiple of the expected limit, but one still has to use representing measures and again knowledge of
stabilizers in order to conclude, see \cite[p. 124]{gjt}. The description of the Martin compactifications above the bottom of
the spectrum is also based on a precise understanding of stabilizers and complete sets of representatives (note that for
the latter point an argument due to Karpelevich is systematically used, see \cite[Proposition 7.20]{gjt} and
\cite{Karpelevic}). 

In the non-Archimedean case, thanks to stronger asymptotics obtained in \cite{tr} we make in Section \ref{sec:7} the computation
of limits of Martin kernels in a purely analytic way. The various factors in the resulting formulas (see Theorem \ref{thm:4}
for the bottom spectrum and Theorem \ref{thm:10} otherwise) can be understood geometrically in terms of fa\c cades at infinity.

These factors do not have the same asymptotic behaviors: in both cases, the factor with polynomial growth mimics the initial
situation in the sense that it corresponds to the ground state spherical function on the involved stratum at infinity
(an affine building of smaller rank). The remaining factors have exponential growth and precisely this analytic difference is
exploited to obtain the needed uniqueness results (see Section \ref{sec:15}, in particular the way equation
\eqref{eq:46} is exploited). This allows us to directly use the geometric parameters of core sequences to describe the Martin
compactifications and compare them with the previous ones. 

\appendix
\section{Distinguished random walk for $\widetilde{\mathrm{BC}}_r$}
\label{ap:1}
In Section \ref{sec:7} we study Martin compactifications of affine buildings for isotropic finite range random walks of good 
vertices. Unfortunately, in the non-reduced case we obtain two different boundaries corresponding to $V_g$ and $V_g^\varepsilon$.
In particular, there are no harmonic analytic tools which allow to study functions on the whole set $V_s$ at the same time.
This limitation is a consequence of a lack of Green's function asymptotics for more general random walks. In this appendix
we define a certain random walk on \emph{all} special vertices in the non-reduced case, for which we compute the limits of 
Martin kernels. This allows us to obtain Martin compactifications of all special vertices of an affine building 
(see Theorem \ref{thm:9}). 

Let us recall that for each $r \geqslant 1$ there is only one non-reduced finite root system, $\mathrm{BC}_r$, that is
\[
	\Phi = \big\{\pm e_i, \pm 2 e_i, \pm e_j \pm e_k : 1 \leqslant i \leqslant r, 1 \leqslant j < k \leqslant r
	\big\}
\]
where $\{e_1, \ldots, e_r\}$ is the standard basis of $\mathfrak{a}$. The standard base of $\Phi$ consists of the roots
\[
	\alpha_j = 
	\begin{cases}
		e_j - e_{j+1} & \text{if } 1 \leqslant j \leqslant r-1, \\
		e_r & \text{if } j = r.
	\end{cases}
\]
Thus the fundamental co-weights are
\[
	\lambda_j = e_1 + e_2 + \ldots + e_j.
\]
Special vertices have type $0$ or $r$, but only type $0$ is good.

Let $\scrX$ be an affine building of non-reduced type. Given a chamber $c \in \calC(\scrX)$, let us denote by $v_j(c)$
the vertex of $c$ having type $j$. For a vertex $v \in V(\scrX)$, let $\calC(v)$ be the set of all chambers sharing
the vertex $v$. For $x \in V_s$, we set
\[
	\calV_r(x) = \big\{v_r(c) : c \in \calC(x) \big\}, \qquad \text{if } \tau(x) = 0,
\]
and
\[
	\calV_0(x) = \big\{v_0(c) : c \in \calC(x) \big\}, \qquad \text{if } \tau(x) = r.
\]
Observe that
\[
	N_r = \#\calV_r(x) 
	= \frac{W(q)}{W_{\lambda_r}(q)}, \qquad\text{and}\qquad 
	N_0 = \#\calV_0(x) = \frac{W^a_{\lambda_r}(q)}{W_{\lambda_r}(q)}
\]
where $W^a_{\lambda} = \{w \in W^a : w . \lambda = \lambda\}$. Now, for each $c \in \calC(\scrX)$ we set
\[
	p\big(v_0(c), v_r(c)\big) = \frac{1}{N_r}, \qquad\text{and}\qquad
	p\big(v_r(c), v_0(c)\big) = \frac{1}{N_0}.
\]
Then $P = \big((p(x, y) : x, y \in V_s)$ generates a reversible Markov chain on special vertices of $\scrX$. Since the 
random walk has period $2$, it is natural to consider
\[
	\tilde{p}(x, y) = \sum_{z \in V_g^\varepsilon} p(x, z) p(z, y), \qquad x, y \in V_g.
\]
For each $y$ belonging to
\[
	\bigcup_{c \in \calC(x)} \big\{v_0(d) : d \in \calC(v_r(c))\big\}
\]
there is $z \in V^\varepsilon_g$ such that
\[
	\sigma(x, z) = \frac{1}{2} \lambda_r, \qquad\text{and}\qquad
	\sigma(z, y) = \frac{1}{2} \lambda_r.
\]
Moreover, there is $\omega \in \Omega$ such that $h(x, y; \omega) \in P^+$. Since
\[
	\sigma(x, y) = h(x, z; \omega) + h(z, y; \omega) \in P^+,
\]
we conclude that $\sigma(x, y) = 0$ or $\lambda_k$ for a certain $k \in I_0$. Hence,
\[
	\bigcup_{c \in \calC(x)} \big\{v_0(d) : d \in \calC(v_r(c))\big\}
	=
	\{x\} \sqcup \bigsqcup_{j \in I'} V_{\lambda_j}(x)
\]
for certain subset $I'$ of $I_0$. 

Now, let us fix $y \in V_{\lambda_j}(x)$. There are $W_{\lambda_j}(q)$ distinct chambers $c \in \calC(x)$ such that
there is $d \in \calC(v_r(c))$ with $v_0(d) = y$, but among them $W_{\lambda_j \lambda_r}(q)$ share the vertex $v_r(c)$
where $W_{\lambda_j \lambda_r} = W_{\lambda_j} \cap W_{\lambda_r}$. Therefore
\begin{align*}
	\sum_{y \in V_g} \tilde{p}(x, y) f(y) 
	&= \frac{1}{N_0 N_r} \bigg(
	N_r f(x) 
	+ 
	\sum_{j \in I'}
	\frac{W_{\lambda_j}(q)}{W_{\lambda_j \lambda_r}(q)}
	\sum_{y \in V_{\lambda_j}(x)} f(y) 
	\bigg).
\end{align*}
Hence, $\tilde{P} = (\tilde{p}(x, y) : x, y \in V_g)$ generates a reversible Markov chain on good
vertices of $\scrX$. The corresponding averaging operator belongs to the algebra $\scrA_0$. Moreover,
if $x, y \in V_g$, then
\begin{align*}
	G_\zeta(x, y) 
	&= \sum_{n = 0}^\infty \zeta^n p(n; x, y) \\
	&= \sum_{n = 0}^\infty \zeta^{2n} \tilde{p}(n; x, y) = \tilde{G}_\zeta(x, y).
\end{align*}
Hence, if $\tau(y) = 0$, then
\[
	K_\zeta(x, y) 
	= \begin{cases}
		\tilde{K}_\zeta(x, y) & \text{if } \tau(x) = 0, \\
		\frac{1}{N_0} \sum_{x' \in \calV_0(x)} \tilde{K}_\zeta(x', y) &\text{if } \tau(x) = r,
	\end{cases}
\]
where we have set
\[
	\tilde{K}_\zeta(x, y) = \frac{\tilde{G}_\zeta(x, y)}{\tilde{G}_\zeta(o, y)},
	\qquad
	x, y \in V_g.
\]
Analogously, we can introduce a random walk whose transition function is given by the formula
\[
	\tilde{p}^\varepsilon(x, y) = \sum_{z \in V_g} p(x, z) p(z, y), \qquad x, y \in V^\varepsilon_g.
\]
For a chosen $o^\varepsilon \in \calV_r(o)$, we set
\[
	\tilde{K}_{\zeta}^\varepsilon(x, y) 
	= \frac{\tilde{G}^\varepsilon_\zeta(x, y)}{\tilde{G}^\varepsilon_\zeta(o^\varepsilon, y)},
	\qquad x, y \in V_g^\varepsilon,
\]
and
\[
	K_\zeta^\varepsilon(x, y) = \frac{G_\zeta(x, y)}{G_\zeta(o^\varepsilon, y)},
	\qquad x, y \in V_s.
\]
Hence, for $x, y \in V_s$,
\begin{equation}
	\label{eq:67}
	K_\zeta(x, y) 
	= \frac{G_\zeta(x, y)}{G_\zeta(o, y)} 
	= \frac{K_\zeta^\varepsilon(x, y)}{K_\zeta^\varepsilon(o, y)}.
\end{equation}
Since for $x, y \in V_g^\varepsilon$,
\[
	G_\zeta(x, y) = \tilde{G}_\zeta^\varepsilon(x, y), 
\]
if $\tau(y) = r$ we have
\begin{align}
	\label{eq:65}
	K_\zeta^\varepsilon(x, y) 
	= 
	\begin{cases}
		\tilde{K}^\varepsilon_\zeta(x, y) & \text{if } \tau(x) = r, \\
		\frac{1}{N_r} \sum_{x' \in \calV_r(x)} \tilde{K}^\varepsilon_\zeta(x', y) & \text{if } \tau(x) = 0.
	\end{cases}
\end{align}
Let us denote by $\calB_\zeta(V_s)$ the set of positive $\zeta$-superharmonic functions on special vertices of $\scrX$, 
normalized to take value $1$ at the vertex $o$. The set $\calB_\zeta(V_s)$ endowed with the topology of pointwise
convergence is a compact second countable Hausdorff space. Let us define the map
\begin{alignat*}{1}
	\iota: V_s &\longrightarrow \calB_\zeta(V_s) \\
	y &\longmapsto K_\zeta(\,\cdot\,, y). 
\end{alignat*}
Since the random walk generated by $P$ is transient, the map $\iota$ gives an equivariant embedding of $V_s$ into
$\calB_\zeta(V_s)$ which has discrete image. Let $\clr{\scrX}_{M, \zeta}$ be the closure of $\iota(V_s)$ in
$\calB_{\zeta}(V_s)$. The space $\calB_\zeta(V_s)$ is metrizable, thus by Lemma \ref{lem:2}, while studying 
$\clr{\scrX}_{M, \zeta}$ we can restrict attention to core and angular core sequences. 

Now, to $\tilde{K}_\zeta$ and $\tilde{K}_\zeta^\varepsilon$ we can apply Theorems \ref{thm:4} and \ref{thm:10} proving
the existence of the limits $(K_\zeta(\:\cdot\:, y_n))$ for core and angular core sequences. It remains to describe
sequences having the same limit (this is the analogue of the uniqueness statements in Theorem \ref{thm:4} and Theorem 
\ref{thm:10}).

Suppose that there are two core sequence $(y_n)$ and $(y_n')$ with parameters $(\omega, J, c)$ and $(\omega', J', c')$,
respectively, that have the same limit. We claim that $J' = J$, $\omega' \sim_J \omega$, and $c' = c$. If
$\tau(y_n) = \tau(y_n') = 0$ for all $n$, it is a direct consequence of Theorem \ref{thm:6}. Suppose that
$\tau(y_n) = \tau(y_n') = r$ for all $n$. Then by \eqref{eq:67} for all $x \in V_g^\varepsilon$,
\[
	\Big( \lim_{n \to \infty} K_\zeta^\varepsilon(o, y_n') \Big) 
	\lim_{n \to \infty} K_\zeta^\varepsilon(x, y_n)
	=
	\Big( \lim_{n \to \infty} K_\zeta^\varepsilon(o, y_n) \Big)
	\lim_{n \to \infty} K_\zeta^\varepsilon(x, y_n'),
\]
and we can repeat the same reasoning as in Theorem \ref{thm:6} to obtain the desired conclusion.

Finally, let us consider $\tau(y_n) = 0$ and $\tau(y_n') = r$. Then for all $x \in V_g$, by \eqref{eq:65}, we get
\[
	\Big( \lim_{n \to \infty} K_\zeta^\varepsilon(o, y_n') \Big)
	\lim_{n \to \infty} \tilde{K}_\zeta(x, y_n)
	=
	\frac{1}{N_r} \sum_{x' \in \calV_r(x)} \lim_{n \to \infty} \tilde{K}^\varepsilon_{\zeta}(x', y_n').
\]
Let $y = y_m$ and $y' = y_m'$ for $m$ sufficiently large. Let $\scrA$ be an apartment containing $[o, \omega]$, and
$\tilde{\omega}$ be opposite to $\omega$ such that $[\tilde{\omega}, \omega] = \scrA$. 
Let us consider $x \in V_g \cap [y', \tilde{\omega}] \cap [o, \tilde{\omega}]$. By repeating the line of reasoning used
in the proof of Theorem \ref{thm:6}, we show that $J' = J$, and 
\begin{align*}
	\frac{\calP_J(\sigma(x, y) + n \lambda)}{\calP_J(\sigma(x, y))}
	=
	\frac{\sum_{x' \in \calV_r(x)} \calP_J^\varepsilon(\sigma(x', y') + n\lambda) 
	\chi^{\frac{1}{2}}(h(x, x'; \omega'))}
	{\sum_{x' \in \calV_r(x)} \calP_J^\varepsilon(\sigma(x', y') ) 
	\chi^{\frac{1}{2}}(h(x, x'; \omega'))}
\end{align*}
for all $\lambda \in P^+$, and $n \in \NN_0$. By comparing the coefficient with $n^{|\Phi^{++}_J|}$ we get
\[
	\frac{\bfb_J(0)}{\calP_J(\sigma(x, y))} 
	=
	\frac{
	\bfb_J^\varepsilon(0)
	\sum_{x' \in \calV_r(x)} \chi^{\frac{1}{2}}(h(x, x'; \omega'))}
	{\sum_{x' \in \calV_r(x)} \calP_J^\varepsilon(\sigma(x', y'))
	\chi^{\frac{1}{2}}(h(x, x'; \omega'))}.
\]
Since $r \in J$, by comparing the coefficients with $n^{|\Phi^{++}_J|-1}$ we arrive at
\begin{equation}
	\label{eq:70}
	\begin{aligned}
	\sprod{\sigma(x, y) }{\alpha_r} + \sprod{\nabla}{\alpha_r} \log \bfb_J(0)
	&=
	\sprod{\sigma(x, y')}{\alpha_r} + \sprod{\nabla}{\alpha_r} \log \bfb_J^\varepsilon(0) \\
	&\phantom{=}-
	\frac{\sum_{x' \in \calV_r(x)} \sprod{h(x, x'; \omega')}{\alpha_r} \chi^{\frac{1}{2}}(h(x, x'; \omega'))}
	{\sum_{x' \in \calV_r(x)} \chi^{\frac{1}{2}}(h(x, x'; \omega'))}.
	\end{aligned}
\end{equation}
We claim that the following holds true.
\begin{claim}
	\label{clm:5}
	For all $x \in V_g$ and $\omega \in \Omega$,
	\begin{equation}
		\label{eq:71}
		\frac{\sum_{x' \in \calV_r(x)} \sprod{h(x, x'; \omega)}{\alpha_r} \chi^{\frac{1}{2}}(h(x, x'; \omega))}
		{\sum_{x' \in \calV_r(x)} \chi^{\frac{1}{2}}(h(x, x'; \omega))}
		=
		\frac{1}{2} \cdot \frac{\sqrt{q_{\alpha_0}} - \sqrt{q_{\alpha_r}}}{\sqrt{q_{\alpha_0}} + \sqrt{q_{\alpha_r}}}.
	\end{equation}
\end{claim}
\noindent
For the proof, we recall that $\alpha_r = e_r$, thus
\[
	\sum_{x' \in \calV_r(x)} \sprod{h(x, x'; \omega)}{\alpha_r} \chi^{\frac{1}{2}}(h(x, x'; \omega))
	=
	\tfrac{1}{2} 
	\sum_{\eta \in \{-1, 1\}^r} \eta_r \chi^{\frac{1}{4}}\left(\sum_{j = 1}^r \eta_j e_j\right) N_\eta
\]
where for $\eta \in \{-1, 1\}^r$, we have set
\[
	N_\eta = \#\left\{x' \in \calV_r(x) : h(x, x'; \omega )
	= \tfrac{1}{2} \sum_{j = 1}^r \eta_j e_j \right\}.
\]
Let us observe that for $\eta' \in \{-1, 1\}^{r-1}$,
\[
	N_{(\eta', -1)} = q_{\alpha_r} N_{(\eta', 1)},
\]
and by \cite[Proposition 5.3]{mz1}
\[
	\chi(e_r) = q_{\alpha_0} q_{\alpha_r}.
\]
Hence,
\[
	\chi^{\frac{1}{4}}\left(\sum_{j = 1}^{r-1} \eta_j e_j + e_r\right) N_{(\eta', 1)}
	-
	\chi^{\frac{1}{4}}\left(\sum_{j = 1}^{r-1} \eta_j e_j - e_r\right) N_{(\eta', -1)}
	=
	\bigg(1 - \sqrt{\frac{q_{\alpha_r}}{q_{\alpha_0}}}\bigg)
	\chi^{\frac{1}{4}}\left(\sum_{j = 1}^{r-1} \eta_j e_j + e_r\right) N_{(\eta', 1)},
\]
which leads to
\begin{align*}
	\sum_{x' \in \calV_r(x)} \sprod{h(x, x'; \omega)}{\alpha_r} \chi^{\frac{1}{2}}(h(x, x'; \omega))
	&=
	\frac{1}{2}
	\bigg(1 - \sqrt{\frac{q_{\alpha_r}}{q_{\alpha_0}}}\bigg)
	\sum_{\eta' \in \{-1, 1\}^{r-1}}
	\chi^{\frac{1}{4}}\left(\sum_{j = 1}^{r-1} \eta_j' e_j + e_r\right) N_{(\eta', 1)}
\end{align*}
and
\begin{align*}
	\sum_{x' \in \calV_r(x)} \chi^{\frac{1}{2}}(h(x, x'; \omega))
	=
	\bigg(1 + \sqrt{\frac{q_{\alpha_r}}{q_{\alpha_0}}}\bigg)
	\sum_{\eta' \in \{-1, 1\}^{r-1}}
	\chi^{\frac{1}{4}}\left(\sum_{j = 1}^{r-1} \eta_j' e_j + e_r\right) N_{(\eta', 1)},
\end{align*}
and the claim follows.

Next, by \cite[Section 5.2]{park2}, for $\theta = \sum_{j = 1}^r \theta_j e_j$, we have
\[
	\log \bfb_J^\varepsilon(\theta) - \log \bfb_J(\theta)
	=
	\sum_{j = 1}^r \log \bigg(1 + \sqrt{\frac{q_{\alpha_r}}{q_{\alpha_0}}} e^{-\theta_j}\bigg)
	-
	\log \bigg(1 + \sqrt{\frac{q_{\alpha_0}}{q_{\alpha_r}}} e^{-\theta_j}\bigg).
\]
Hence,
\[
	\sprod{\nabla}{\alpha_r} \log \bfb_J^\varepsilon(0) - \sprod{\nabla}{\alpha_r} \log \bfb_J(0)
	=
	\frac{\sqrt{q_{\alpha_0}} - \sqrt{q_{\alpha_r}}}{\sqrt{q_{\alpha_0}} + \sqrt{q_{\alpha_r}}}.
\]
Therefore, by \eqref{eq:71}, the formula \eqref{eq:70} takes the form
\[
	\sprod{\sigma(x, y)}{\alpha_r} - \sprod{\sigma(x, y')}{\alpha_r} =
	\frac{1}{2} \cdot \frac{\sqrt{q_{\alpha_0}} - \sqrt{q_{\alpha_r}}}{\sqrt{q_{\alpha_0}} + \sqrt{q_{\alpha_r}}}.
\]
Since the left-hand side belongs to $\frac{1}{2} \ZZ$, we must have
$q_{\alpha_0} = q_{\alpha_r}$ which leads to contradiction. Analogously, we treat angular core sequences.

\begin{theorem}
	\label{thm:9}
	Assume that the finite root system associated with the building $\scrX$ has type $\mathrm{BC}_r$. 
	Then for the random walk generated by $P$ as constructed above, the following dichotomy holds:
	\begin{itemize}
		\item[$\bullet$]~\emph{[At the bottom of the spectrum case]}~If $\zeta=\varrho$, then $\clr{\scrX}_{M, \varrho}$
		is $\Aut(\scrX)$-equivariantly isomorphic to any of the Furstenberg (measure-theoretic) 
		or the Caprace--L\'ecureux (combinatorial) compactifications of the set $V_s$ of special vertices.
		\item[$\bullet$]~\emph{[Above the bottom of the spectrum case]}~If $\zeta > \varrho$, then
		$\clr{\scrX}_{M, \zeta}$ is $\Aut(\scrX)$-equivariantly isomorphic to the join of any of the previous
		compactifications with the Gromov (horofunction) compactification of the set $V_s$ of special vertices.
	\end{itemize}
	Moreover, if there exists a locally compact closed subgroup of
	$\Aut(\scrX)$, say $G$, acting strongly transitively and type-preserving on the building $\scrX$, then 
	the Guivarc'h compactification $\clr{\scrX}_G$ is $G$-isomorphic to the Martin compactification $\clr{\scrX}_{M, \varrho}$.
\end{theorem}

\begin{proof} 
	In view of Theorem \ref{thm:5} the Martin compactification $\clr{\scrX}_{M, \varrho}$ for the random walk generated by $P$ is 
	$\Aut(\scrX)$-isomorphic to the Furstenberg compactification $\clr{\scrX}_F$. Moreover, if there is a locally compact
	group acting strongly transitively on the building $\scrX$ then $\clr{\scrX}_{M, \varrho}$ is $G$-isomorphic to
	the Guivarc'h compactification. Lastly, by Theorems \ref{thm:11} and \ref{thm:5} we conclude that
	$\clr{\scrX}_{M, \zeta}$, $\zeta > \varrho$ is $\Aut(\scrX)$-isomorphic to $\clr{\scrX}_F \vee \clr{\scrX}_V$ and
	$\clr{\scrX}_{M, \varrho} \vee \clr{\scrX}_V$.
\end{proof}

\begin{bibliography}{buildings}
	\bibliographystyle{amsplain}

\newcommand{\SL}{{\operatorname{SL}}} \newcommand{\SU}{{\operatorname{SU}}}
  \newcommand{\UU}{{\operatorname{U}}} \newcommand{\Sp}{{\operatorname{Sp}}}
  \newcommand{\SO}{{\operatorname{SO}}}
\providecommand{\bysame}{\leavevmode\hbox to3em{\hrulefill}\thinspace}
\providecommand{\MR}{\relax\ifhmode\unskip\space\fi MR }
\providecommand{\MRhref}[2]{%
  \href{http://www.ams.org/mathscinet-getitem?mr=#1}{#2}
}
\providecommand{\href}[2]{#2}
\begin{thebibliography}{10}

\bibitem{Abramenko2008}
P.~Abramenko and K.S. Brown, \emph{Buildings}, Graduate Text in Mathematics,
  Springer-Verlag New York, 2008.

\bibitem{a1}
{J.-Ph.} Anker, \emph{{L}a forme exacte de l'estimation fondamentale de
  {H}arish-{C}handra}, C. R. Acad. Sci. Paris S\'erie I Math. \textbf{305}
  (1987), 371--374.

\bibitem{aj}
{J.-Ph.} Anker and L.~Ji, \emph{{H}eat kernel and {G}reen function estimates on
  noncompact symmetric spaces}, Geom. Funct. Anal. \textbf{9} (1999),
  1035--1091.

\bibitem{BadCapLec}
U.~Bader, P.-E. Caprace, and J.~L\'{e}cureux, \emph{On the linearity of
  lattices in affine buildings and ergodicity of the singular {C}artan flow},
  J. Amer. Math. Soc. \textbf{32} (2019), no.~2, 491--562.

\bibitem{BaderShalom}
U.~Bader and Y.~Shalom, \emph{Factor and normal subgroup theorems for lattices
  in products of groups}, Invent. Math. \textbf{163} (2006), 415--454.

\bibitem{Ballman1985}
W.~Ballmann, M.~Gromov, and V.~Schroeder, \emph{Manifolds of nonpositive
  curvature}, Birk\"auser Borston Inc., 1985.

\bibitem{Berkovich}
V.G. Berkovich, \emph{Spectral theory and analytic geometry over
  non-{A}rchimedean fields}, Mathematical Surveys and Monographs, vol.~33,
  American Mathematical Society, Providence, RI, 1990.

\bibitem{BorelBook}
A.~Borel, \emph{Linear algebraic groups}, second ed., Graduate Texts in
  Mathematics, vol. 126, Springer-Verlag, New York, 1991.

\bibitem{BorelTits65}
A.~Borel and J.~Tits, \emph{Groupes r\'{e}ductifs}, Inst. Hautes \'{E}tudes
  Sci. Publ. Math. (1965), no.~27, 55--150.

\bibitem{Bourbaki2002}
N.~Bourbaki, \emph{Lie groups and {L}ie algebras. {C}hapters {4--6}}, Elements
  of Mathematics, Springer-Verlag Berlin Heidelberg, 2002.

\bibitem{Bourbaki2004}
\bysame, \emph{Integration {II}. {C}hapters {7--9}}, Elements of Mathematics,
  Springer-Verlag Berlin Heidelberg, 2004.

\bibitem{Bridson1999}
M.R. Bridson and A.~Haefliger, \emph{Metric spaces of non-positive curvature},
  Grundlehren der mathematischen Wissenschaften, Springer-Verlag Berlin
  Heidelberg, 1999.

\bibitem{BruhatTits1972}
F.~Bruhat and J.~Tits, \emph{Groupes r\'eductifs sur un corps local, {I}},
  Publ. Math. IH\'ES \textbf{41} (1972), 5--251.

\bibitem{BruhatTits1984}
\bysame, \emph{Groupes r\'{e}ductifs sur un corps local, {II}}, Publ. Math.
  IH\'ES \textbf{60} (1984), 197--376.

\bibitem{Caprace2011}
P.-E. Caprace and J.~L{\'e}cureux, \emph{Combinatorial and group-theoretic
  compactification of buildings}, Ann. I. Fourier \textbf{61} (2011), no.~2,
  619--672.

\bibitem{Chanfi}
D.~Chanfi, \emph{Wonderful compactifications of {B}ruhat-{T}its buildings in
  the non-split case}, Israel J. Math. \textbf{260} (2024), 105--140.

\bibitem{Charignon2009}
C.~Charignon, \emph{Compactifications polygonales d'un immeuble affine}, arXiv:
  0903.0502, 2009.

\bibitem{Ciobotaru2020}
C.~Ciobotaru, B.~M\"{u}hlherr, and G.~Rousseau, \emph{The cone topology on
  masures}, Advances in Geometry \textbf{20} (2020), 1--28.

\bibitem{DoobDiscretePot}
J.L. Doob, \emph{Discrete potential theory and boundaries}, J. Math. Mech.
  \textbf{8} (1959), 433--458; erratum 993.

\bibitem{Furstenberg1963}
H.~Furstenberg, \emph{A {P}oisson formula for semi-simple {L}ie groups}, Ann.
  Math. \textbf{77} (1963), 335--386.

\bibitem{gu}
Y.~Guivarc'h, \emph{Compactifications of symmetric spaces and positive
  eigenfunctions of the {L}aplacian}, S\'em. Math. Univ. Rennes, 1995.

\bibitem{MR1832435}
\bysame, \emph{Compactifications of symmetric spaces and positive
  eigenfunctions of the {L}aplacian}, Topics in probability and {L}ie groups:
  boundary theory, CRM Proc. Lecture Notes, vol.~28, Amer. Math. Soc.,
  Providence, RI, 2001, In collaboration with J. C. Taylor and Lizhen Ji,
  pp.~69--116.

\bibitem{gjt}
Y.~Guivarc'h, L.~Ji, and J.C. Taylor, \emph{Compactifications of symmetric
  spaces}, Progress Math., vol. 156, Birkh\"auser, 1998.

\bibitem{gure}
Y.~Guivarc'h and B.~R\'emy, \emph{{G}roup-theoretic compactification of
  {B}ruhat--{T}its buildings}, Ann. Sci. \'Ecole Norm. Sup. \textbf{39} (2006),
  no.~6, 871--920.

\bibitem{Helgason1979}
S.~Helgason, \emph{Differential geometry, {L}ie groups, and symmetric spaces},
  Academic Press, 1978.

\bibitem{HHL}
C.~Horbez, J.~Huang, and J.~L\'ecureux, \emph{Proper proximality in
  non-positive curvature}, Amer. J. Math. \textbf{145} (2023), 1327--1364.

\bibitem{Humphreys1990}
J.E. Humphreys, \emph{Reflection groups and {C}oxeter groups}, Cambridge
  University Press, 1990.

\bibitem{Karpelevic}
F.I. Karpelevich, \emph{The geometry of geodesics and the eigenfunctions of the
  {B}eltrami-{L}aplace operator on symmetric spaces}, Trans. Moscow Math. Soc.
  \textbf{1965} (1967), 51--199. Amer. Math. Soc., Providence, R.I., 1967.

\bibitem{Kostant}
B.~Kostant, \emph{On convexity, the {Weyl} group and the {Iwasawa}
  decomposition}, Ann. Sci. \'Ec. Norm. Sup\'er. \textbf{Ser. 4, 6} (1973),
  413--455.

\bibitem{Landvogt}
E.~Landvogt, \emph{A compactification of the {B}ruhat-{T}its building}, Lecture
  Notes in Mathematics, vol. 1619, Springer-Verlag, Berlin, 1996.

\bibitem{LecSalWit}
J.~L\'ecureux, M.~de~la Salle, and S.~Witzel, \emph{Strong property {(T)}, weak
  amenability and {$\ell^p$}-cohomology in {$\widetilde{A}_2$}-buildings}, to
  appear in Ann. Sci. \'Ecole Norm. Sup., 2024.

\bibitem{macdo0}
I.G. Macdonald, \emph{{S}pherical functions on a group of {$p$}-adic type},
  Publications of the Ramanujan Institute, vol.~2, Ramanujan Institute, Centre
  for Advanced Study in Mathematics, University of Madras, Madras, 1971.

\bibitem{MacdonaldSym}
\bysame, \emph{Symmetric functions and {H}all polynomials}, second ed., Oxford
  Mathematical Monographs, Oxford University Press, 1995.

\bibitem{mz1}
A.M. Mantero and A.~Zappa, \emph{Eigenvalues of the vertex set {H}ecke algebra
  of an affine building}, Trends in Harmonic Analysis (M.A. Picardello, ed.),
  INdAM, vol.~3, Springer, Milano, 2011, pp.~291--370.

\bibitem{mz2}
\bysame, \emph{Macdonald formula for spherical functions on affine buildings},
  Ann. Fac. Sci. Toulouse Math. (6) \textbf{20} (2011), no.~4, 669--758.

\bibitem{Margulis}
G.A. Margulis, \emph{Discrete subgroups of semisimple {L}ie groups}, Ergebnisse
  der Mathematik und ihrer Grenzgebiete (3), vol.~17, Springer-Verlag, Berlin,
  1991.

\bibitem{Matsumoto77}
H.~Matsumoto, \emph{Analyse harmonique dans les syst\`emes de {T}its
  bornologiques de type affine}, Lecture Notes in Mathematics, Vol. 590,
  Springer-Verlag, Berlin-New York, 1977.

\bibitem{Maubon}
J.~Maubon, \emph{Symmetric spaces of the non-compact type: differential
  geometry}, G\'{e}om\'{e}tries \`a courbure n\'{e}gative ou nulle, groupes
  discrets et rigidit\'{e}s, S\'{e}min. Congr., vol.~18, Soc. Math. France,
  Paris, 2009, pp.~1--38.

\bibitem{Moore64}
C.C. Moore, \emph{Compactifications of symmetric spaces}, Amer. J. Math.
  \textbf{86} (1964), 201--218.

\bibitem{Moore80}
\bysame, \emph{Amenable subgroups of semisimple groups and proximal flows},
  Israel J. Math. \textbf{34} (1979), 121--138.

\bibitem{Munkres1996}
J.R. Munkres, \emph{Elements of algebraic topology}, Westview Press, 1996.

\bibitem{park}
J.~Parkinson, \emph{{B}uildings and {H}ecke algebras}, Ph.D. thesis, University
  of Sydney, 2005.

\bibitem{park3}
\bysame, \emph{Buildings and {H}ecke algebras}, J. Algebra \textbf{297} (2006),
  no.~1, 1--49.

\bibitem{park2}
\bysame, \emph{{S}pherical harmonic analysis on affine buildings}, Math. Z
  \textbf{253} (2006), 571--606.

\bibitem{RTW10}
B.~R\'emy, A.~Thuillier, and A.~Werner, \emph{{B}ruhat-{T}its theory from
  {B}erkovich's point of view. {I}. realizations and compactifications of
  buildings}, Ann. Sci. \'Ec. Norm. Sup\'er \textbf{43} (2010), no.~3,
  461--554.

\bibitem{RTW12}
\bysame, \emph{{B}ruhat-{T}its theory from {B}erkovich's point of view. {II}
  {S}atake compactifications of buildings}, J. Inst. Math. Jussieu \textbf{11}
  (2012), no.~2, 421--465.

\bibitem{RTW17}
\bysame, \emph{Wonderful compactifications of {B}ruhat-{T}its buildings},
  \'{E}pijournal G\'{e}om. Alg\'{e}brique \textbf{1} (2017), Art. 10, 18.

\bibitem{RonanExotic}
M.A. Ronan, \emph{A construction of buildings with no rank {$3$} residues of
  spherical type}, Buildings and the geometry of diagrams ({C}omo, 1984),
  Lecture Notes in Math., vol. 1181, Springer, Berlin, 1986, pp.~242--248.

\bibitem{ron}
\bysame, \emph{{L}ectures on {B}uildings}, Perspectives in Mathematics,
  Academic Press, 1989.

\bibitem{Rousseau2023}
G.~Rousseau, \emph{Euclidean buildings}, EMS tracts in mathematics, vol.~35,
  European Math. Soc., 2023.

\bibitem{Satake1963}
I.~Satake, \emph{Theory of spherical functions on reductive algebraic groups
  over {$p$}-adic fields}, Publ. Math. IH\'ES \textbf{18} (1963), 1--69.

\bibitem{Serfozo1982}
R.~Serfozo, \emph{Convergence of {L}ebesgue integrals with varying measures},
  Sankhya Ser. A. (1982), 380--402.

\bibitem{tits}
J.~Tits, \emph{Buildings of spherical type and finite {BN}-pairs}, Lectures
  Notes in Mathematics, vol. 386, Springer, Berlin, Heidelberg, 1974.

\bibitem{Tits1977}
\bysame, \emph{Reductive groups over local fields}, Automorphic Forms,
  Representations and {$L$}-Functions, {P}art 1, Proceedings of Symposia in
  Pure Mathematics, vol.~33, AMS, 1979.

\bibitem{Tits1986}
\bysame, \emph{Immeubles de type affine}, Buildings and the geometry of
  diagrams ({C}omo, 1984), Lecture Notes in Math., vol. 1181, Springer, Berlin,
  1986.

\bibitem{tr}
B.~Trojan, \emph{Asymptotic behavior of heat kernels and {G}reen functions on
  affine buildings}, to appear in Journal of the European Mathematical Society,
  2024.

\bibitem{Werner07}
A.~Werner, \emph{Compactifications of {B}ruhat--{T}its buildings associated to
  linear representations}, Proc. Lond. Math. Soc. (3) \textbf{95} (2007),
  497--518.

\bibitem{Woess}
W.~Woess, \emph{Random walks on infinite graphs and groups}, Cambridge Tracts
  in Mathematics, vol. 138, Cambridge University Press, Cambridge, 2000.

\end{thebibliography}
\end{bibliography}

\end{document}